\documentclass[10pt]{article}
\usepackage{verbatim}
\usepackage[multiple]{footmisc}
\usepackage{amsfonts,color}
\usepackage{longtable}
\usepackage{makecell, multirow, tabularx}
\usepackage{amsmath,amssymb}
\usepackage{adjustbox}
\usepackage{epsfig}
\usepackage{lscape}
\usepackage{amssymb}
\usepackage{graphicx}
\usepackage[utf8]{inputenc}	
\usepackage{amsthm,amssymb}
\usepackage{aligned-overset}
\usepackage{amsfonts}
\usepackage[english]{babel}
\usepackage{dsfont}
\usepackage{bbm}
\usepackage[T1]{fontenc}
\usepackage{colonequals}
\usepackage{geometry}
\usepackage{mathtools}
\usepackage{csquotes}
\usepackage{url}
\usepackage{lipsum}
\usepackage{array}
\usepackage{tablefootnote}
\usepackage{wrapfig}
\usepackage[utf8]{inputenc}
\usepackage{cancel}
\usepackage[dvipsnames]{xcolor}
\usepackage{accents}
\usepackage{marvosym}
\usepackage{centernot}
\usepackage{tikz}
\usepackage{nicematrix}
\usepackage[new]{old-arrows}
\usepackage{titling}
\usepackage{multicol}

\usepackage{amsthm}
\newtheoremstyle{mystyle}
{}
{}
{}
{}
{\bfseries}
{.}
{ }
{}
\theoremstyle{plain}
\theoremstyle{mystyle}
\newtheorem{thm}{Theorem}[section]
\theoremstyle{plain}
\theoremstyle{mystyle}
\newtheorem{prop}[thm]{Proposition}
\theoremstyle{plain}
\theoremstyle{mystyle}
\newtheorem{cor}[thm]{Corollary}
\theoremstyle{plain}
\theoremstyle{mystyle}
\newtheorem{lem}[thm]{Lemma}
\theoremstyle{plain}
\theoremstyle{mystyle}
\newtheorem{conjecture}[thm]{Conjecture}
\theoremstyle{plain}
\theoremstyle{mystyle}
\newtheorem{example}[thm]{Example}
\theoremstyle{plain}
\theoremstyle{mystyle}
\newtheorem{rem}[thm]{Remark}

\theoremstyle{plain}
\theoremstyle{mystyle}
\newtheorem{definition}[thm]{Definition}

\newcommand{\red}[1]{{\color{red}#1}}

\newcommand{\violet}[1]{{\color{violet}#1}}
\newcommand{\blue}[1]{{\color{blue}#1}}

\newcommand{\teal}[1]{{\color{teal}#1}}
\newcommand{\id}{{\boldsymbol{\mathbbm{1}}}}

\allowdisplaybreaks[1]
\makeindex
\newcommand{\mcZ}{\mathcal{Z}}
\newcommand{\R}{\mathbb{R}}

\newcommand{\norm}[1]{\lVert #1 \rVert}

\newcommand{\SO}{{\rm SO}}
\newcommand{\OO}{{\rm O}}
\newcommand{\DD}{\mathrm{D}}
\newcommand{\WW}{\mathrm{W}}

\newcommand*\dif{\mathop{}\!\mathrm{d}}

\newlength{\dhatheight}

\DeclareMathOperator{\diag}{diag}
\DeclareMathOperator{\Sym}{Sym}
\DeclareMathOperator{\Cof}{Cof}
\DeclareMathOperator{\dev}{dev}

\DeclareMathOperator{\sym}{sym}
\DeclareMathOperator{\bfsym}{\textbf{sym}}
\DeclareMathOperator{\tr}{tr}
\DeclareMathOperator{\iso}{iso}
\DeclareMathOperator{\Aif}{Aif}
\DeclareMathOperator{\Old}{Old}
\DeclareMathOperator{\CR}{CR}
\DeclareMathOperator{\ZJ}{ZJ}
\DeclareMathOperator{\GN}{GN}
\DeclareMathOperator{\GS}{GS}

\DeclareMathOperator{\TR}{TR}

\def\barr{\begin{array}}
\def\tr{\textnormal{tr}}

\def\sk{\textnormal{skew}}

\def\dd{\displaystyle}

\def\barr{\begin{array}}
\def\earr{\end{array}}
\def\becn{\begin{equation*}}
\def\endec{\end{equation}}
\def\endecn{\end{equation*}}
\def\C{\mathbb{C}}
\def\H{\mathbb{H}}

\makeatletter
\let\@fnsymbol\@arabic
\makeatother
\numberwithin{equation}{section}

\title{Hypo-elasticity and logarithmic strain}
\begin{document}
\title{A natural requirement for objective corotational rates - on structure preserving corotational rates}
\author{
Patrizio Neff\thanks{
Patrizio Neff, University of Duisburg-Essen, Head of Chair for Nonlinear Analysis and Modelling, Faculty of Mathematics, Thea-Leymann-Stra{\ss}e 9,
D-45127 Essen, Germany, email: patrizio.neff@uni-due.de
}, \qquad
Sebastian Holthausen\thanks{
Sebastian Holthausen, University of Duisburg-Essen, Chair for Nonlinear Analysis and Modelling,  Faculty of Mathematics, Thea-Leymann-Stra{\ss}e 9,
D-45127 Essen, Germany, email: sebastian.holthausen@uni-due.de
}, \qquad
Sergey N.~Korobeynikov\thanks{
Sergey N.~Korobeynikov, Principal Investigator of Composite Mechanics Laboratory of Lavrentyev Institute of Hydrodynamics, Lavrentyev Prospekt 15, Novosibirsk, 630090, Russia, email: S.N.Korobeynikov@mail.ru
}, \\[0.8em]
Ionel-Dumitrel Ghiba\thanks{
Ionel-Dumitrel Ghiba, Alexandru Ioan Cuza University of Ia\c si, Department of Mathematics,  Blvd. Carol I, no. 11, 700506 Ia\c si, Romania;  Octav Mayer Institute of Mathematics
of the Romanian Academy, Ia\c si Branch, 700505 Ia\c si, email: dumitrel.ghiba@uaic.ro
}, \quad and \quad
Robert J. Martin\thanks{
Robert J. Martin,  Lehrstuhl für Nichtlineare Analysis und Modellierung, Fakultät für Mathematik, Universität Duisburg-Essen, Thea-Leymann Str. 9, 45127 Essen, Germany, email: robert.martin@uni-due.de
}
}
\maketitle
\vspace{-0,6cm}
\begin{abstract}
\noindent We investigate objective corotational rates satisfying an additional, physically plausible assumption. More precisely, we require for
\begin{equation*}
	\frac{\DD^{\circ}}{\DD t}[B] = \mathbb{A}^{\circ}(B).D
\end{equation*}
that $\mathbb{A}^{\circ}(B)$ is positive definite. Here, $B = F \, F^T$ is the left Cauchy-Green tensor, $\frac{\DD^{\circ}}{\DD t}$ is a specific objective corotational rate, $D = \sym \, \DD v$ is the Eulerian stretching and $\mathbb{A}^{\circ}(B)$ is the corresponding induced fourth order tangent stiffness tensor. Well known corotational rates like the Zaremba-Jaumann rate, the Green-Naghdi rate and the logarithmic rate belong to this family of ``positive'' corotational rates.

For general objective corotational rates $\frac{\DD^{\circ}}{\DD t}$ we determine several conditions characterizing positivity. Among them an explicit condition on the material spin-functions of Xiao, Bruhns and Meyers \cite{xiao98_1}. We also give a geometrical motivation for invertibility and positivity and highlight the structure preserving properties of corotational rates that distinguish them from more general objective stress rates. Applications of this novel concept are indicated. 
\\
\textbf{Mathscinet} classification
15A24, 73G05, 73G99, 74B20
\end{abstract}
\clearpage
\tableofcontents
\section{Introduction}
Corotational rates of tensors (see, for example, \cite{hashiguchi2009elastoplasticity, korobeynikovbook2000, meyers2000some, pozdeev1986, prager1961}) are essential for the proper formulation of rate-form equations in continuum mechanics. Constitutive relations that describe the behavior of engineering materials during finite elastoplastic deformations are typically expressed in terms of the rates of stresses and strains. Currently, the field includes a significant variety of corotational tensor rates (see, e.g.~\cite{asghari2008, dienes1979, dienes1987discussion, ghavam2007, gurtin1983relationship, hashiguchi2009elastoplasticity, hill1978, holzapfel2000, korobeynikovbook2000,  korobeynikov2008objective, lehmann1991, macmillan1992spin, mehrabadi1987, metzger1986objective, pozdeev1986, prager1961, reinhardt1995eulerian, reinhardt1996application, reinhardt1996, simo2006computational, sowerby1984rotations, szabo1989comparison, Truesdell65, xiao97, xiao1998direct, xiao98_2, xiao1998objective, xiao98_1, xiao1999natural, xiao2006elastoplasticity}), which have been developed over more than a century, beginning with the Zaremba–Jaumann rate (1903) (cf.~\cite{dienes1979}). Given the multitude of corotational rates, it is important to categorize these rates, along with their corresponding spin tensors, into families based on shared characteristics. Some methods for generating families of corotational tensor rates in continuum mechanics equations are discussed in \cite{xiao98_2, xiao1998objective, xiao98_1}. These methods derive families of material spin tensors through general expressions involving skew-symmetric tensor functions of basic kinematic tensors. Various techniques for generating spin tensor families associated with corotational tensor rates are explored in \cite{asghari2008, ghavam2007}.

In this contribution we will consider the advantages of corotational rates (qualifying them as the superior choice, cf.~\cite{bruhns2004}) compared to general objective stress rates or Lie-derivatives and we will define further subclasses of corotational rates so that certain structural properties are inherently conserved. In this respect, we discuss the invertibility and positivity of corotational rates. For both we provide necessary and sufficient conditions and we show that most classical corotational rates are in fact positive. This property will prove instrumental when generalizing the results from \cite{tobedone, CSP2024} to a larger class of corotational rates. Note that our aim is not to find a unique corotational rate with exceptional properties, as e.g.~the logarithmic rate is the unique corotational rate satisfying $\frac{\DD^{\log}}{\DD t}[\log V] = D$, but to make statements for whole classes of \emph{structure preserving corotational rates}. \\
\\
The outline of this paper is now as follows. We first revisit general objective tensor rates as they appear in rate-form hypo-elastic models and show their shortcomings as compared to corotational rates. Then we motivate and present novel structural conditions for corotational rates. All these conditions can be connected to the fourth order stiffness tensor $\mathbb{A}^{\circ}(B)$ appearing in the defining equation
	\begin{align}
	\frac{\DD^{\circ}}{\DD t}[B] =: \mathbb{A}^{\circ}(B).D \, ,
	\end{align}
where $B = F \, F^T$ is the left Cauchy-Green tensor, $D = \sym \, \DD v$ is the Eulerian stretching and $\frac{\DD^{\circ}}{\DD t}$ is a general corotational rate. The properties of $\mathbb{A}^{\circ}(B)$ are then investigated. In a first approach we use some rather straight forward representation formulas for material spin tensors to arrive at sufficient conditions for positive definiteness of $\mathbb{A}^{\circ}(B)$ along with conditions for the invertibility of $\mathbb{A}^{\circ}(B)$. Thereafter, we refine the analysis of $\mathbb{A}^{\circ}(B)$ by using the eigenprojection representation which allows us to arrive at necessary and sufficient conditions for positive definiteness of $\mathbb{A}^{\circ}(B)$. Finally, we apply our result to well-known corotational rates from the literature.
\subsection{Rate-form constitutive equations and objective stress-rates}
In the traditional sense of Truesdell \cite{truesdellremarks} and Noll \cite{Noll55}, a hypo-elastic material obeys a constitutive law of the following rate-form
	\begin{align}
	\label{eqthedoublestr}
	\frac{\DD^{\sharp}}{\DD t}[\sigma] = \H^{*}(\sigma) . D \qquad \iff \qquad D = [\H^{*}(\sigma)]^{-1} . \frac{\DD^{\sharp}}{\DD t}[\sigma] = \mathbb{S}^{*}(\sigma) . \frac{\DD^{\sharp}}{\DD t}[\sigma].
	\end{align}
In this format, $\frac{\DD^{\sharp}}{\DD t}$ describes an objective rate of the Cauchy stress $\sigma$, $\H^{*}(\sigma)$ is a constitutive fourth-order tangent stiffness tensor, $\mathbb{S}^*(\sigma) = [\H^*(\sigma)]^{-1}$ is the corresponding fourth-order tangent compliance tensor and $D = \sym \DD v$ is the Eulerian strain rate tensor, where $v$ describes the spatial velocity in the current configuration. For a general overview of the notation we refer to the Appendix \ref{appendixnotation}. From the representation \eqref{eqthedoublestr} it is evident that the stretch rate $D$ depends exclusively on the current stress level $\sigma$, along with the corresponding stress rate $\frac{\DD^{\sharp}}{\DD t}[\sigma]$. 

As already indicated by the notation, the objective rate $\frac{\DD^{\sharp}}{\DD t}$ and the constitutive fourth-order tangent stiffness tensor $\H^*$ a-priori need not be related to each other, meaning that once an objective rate $\frac{\DD^{\sharp}}{\DD t}$ is chosen, the tensor $\H^*$ can still be determined arbitrarily. However, if the tensor $\H^*$ is prescribed independently of the chosen objective rate $\frac{\DD^{\sharp}}{\DD t}$, as for example in the case of zero-grade hypo-elasticity, inconsistencies will occur. For this reason it appears to be a sound choice (see also the discussion in \cite[Section 2]{CSP2024}) to pick the induced tangent stiffness tensor $\H^{\sharp}$ as default choice\footnote
{
In general, when $\H^*$ is chosen as in \eqref{eqintro001}, it is defined as a function of the Finger tensor $B = F \, F^T$. Only if the Cauchy stress $B \mapsto \sigma(B)$ is invertible with inverse function $\mathcal{F}^{-1}(\sigma)$, we may write $\H^{\sharp}(B).D = \H^{\sharp}(\mathcal{F}^{-1}(\sigma)).D \equalscolon \H^{\sharp}(\sigma).D$.
},
i.e.
	\begin{align}
	\label{eqintro001}
	\H^*(B).D \colonequals \mathbb{H}^{\sharp}(B).D \colonequals \frac{\DD^{\sharp}}{\DD t}[\sigma].
	\end{align}
Here, $\mathbb{H}^{\sharp} \in \Sym_4(6)$ denotes a minor symmetric fourth order tangent stiffness tensor. It is important to realize that \eqref{eqintro001} describes a consistent constitutive law for every objective rate $\frac{\DD^{\sharp}}{\DD t}$ as long as $\H^{\sharp}(B)$ is determined according to a given Cauchy stress $B \mapsto \sigma(B)$.

In order for such a formulation to transform properly under Euclidean transformation, the stress rate on the left of \eqref{eqintro001} must satisfy the \textbf{objectivity requirement (frame-indifference)}
	\begin{align}
	\label{eqintro002}
	\frac{\DD^{\sharp}}{\DD t}[Q^T \, \sigma \, Q] = Q^T \, \frac{\DD^{\sharp}}{\DD t}[\sigma] \, Q, \qquad \forall \, Q \in \OO(3).
	\end{align}
We observe readily, that the material (or substantial) derivative $\frac{\DD}{\DD t}[\sigma]$ does not satisfy \eqref{eqintro002}. However, there are infinitely many possible pseudo stress rates (derivations) satisfying the invariance condition \eqref{eqintro002}. The only requirements necessary for such stress rates is linearity of the operation $\frac{\DD^{\sharp}}{\DD t}$ and satisfaction of a Leibniz-rule (cf.~\cite[p.~10]{kolev2024objective}) in the form
	\begin{align}
	\frac{\DD^{\sharp}}{\DD t}[f(t) \, \sigma(t)] = f(t) \, \frac{\DD^{\sharp}}{\DD t}[\sigma(t)] + f'(t) \, \sigma(t),
	\end{align}
which holds for an arbitrary objective derivative $\frac{\DD^{\sharp}}{\DD t}$, a differentiable $\sigma: \Sym^{++}(3) \to \Sym(3), B \mapsto \sigma(B)$ and $f \in C^1(\R, \R)$ (cf.~\cite[Remark 3.13]{CSP2024}). Typical representatives for objective derivatives $\frac{\DD^{\sharp}}{\DD t}$ are
	\begin{alignat}{2}
	\frac{\DD^{\CR}}{\DD t}[\sigma] &:= &&\; \frac{\DD}{\DD t}[\sigma] + L^T \, \sigma + \sigma \, L \quad \text{(non-corotational Cotter-Rivlin derivative (cf.~\cite{Cotter1955TENSORSAW})).} \notag \\
	\frac{\DD^{\Old}}{\DD t}[\sigma] &:= &&\; \frac{\DD}{\DD t}[\sigma] - (L \, \sigma + \sigma \, L^T) \quad \text{(non-corotational convective contravariant Oldroyd derivative (cf.~\cite{oldroyd1950}))}, \notag \\
	\frac{\DD^{\text{Hencky}}}{\DD t}[\sigma] &:= &&\; \frac{\DD}{\DD t}[\sigma] + \sigma \, W - W \, \sigma + \sigma \, \tr(D) \quad \text{(non-corotational Biezeno-Hencky derivative (cf.~\cite{biezeno1928})}, \notag \\
	\label{eqobjrates}
	& &&\hspace{5.1cm} \text{sometimes also called Hill-rate (cf.~\cite{korobeynikov2023})),} \\
	\frac{\DD^{\TR}}{\DD t}[\sigma] &:= &&\; \frac{\DD}{\DD t}[\sigma] - (L \, \sigma + \sigma \, L^T) + \sigma \, \tr(D) \quad \text{(non-corotational Truesdell derivative (cf.~\cite[eq.~3]{truesdellremarks})).} \notag
	\end{alignat}
Even though it will not play a role in our development, we would like to point out that there is an intimate relation of these objective rates to Lie-derivatives and covariant derivatives (cf.~\cite{kolev2024objective, Marsden83}).
\subsection{Corotational derivatives - general relations and first properties}
Further investigating which kind of rates are especially meaningful for rate type equations of the format \eqref{eqintro001}, we encounter another important subclass of objective rates: the corotational rates, denoted by $\frac{\DD^{\circ}}{\DD t}$. While \emph{objective} means that the additional rate of rotation, or spin, should be determined solely by the underlying spins in the problem and by the velocity gradient $L$, \emph{corotational} implies that the rate is taken in relation to a frame rotating relative to the observer's position (cf.~\cite[Section 2]{CSP2024}). For further insights on this topic, we refer to the work by Korobeynikov et al.~\cite{korobeynikov2018, korobeynikov2023, korobeynikov2023book, korobeynikov2024} as well as other contributions from different authors \cite{Aubram2017, bellini2015, fiala2009, fiala2016, fiala2020objective, govindjee1997, palizi2020consistent, pinsky1983, xiao97, xiao98_2, zohdi2006}. \\
\\
Recall, that corotational derivatives $\frac{\DD^{\circ}}{\DD t}$ have the general format
	\begin{align}
	\label{equniversalcororate}
	\frac{\DD^{\circ}}{\DD t}[\sigma] = \frac{\DD}{\DD t}[\sigma] - \Omega^{\circ} \, \sigma + \sigma \, \Omega^{\circ} = Q^{\circ} \, \frac{\DD}{\DD t}[(Q^{\circ})^T \, \sigma \, Q^{\circ}] \, (Q^{\circ})^T, \qquad \Omega^{\circ}(t) \in \mathfrak{so}(3),
	\end{align}
where $\Omega^{\circ} = \dot{Q}^{\circ} \, (Q^{\circ})^T$ is the defining \textbf{spin tensor} for some given ``corotated'' frame $Q^{\circ} \in \OO(3)$. The latter equality in \eqref{equniversalcororate} resembles a Lie-type format where the Cauchy stress is backrotated via the corotated frame $Q^{\circ}$. It is important to note, that a corotational rate is objective if and only if the spin tensor $\Omega^{\circ}$ transforms according to 
	\begin{align}
	\label{eqboxtrafo}
	\boxed{\Omega^{\circ} \mapsto \dot{Q} \, Q^T + Q \, \Omega^{\circ} \, Q^T}
	\end{align}
under an arbitrary Euclidean transformation $F \mapsto Q(t) \, F(t)$ (cf.~\cite[p.~49]{CSP2024}). In terms of the corotated frame $Q^{\circ}$, objectivity is expressed as $Q^{\circ} \mapsto Q \, Q^{\circ}$, i.e.~the corotated frame transforms as the continuum rotation under a Euclidean transformation.

The most well-known members of this family of corotational derivatives are perhaps
	\begin{alignat}{2}
	\frac{\DD^{\ZJ}}{\DD t}[\sigma] &:= && \; \frac{\DD}{\DD t}[\sigma] + \sigma \, W - W \, \sigma = Q^W \, \frac{\DD}{\DD t}[(Q^W)^T \, \sigma \, Q^W] \, (Q^W)^T, \quad  \text{for $Q^W(t) \in \OO(3)$ with} \; W = \dot{Q}^W \, (Q^W)^T, \notag \\
	& &&\; \text{where} \; W = \sk L \;\; \text{is the vorticity\footnotemark} \quad (\text{corotational \textbf{Zaremba-Jaumann derivative} (cf.~\cite{jaumann1905, jaumann1911geschlossenes, zaremba1903forme}}), \notag \\
	\frac{\DD^{\GN}}{\DD t}[\sigma] &:= &&\; \frac{\DD}{\DD t}[\sigma] + \sigma \, \Omega^R - \Omega^R \, \sigma = R \, \frac{\DD}{\DD t}[R^T \, \sigma \, R] \, R^T, \quad \text{for $R(t) \in \OO(3)$ with the ``polar spin''} \, \Omega^R := \dot{R} \, R^T, \notag \\
	& &&\; \text{with the polar decomposition} \; F = R \, U \quad \text{(corotational \textbf{Green-Naghdi derivative} (cf.~\cite{bellini2015, Green_McInnis_1967,  Green1965, naghdi1961}))}, \notag \\
	\frac{\DD^{\log}}{\DD t}[\sigma] &:= &&\; \frac{\DD}{\DD t}[\sigma] + \sigma \, \Omega^{\log} - \Omega^{\log} \, \sigma, \quad \text{for $Q^{\log}(t) \in \OO(3)$ with the ``logarithmic spin''} \; \Omega^{\log} = \dot{Q}^{\log} \, (Q^{\log})^T \notag \\
	& &&\; \text{(corotational \textbf{logarithmic derivative} (cf.~\cite{xiao98_1}))}. \\
	\frac{\DD^{\GS}}{\DD t}[\sigma] &:= &&\; \frac{\DD}{\DD t}[\sigma] + \sigma \, \Omega^{\GS} - \Omega^{\GS} \, \sigma, \quad \text{for $Q^{\GS}(t) \in \OO(3)$ with the spin} \; \Omega^{\GS} = \dot{Q}^{\GS} \, (Q^{\GS})^T, \notag \\
	& &&\; \text{where} \quad V = Q^{\GS}  \diag(\lambda_1, \lambda_2, \lambda_3) \, (Q^{\GS})^T \quad \text{(corotational \textbf{Gurtin-Spear derivative} (cf.~\cite{gurtin1983relationship, hill1978}))}. \notag \\
	\frac{\DD^{\Aif}}{\DD t}[\sigma] &:= &&\; \frac{\DD}{\DD t}[\sigma] + \sigma \, \Omega^{\Aif} - \Omega^{\Aif} \, \sigma, \quad \text{with the Aifantis spin} \; \Omega^{\Aif} = W + \zeta \, (\sigma^{\iso} \, D - D \, \sigma^{\iso}), \notag \\
	& &&\; \text{where $\zeta$ is a proportionality factor} \quad \text{(corotational \textbf{Aifantis derivative}\footnotemark (cf.~\cite[eq.~(2.17)]{Zbib1988}))}. \notag
	\end{alignat}
\addtocounter{footnote}{-1}\footnotetext
{
It can be shown that $W = \dot{R}R^T + R \, \sk(\dot{U}U^{-1}) \, R^T$ cf. Gurtin et al. \cite{Gurtin2010}, Nasser et al. \cite{mehrabadi1987} and the books by Ogden \cite[p.126]{Ogden83} and Truesdell \cite[p.21]{truesdell1966}.
}
\addtocounter{footnote}{+1}\footnotetext
{
Using the Richter representation \cite{richter1948isotrope, richter1949hauptaufsatze, Richter50, Richter52} $\sigma(B) = \varphi_0 \, \id + \varphi_1 \, B + \varphi_2 \, B^2$ shows that $\Omega^{\Aif}$ belongs to the class of material spins (cf.~Definition \ref{appmaterialspins}) if $\sigma^{\iso}(\alpha \, B) = \sigma^{\iso}(B)$.
}
In \cite{CSP2024}, the authors discuss the advantages\footnote
{
An early investigation favoring the Zaremba-Jaumann rate is given by \cite[p.~603]{zaremba1903forme}.
}
and disadvantages of using corotational rates in a hypo-elastic framework. For example, even when choosing corotational rates, there is still an infinite number of possibilities, as different choices of the corotated frame $Q^{\circ}$ in formulas \eqref{equniversalcororate} and \eqref{eqboxtrafo} demonstrate.

In any case, corotational rates have several properties that distinguish them from more general objective stress rates like the ones in \eqref{eqobjrates} and Lie-derivatives (cf.~\cite{bruhns2004, guo63, Norris2008, prager1961, prager1962}). Notably, corotational derivatives satisfy the product rule\footnote
{
Guo \cite[p.157]{guo63} and Prager \cite[eq.~(2.6)]{prager1962} have already observed this property for the Zaremba-Jaumann rate.
}
(Proposition \ref{prodrulecoro}) and a hitherto unknown universal chain rule\footnote
{
In \cite[p.~920]{bruhns2004} the authors mention: ``Since the Euclidean structure of the Galilean space-time remains unchanged under the change of frame, a corotational rate associated with a spinning frame obeys the basic rules for derivatives, \emph{such as the product rule and Leibniz chain rule, etc}. However, the same might not be true for a non-corotational rate associated with a convective frame, since the Euclidean structure of the Galilean space-time is distorted under the change of frame.'' Unfortunately, there is \emph{no proof} for this statement or it might refer only to the primary matrix case \cite{xiao98_1}.
}
(Proposition \ref{apppropa20}) for isotropic tensor functions (for a proof we refer to \cite[Section 3.2.3]{CSP2024}). A basic ingredient in the related calculus is (cf.~\cite[Lemma A.13]{CSP2024} for a proof and more details)
\begin{rem} \label{theoneremark}
Consider an \textbf{arbitrary corotational derivative} $\frac{\DD^{\circ}}{\DD t}$ (not necessarily objective) with spin tensor $\Omega^{\circ} \in \mathfrak{so}(3)$ for an isotropic function $\sigma = \sigma(B)$, i.e.
	\begin{align}
	\frac{\DD^{\circ}}{\DD t}[\sigma] = \frac{\DD}{\DD t}[\sigma] - \Omega^{\circ} \, \sigma + \sigma \, \Omega^{\circ}.
	\end{align}
Then the general relation
	\begin{equation}
	\label{eqgeneralomsig}
	\boxed{\Omega^{\circ} \, \sigma(B) - \sigma(B) \, \Omega^{\circ} = \DD_B\sigma(B).[\Omega^{\circ} \, B - B \, \Omega^{\circ}],}
	\end{equation}
holds, alternatively expressed via the Lie-bracket $[A,B] = A \, B - B \, A$ as
	\begin{equation}
	\boxed{[\Omega^{\circ}, \sigma(B)] = \DD_B\sigma(B).[\Omega^{\circ}, B].}
	\end{equation}
\end{rem}
\begin{prop}[Product rule for corotational rates] \label{prodrulecoro}
For two isotropic and differentiable tensor functions $\sigma_1, \sigma_2: \Sym^{++}(3) \to \Sym(3)$ we have the product rule
	\begin{align}
	\frac{\DD^{\circ}}{\DD t}[\sigma_1(B) \, \sigma_2(B)] = \frac{\DD^{\circ}}{\DD t}[\sigma_1(B)] \, \sigma_2(B) + \sigma_1(B) \, \frac{\DD^{\circ}}{\DD t}[\sigma_2(B)] \, .
	\end{align}
\end{prop}
\begin{prop}[Chain rule for corotational rates] \label{apppropa20}
Let $\frac{\DD^{\circ}}{\DD t}$ be an arbitrary corotational rate with spin tensor $\Omega^{\circ} \in \mathfrak{so}(3)$ and an isotropic, differentiable function $\sigma = \sigma(B) = \widehat \sigma(\log B)$. Then we have the chain rule
	\begin{align}
	\frac{\DD^{\circ}}{\DD t}[\widehat \sigma] = \DD_{\log B} \widehat \sigma(\log B) . \frac{\DD^{\circ}}{\DD t}[\log B] = \DD_B \sigma(B). \frac{\DD^{\circ}}{\DD t}[B].
	\end{align}
\end{prop}
\begin{proof}
For the proof we refer to \cite[Section 3.2.3]{CSP2024}, also using results from \cite[p.~19, Theorem 2]{xiao98_1}, \cite[p.~7]{fiala2020objective}, \break \cite[p.~1066, Theorem 2.3]{korobeynikov2018} and \cite[p.~252, Lemma 1]{Norris2008}, where the chain rule for primary matrix functions is supplied.
\end{proof}
\noindent Note, that an arbitrary objective rate $\frac{\DD^{\sharp}}{\DD t}$ (like e.g.~the Truesdell rate) in general neither satisfies a chain nor a product rule (cf.~\cite[p.22-23]{CSP2024}).

Another difference (see also the discussion in \cite[p.~919]{bruhns2004}) between corotational rates $\frac{\DD^{\circ}}{\DD t}$ and arbitrary objective rates $\frac{\DD^{\sharp}}{\DD t}$ is displayed by the fact that for any corotational rate we have the identity\footnote
{
This identity was already observed by Guo \cite[p.157]{guo63} for the Zaremba-Jaumann rate.
}
	\begin{align}
	\label{eqcid}
	\frac{\DD^{\circ}}{\DD t}[c \, \id] = \frac{\DD}{\DD t}[c \, \id] + c \, \id \, \Omega^{\circ} - \Omega^{\circ} \, c \, \id = 0, \qquad c = \text{const.} \, ,
	\end{align}
while e.g.~for the Truesdell (non-corotational) rate it is
	\begin{align}
	\label{eqTRnoncor}
	\frac{\DD^{\TR}}{\DD t}[c \, \id] = \frac{\DD}{\DD t}[c \, \id] - (L \, c \, \id + c \, \id \, L^T) + c \, \id \, \tr(D) = c \, (\tr(D) \, \id - 2 \, D) \neq 0.
	\end{align}
More generally, for a corotational rate $\frac{\DD^{\circ}}{\DD t}$ with spin tensor $\Omega^{\circ}$ and any constant tensor $S \in \Sym(3)$ we obtain
	\begin{align}
	\frac{\DD^{\circ}}{\DD t}[S] = \frac{\DD}{\DD t}[S] + S \, \Omega^{\circ} - \Omega^{\circ} \, S = [S, \Omega^{\circ}].
	\end{align}

Furthermore, corotational rates conserve physical properties of the stress tensor $\sigma$. More precisely, if $I(\sigma(t))$ is any isotropic scalar invariant of $\sigma$, i.e.~$I(\sigma(t))$ fulfills
	\begin{align}
	I(\sigma(t)) = I(Q^T(t) \, \sigma(t) \, Q(t)) \qquad \text{for any $Q(t) \in \OO(3),$}
	\end{align}
then (cf.~\cite{prager1961} and \cite[p.~6]{CSP2024}) using $\eqref{equniversalcororate}_2$
	\begin{align}
	\label{eqlemma3.1}
	\frac{\DD^{\circ}}{\DD t}[\sigma] = 0 \qquad \implies \qquad \frac{\DD}{\DD t}I(\sigma(t)) = 0.
	\end{align}
The latter implies that if $\frac{\DD^{\circ}}{\DD t}[\sigma] = 0$, all eigenvalues of $\sigma(t)$ remain constant. Similarly, we can see that for all corotational rates $\frac{\DD^{\circ}}{\DD t}$
	\begin{align}
	2 \, \langle \frac{\DD^{\circ}}{\DD t}[\sigma], \sigma \rangle = \frac{\DD}{\DD t}[\norm{\sigma}^2] \, ,
	\end{align}
generalizing a formula for the Zaremba-Jaumann rate given in \cite[p.~193]{holzapfel2000}.

Lastly, for a perfect elastic fluid (cf.~\cite[p.~10]{Marsden83}), i.e.~$\sigma$ is given by $\sigma(B) = h'(\sqrt{\det B}) \, \id$, all corotational rates $\frac{\DD^{\circ}}{\DD t}$ of $\sigma$ coincide. To see this, we first use the special structure $\sigma(B) = \alpha(B) \, \id$ of $\sigma$ to conclude that
	\begin{align}
	\frac{\DD^{\circ}}{\DD t}[\sigma] = \frac{\DD}{\DD t}[\sigma] \underbrace{+ \sigma \, \Omega^{\circ} - \Omega^{\circ} \, \sigma}_{= \, 0 \; \text{by the structure of $\sigma$}} = \frac{\DD}{\DD t}[\sigma].
	\end{align}
Then we calculate directly
	\begin{equation}
	\label{eqperfcomp01}
	\begin{alignedat}{2}
	\frac{\DD}{\DD t}[h'(\sqrt{\det B})] &= h''(\sqrt{\det B}) \, \frac12 \, (\sqrt{\det B})^{-\frac12} \, \frac{\DD}{\DD t}[\det B] = h''(\sqrt{\det B}) \, \frac12 \, (\sqrt{\det B})^{-\frac12} \, \langle \Cof B, \frac{\DD}{\DD t}[B] \rangle \\
	&= h''(\sqrt{\det B}) \, \frac12 \, (\sqrt{\det B})^{-\frac12} \, \det B \, \langle B^{-1}, L \, B + B \, L^T \rangle \\
	&= h''(\sqrt{\det B}) \, \frac12 \, \sqrt{\det B} \, \langle \id, L + L^T \rangle = h''(\sqrt{\det B}) \, \sqrt{\det B} \, \tr(D) \, ,
	\end{alignedat}
	\end{equation}
yielding for every corotational rate $\frac{\DD^{\circ}}{\DD t}$ the expression
	\begin{align}
	\H^{\circ}(\sigma).D = \frac{\DD^{\circ}}{\DD t}[\sigma] = \frac{\DD}{\DD t}[h'(\sqrt{\det B}) \, \id] = h''(\sqrt{\det B}) \, \sqrt{\det B} \, \tr(D) \, \id \, .
	\end{align}
For non-corotational rates the same does not hold. For example, the Truesdell derivative for this choice of $\sigma(B)$ is given by
	\begin{align}
	\frac{\DD^{\TR}}{\DD t}[\sigma] = h''(\sqrt{\det B}) \, \sqrt{\det B} \, \tr(D) \, \id + h'(\sqrt{\det B}) \, (\tr(D) \, \id - 2 \, D),
	\end{align}
which can be seen from the identity $\frac{\DD}{\DD t}[h'(\sqrt{\det B})] = h''(\sqrt{\det B}) \, \sqrt{\det B} \, \tr(D)$ proven in \eqref{eqperfcomp01} combined with \eqref{eqTRnoncor}.
\subsection{Structure preserving properties - motivation from the one-dimensional case}
Next, we propose further qualitative properties that may single out a physically useful class of corotational rates among all possible corotational rates. For the general idea, we first reconsider the one-dimensional case. Consider a monotone increasing, differentiable scalar function $e: \R^+ \to \R$. We can express the monotonicity of $e$ as the local condition $\DD e(\ell) = e'(\ell) \ge 0$. We say $e: \R^+ \to \R$ is \textbf{strongly monotone}, if $\DD e(\ell) > 0 \; \forall \, \ell \in \R^+$. Considering a differentiable parametrization $t \mapsto \ell(t) \in \R^+$, we may rewrite the strong monotonicity of $e$ in a pseudo rate-type format, requiring for all parametrizations $t \mapsto \ell(t) \in \R^+$
	\begin{align}
	\label{eqonedimmon}
	\frac{\dif}{\dif t}[e(\ell(t))] \, \dot \ell(t) > 0, \quad \forall \, \dot \ell(t) \neq 0 \qquad \iff \qquad \DD e(\ell(t)) \, \dot \ell(t) \, \dot \ell(t) > 0 \qquad \iff \qquad \DD e(\ell(t)) > 0.
	\end{align}
In view of \eqref{eqcid}, we trivially have $\DD[c \cdot 1] = 0$. We now interpret the positivity $\eqref{eqonedimmon}_1$ as a property of the usual derivative operation ``$\frac{\dif}{\dif t}$'' in conjunction with strongly monotone functions $e(\ell(t))$. This means that multiplying $\frac{\dif}{\dif t}[e(\ell(t))]$ by $\dot \ell(t)$ shows monotonicity of $e$.

Transfering this idea to corotational rates $\frac{\DD^{\circ}}{\DD t}$, we would like to have an ``equivalent'' structural property satisfied. In the three-dimensional setting, generalizing the strongly monotone function $e$ from above, we restrict our attention to \textbf{spatial strain tensors} (cf.~\cite[p.~509]{Neff_Osterbrink_Martin_Hencky13}), i.e.~primary matrix functions $\mathcal{E}: \Sym^{++}(3) \to \Sym(3)$, satisfying for any $Q \in \OO(3)$ the defining relation
	\begin{equation}
	\label{eqdefstrain1}
	\begin{alignedat}{2}
	&\left\{
		\begin{array}{ll}
		&\mathcal{E}(B) = \mathcal{E}(Q^T \diag(B) \, Q) \overset{\text{isotropy}}{=} Q^T \, \mathcal{E}(\diag(B)) \, Q \\
		= &Q^T \, \mathcal{E}(\diag(\lambda_1^2, \lambda_2^2, \lambda_3^2)) \, Q \overset{\substack{\text{primay matrix} \\ \text{function}^\ast}}{=} Q^T \, \diag(e(\lambda_1^2), e(\lambda_2^2), e(\lambda_3^2)) \, Q
		\end{array}
	\right. \\
	& \qquad \text{and} \qquad \mathcal{E}(B) = 0 \qquad \iff \qquad B = \id
	\end{alignedat}
	\end{equation}
where $e: \R^+ \to \R$ is an appropriately chosen\footnote
{
The exact conditions for the scale function differ among authors. For instance Hill \cite[p.~459]{hill1970constitutive} and \cite[p.~14]{hill1978} requires $e$ to be ``suitably smooth'' and monotone with $e(1) = 0$ and $e'(1) = 1$, whereas Ogden \cite[p.~118]{Ogden83} also requires $e$ to be infinitely differentiable and $e'>0$ to hold on all of $\R^+$ (see also \cite[p.~509]{Neff_Osterbrink_Martin_Hencky13}).
}
\textbf{scale function}. In our context, ``appropriate'' requires that the scale function $e: \R^+ \to \R$ is differentiable, \textbf{strongly monotone increasing}, i.e.~$\DD e(\lambda)>0 \; \forall \, \lambda \in \R^+$, and fulfills a normalization property, depending on the chosen strain (e.g.~the Seth-Hill family with \break $e(1) = 2 \, e'(1) - 1 = 0$, see Section \ref{secsethhill}). 
\begin{rem}
Tensor functions $p$ that only satisfy the relation $\eqref{eqdefstrain1}_{\ast}$ are here called ``primary matrix functions''. Sometimes they are denoted by classical or analytic tensor functions. Primary matrix functions form a strict subclass of isotropic tensor functions. Some prominent examples for primary matrix functions are $p(B) = B$, $p(B) = \log B$ and $p(B) = \sqrt{B}$, while $p(B) = \det B$ is not a primary matrix function.
\end{rem}
As a first observation we point out that any spatial strain measure $\mathcal{E}(B) = \widehat{\mathcal{E}}(\log B)$ vanishes for a rigid motion since $F = Q(t) \in \SO(3)$ implies $B = F\, F^T = \id$. On the other hand in the spatial picture rigid motions correspond to $D \equiv 0$. Hence, for a rigid motion, we are led to impose the equivalence
	\begin{align}
	\label{eqinvert001}
	0 = \frac{\DD^{\circ}}{\DD t}[\mathcal{E}(B)] \overset{\text{chain rule}}{=} \DD_B \mathcal{E}(B). \frac{\DD^{\circ}}{\DD t}[B] \overset{\text{chain rule}}{=} \DD_{\log B} \widehat{\mathcal{E}}(B). \DD_B \log B. \frac{\DD^{\circ}}{\DD t}[B] \qquad \iff \qquad D = 0 \, .
	\end{align}
Since $\DD_B \log B \in \Sym^{++}_4(6)$ as well as $\DD_{\log B} \widehat{\mathcal{E}}(B) \in \Sym^{++}_4(6)$ due to the strong monotonicity of the scale function $e$ and the logarithm (cf.~\cite{CSP2024}), we see that requirement \eqref{eqinvert001} is equivalent to
	\begin{align}
	\label{eqinvert002}
	\frac{\DD^{\circ}}{\DD t}[B] = 0 \qquad \iff \qquad D = 0 \, .
	\end{align}
\begin{rem}
Note the following relation of \eqref{eqinvert001} to the one-dimensional case:
	\begin{align}
	\label{eqreme01}
	\left(\frac{\dif}{\dif t} e(\ell(t)) = 0 \quad \text{and} \quad e(1) = 0\right) \quad \implies \quad \bigg(e(\ell(t)) = 0 \quad \implies \quad \ell(t) = 1 \quad \implies \quad \dot \ell(t) = 0\bigg) \, ,
	\end{align}
which translates in three dimensions to
	\begin{align}
	\left(\frac{\DD^{\circ}}{\DD t} \mathcal{E}(B) = 0 \quad \text{and} \quad \mathcal{E}(\id) = 0\right) \quad \overset{!}{\implies} \quad \bigg(\mathcal{E}(B) = 0 \quad \implies \quad B = \id \quad \implies \quad D = 0\bigg) \, .
	\end{align}
Note also that \eqref{eqreme01} can be rewritten using $\widehat e \colon \R \to \R, \; \widehat e(\log \ell(t)) := e(\ell(t))$ with $\widehat e(0) = 0$ and observe\footnote
{
In a one-dimensional setting we have $F(t) = \diag(\ell(t),1,1)$ for $\varphi(x,t) = (\ell(t) \, x_1, x_2, x_3)$ with $\dot F(t) = \diag(\dot \ell(t), 0, 0), \break F^{-1}(t) = \diag(\frac{1}{\ell(t)}, 1, 1)$ so that $D = \sym \, L = \diag(\dot \ell(t), 0, 0) \cdot \diag(\frac{1}{\ell(t)}, 1, 1) = \diag(\frac{\dot \ell(t)}{\ell(t)}, 0, 0)$. This justifies denoting $\frac{\dot \ell(t)}{\ell(t)} \cong \overline D$.
}
	\begin{align}
	\frac{\dif}{\dif t}[\widehat e(\log \ell(t))] = \DD_{\log \ell} \widehat e(\log \ell(t)) \, \DD_{\ell} \log \ell(t) \, \dot \ell(t) = \DD_{\log \ell} \widehat e(\log \ell(t)) \, \frac{\dot \ell(t)}{\ell(t)} = \DD_{\log \ell} \widehat e(\log \ell(t)) \, \overline D, \quad \overline D = \frac{\dot \ell(t)}{\ell(t)} \, ,
	\end{align}
so that we can also write 
	\begin{align}
	\forall \, \overline D \neq 0: \qquad \frac{\dif}{\dif t}[\widehat e(\log \ell(t))] \, \overline D > 0 \qquad \iff \qquad \DD_{\log \ell} \widehat e (\log \ell(t)) > 0 \, .
	\end{align}
\end{rem}
If we have a representation $\frac{\DD^{\circ}}{\DD t}[B] = \mathbb{A}^{\circ}(B).D$ then \eqref{eqinvert002} demands invertibility\footnote
{
Xiao et al.~\cite{xiao98_1} write, regarding the invertibility of an arbitrary objective corotational rate $\frac{\DD^{\circ}}{\DD t}$ in our notation: ``Now another relevant question is wheter or not a given objective corotational rate $\frac{\DD^{\circ}}{\DD t}$ can serve as a complete measure of the rate of change of deformation. Precisely, we say that an objective corotational rate $\frac{\DD^{\circ}}{\DD t}$ is equivalent to the stretching $D$ if there is a one-to-one correspondence between them.'' Additionally, they have observed \cite[eq.~(4.68)]{xiao98_1} that $\mathbb{A}^{\GS}(B)$ is not invertible, where $\mathbb{A}^{\GS}(B)$ denotes the Gurtin-Spear corotational rate $\frac{\DD^{\GS}}{\DD t}[B] = \mathbb{A}^{\GS}(B).D$. Thus, $\frac{\DD^{\GS}}{\DD t}$ is not an invertible rate.
}
of the fourth order stiffness tensor $\mathbb{A}^{\circ}(B)$. Thus we are led to
\begin{definition}[Invertible corotational rates]
An arbitrary corotational rate $\frac{\DD^{\circ}}{\DD t}$ is called \textbf{invertible corotational rate}, if for all $B = F \, F^T \in \Sym^{++}(3)$:
	\begin{align}
	\frac{\DD^{\circ}}{\DD t}[B] = 0 \qquad \iff \qquad D = 0.
	\end{align}
\end{definition}
\noindent Analogously, the considerations about strong monotonicity in one dimension lead to the following new
\begin{definition}[Positive corotational rates] \label{appendixpcd}
An arbitrary corotational rate $\frac{\DD^{\circ}}{\DD t}$ is called \textbf{positive corotational rate}, if for all $B = F \, F^T \in \Sym^{++}(3)$:
	\begin{align}
	\label{eqdefposcor001}
	\langle \frac{\DD^{\circ}}{\DD t}[B], D \rangle > 0 \qquad \forall \, D \in \Sym(3) \! \setminus \! \{0\}.
	\end{align}
In other words, for $\mathbb{A}^{\circ}(B)$ defined by 
	\begin{align}
	\frac{\DD^{\circ}}{\DD t}[B] = \mathbb{A}^{\circ}(B).D
	\end{align}
we have $\mathbb{A}^{\circ}(B) \in \Sym^{++}_4(6)$.
\end{definition}
\noindent Recalling some well-known corotational rates like the Zaremba-Jaumann, the Green-Naghdi or the logarithmic (objective) corotational rate, we readily observe
	\begin{equation}
	\label{eqprominentposdef}
	\begin{alignedat}{2}
	\langle \frac{\DD^{\ZJ}}{\DD t}[B], D \rangle = \langle B \, D + D \, B, D \rangle &> 0, \qquad \langle \frac{\DD^{\GN}}{\DD t}[B], D \rangle = 2 \, \langle V \, D \, V, D \rangle > 0 \\
	\text{and} \qquad \langle \frac{\DD^{\log}}{\DD t}[B], D \rangle &= 2 \, \langle [\DD_B \log B]^{-1}. D , D \rangle > 0 \, ,
	\end{alignedat}
	\end{equation}
showing that these corotational rates belong to the newly defined classes of invertible and even positive corotational rates (cf.~\cite[Section 3.4.1]{CSP2024}). \\
\\
Motivated by our foregoing considerations, we additionally conjecture
\begingroup
\renewcommand*{\arraystretch}{2.2}
\begin{conjecture} \label{conj1.5}
	Let $\mathcal{E}: \Sym^{++}(3) \to \Sym(3)$ be a spatial strain tensor with strongly monotone scale function $e: \R^+ \to \R$. Then we have the equivalence:
	\begin{align*}
	\left\{
		\begin{array}{l}
		\dd \frac{\DD^{\circ}}{\DD t} \; \text{is a positive corotational rate} \\
		\langle \dd \frac{\DD^{\circ}}{\DD t}[B], D \rangle > 0
		\end{array}
	\right. \quad \iff \quad \left\{
		\begin{array}{l}
		\text{(1D):} \quad \left\{
			\begin{array}{l}
			\dd \frac{\dif}{\dif t}[e(\ell(t))] \, \dot \ell(t) > 0, \qquad \forall \, \dot \ell(t) \neq 0, \\
			\dd \frac{\dif}{\dif t}[\widehat e(\log \ell(t)] \, \overline D > 0, \qquad \forall \, \overline D = \frac{\dot \ell(t)}{\ell(t)} \neq 0,
			\end{array} \right. \\
		\text{(3D):} \quad \langle \dd \frac{\DD^{\circ}}{\DD t}[\mathcal{E}(B)], D \rangle > 0 \qquad \forall \, D \in \Sym(3) \! \setminus \! \{0\}.
		\end{array}
	\right.
	\end{align*}
\end{conjecture}
\endgroup
\begin{rem}
The $\Longleftarrow \;$ direction is clear by choosing $\mathcal{E}(B) = \frac12 (B - \id)$ and noting that $\frac{\DD^{\circ}}{\DD t}[\id] = 0$ for corotational rates. It remains to show that
	\begin{align}
	\forall \, D \in \Sym(3) \! \setminus \! \{0\}: \qquad \langle \frac{\DD^{\circ}}{\DD t}[B], D \rangle > 0 \qquad \implies \qquad \langle \frac{\DD^{\circ}}{\DD t}[\mathcal{E}(B)], D \rangle = \langle \DD_B \mathcal{E}(B). \frac{\DD^{\circ}}{\DD t}[B], D \rangle > 0.
	\end{align}
\end{rem}
\begin{prop}
Conjecture \ref{conj1.5} is true for the Zaremba-Jaumann rate.
\end{prop}
\begin{proof}
As seen in \eqref{eqprominentposdef}, the inequality $\langle \frac{\DD^{\ZJ}}{\DD t}[B], D \rangle > 0$ is satisfied. By monotonicity of $\widehat{\mathcal{E}}(\log B)$ in $\log B$, we have $\DD_{\log B} \widehat{\mathcal{E}}(\log B) \in \Sym^{++}_4(6)$ and thus by using Theorem 3.1 of \cite{CSP2024}, we obtain
	\begin{align}
	\langle \frac{\DD^{\ZJ}}{\DD t}[\mathcal{E}(B)], D \rangle &= \langle \underbrace{\DD_{\log B} \widehat{\mathcal{E}}(\log B)}_{\in \Sym^{++}_4(6)}. \DD_B \log B. [B \, D + D \, B], D \rangle > 0 \notag \qedhere
	\end{align}

\end{proof}
\begin{rem}[Structure preserving property]
The latter signifies that the positive corotational rates are able to reveal the monotonicity of the spatial strain tensors $\mathcal{E}(B)$ by evaluating $\langle \frac{\DD^{\circ}}{\DD t}[\mathcal{E}(B)], D \rangle > 0$.
\end{rem}

\subsubsection{Examples from the Seth-Hill family} \label{secsethhill}
As described in \cite[p.~510]{Neff_Osterbrink_Martin_Hencky13} and \cite[p.~912]{bruhns2004}, some of the most common examples of spatial strain tensors $\mathcal{E}$ used in nonlinear elasticity is the \textbf{Seth-Hill family} \cite{seth1961}
\begingroup
\renewcommand*{\arraystretch}{2.2}
	\begin{align}
	\label{eqsethhill01}
	\mathcal{E}_m(B) = \left\{
		\begin{array}{ll}
		\dd \frac{1}{2m}(B^m-1), &\text{if} \quad m \in \R \! \setminus \! \{0\} \\
		\dd \frac12 \, \log B, &\text{if} \quad m = 0,
		\end{array}
	\right.
	\end{align}
using the scale functions
	\begin{align}
	\label{eqsethhill}
	e_m(\chi) = \left\{
		\begin{array}{ll}
		\dd \frac{1}{2m}(\chi^m-1), &\text{if} \quad m \in \R \! \setminus \! \{0\} \\
		\dd \frac12 \, \log \chi, &\text{if} \quad m = 0.
		\end{array}
	\right.
	\end{align}
\endgroup
Some examples for scale functions $e$ belonging to this family are depicted in Figure \ref{figsethhill}.

	\begin{figure}[h!]
		\begin{center}		
		\begin{minipage}[h!]{0.4\linewidth}
			\centering
			\includegraphics[scale=0.25]{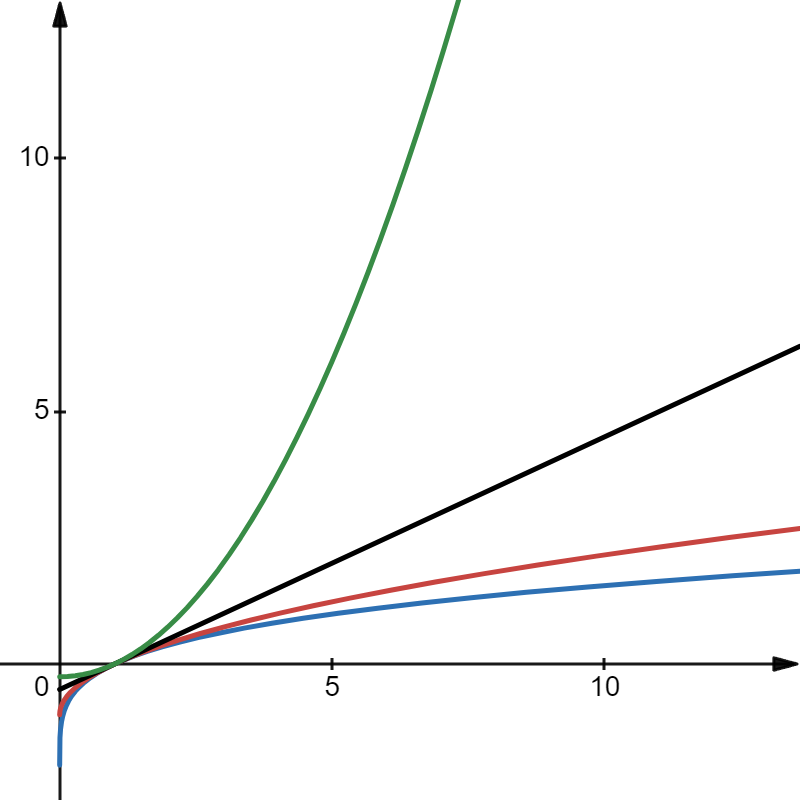}
			\put(-7,22){\footnotesize $\chi$}
			\put(-180,190){\footnotesize $e_m(\chi)$}
		\end{minipage} \qquad
		\begin{minipage}[h!]{0.4\linewidth}
			\centering
			\includegraphics[scale=0.25]{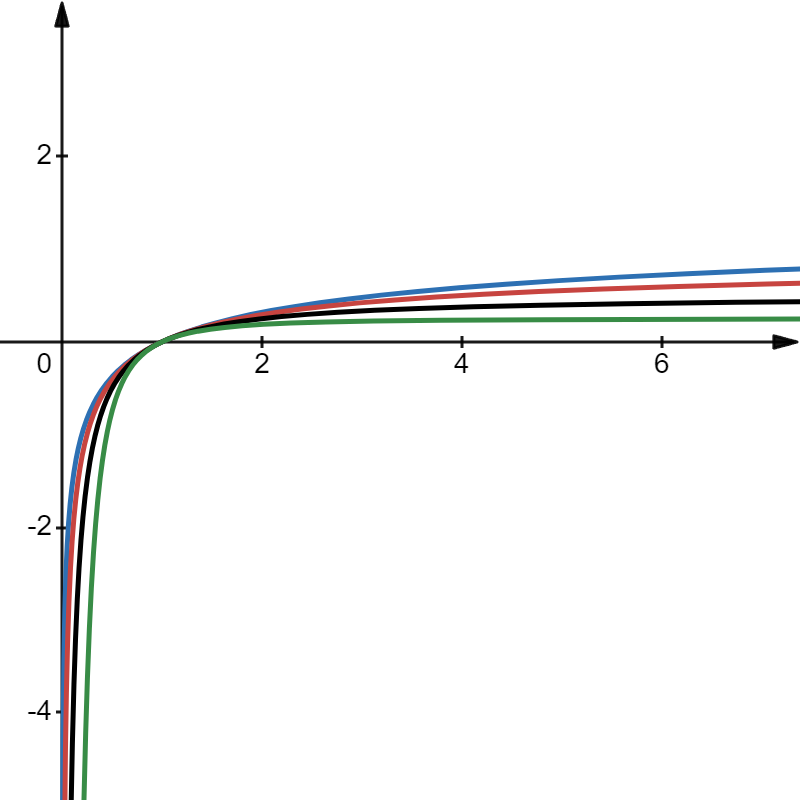}
			\put(-10,105){\footnotesize $\chi$}
			\put(-180,190){\footnotesize $e_m(\chi)$}
		\end{minipage}
		\caption{Picture of \textbf{monotone scale functions} from the Seth-Hill family \eqref{eqsethhill} (left $e_m(\chi) = \frac{1}{2m}(\chi^m-1)$ and right $e_m(\chi)= \frac{1}{2m}(1-\chi^{-m})$ for $\blue{m=\frac{1}{4}}, \red{m=\frac{1}{2}}, m=1$ and $\teal{m=2}$).}
		\label{figsethhill}
		\end{center}
	\end{figure}

\noindent Let us exemplary show positivity of some corotational rates evaluated for some spatial strain tensors $\mathcal{E}(B)$ belonging to the Seth-Hill family.
\clearpage
\begin{example}[Green-Naghdi, $m=0$ {[Hencky's logarithmic strain]}]
For $m=0$ we obtain \newline $\mathcal{E}(B) = \log B = 2 \, \log V$. In combination with the Green-Naghdi rate this leads to
	\begin{align}
	\label{eqGNdalekii}
	\langle \frac{\DD^{\GN}}{\DD t}\bigg[\underbrace{\frac12 \, \log B}_{\substack{\text{strain} \\ \text{tensor}}}\bigg], D \rangle \overset{\text{chain rule}}{=} \frac12 \, \langle \DD_B \log B. \frac{\DD^{\GN}}{\DD t}[B], D \rangle = \frac12 \, \underbrace{\langle \DD_B \log B. [V \, D \, V], D \rangle}_{\text{Daleckii-Krein\footnotemark}} > 0 \qquad \forall \, D \in \Sym(3) \! \setminus \! \{0\}.
	\end{align}
\end{example}
\footnotetext
{
With the help of the Daleckii-Krein formula \cite{daleckii1965}, it is possible to derive a semi-explicit expression for $\DD_B \log B$ given by $\DD_B \log B . H = R \, [\mathcal{F} \circ (R^T \, H \, R)] \, R^T$, where $\mathcal{F}$ is a matrix with non-negative entries, $R \in \OO(3)$ and $\circ$ denotes the Schur product \cite{horn2013} (cf.~\cite[Corollary A.24]{CSP2024}). Considering the scalar product in \eqref{eqGNdalekii}, this formula leads to
	\begin{equation}
	\begin{alignedat}{2}
	\langle \DD_B \log B.[V \, D \, V], D \rangle &= \langle R \, [\mathcal{F} \circ (R^T \, [V \, D \, V] \, R)] \, R^T, D \rangle = \langle \mathcal{F} \circ (R^T \, [V \, D \, V] \, R), R^T \, D \, R \rangle \\
	&\ge \min \, \mathcal{F} \, \langle \underbrace{R^T \, V \, R}_{\diag(\lambda_1, \lambda_2, \lambda_3)} \, \underbrace{R^T \, D \, R}_{S} \, \underbrace{R^T \, V \, R}_{\diag(\lambda_1, \lambda_2, \lambda_3)}, \underbrace{R^T \, D \, R}_{S} \rangle > 0.
	\end{alignedat}
	\end{equation}
}
\begin{example}[Zaremba-Jaumann, $m=1$ {[Finger strain]}]
	\begin{align}
	\langle \frac{\DD^{\ZJ}}{\DD t}\bigg[\underbrace{\frac12 \, (B - \id}_{\substack{\text{strain} \; \text{tensor}}})\bigg], D \rangle = \frac12 \, \langle D \, B + B \, D, D \rangle > 0, \qquad \forall \, D \in \Sym(3) \! \setminus \! \{0\}.
	\end{align}
\end{example}
\begin{example}[Zaremba-Jaumann, $m=2$] \label{exZJm2}
	\begin{align}
	\langle \frac{\DD^{\ZJ}}{\DD t}\bigg[\underbrace{\frac14 \, (B^2 - \id}_{\substack{\text{strain} \\ \text{tensor}}})\bigg], D \rangle \overset{\text{chain rule}}&{=} \frac14 \, \langle B \, H + H \, B\bigg\vert_{H = \frac{\DD^{\ZJ}}{\DD t}[B]}, D \rangle = \frac14 \, \langle B \, (B \, D + D \, B) + (B \, D + D \, B) \, B, D \rangle \notag \\
	&= \frac14 \, \langle B^2 \, D + B \, D \, B + B \, D \, B + D \, B^2, D \rangle > 0, \qquad \forall \, D \in \Sym(3) \! \setminus \! \{0\}.
	\end{align}
\end{example}
\begin{example}[Zaremba-Jaumann, $m=k$]
Similarly to the previous example we obtain for $k \in \mathbb{N}$
	\begin{equation}
	\label{eqZJmk}
	\begin{alignedat}{2}
	\langle \frac{\DD^{\ZJ}}{\DD t}\bigg[\underbrace{\frac{1}{2k} \, (B^k - \id}_{\substack{\text{strain} \\ \text{tensor}}})\bigg], D \rangle &= \frac{1}{2k} \, \langle \sum_{i=0}^k B^i \, H \, B^{k-i} \bigg\vert_{H = \frac{\DD^{\ZJ}}{\DD t}[B]}, D \rangle = \frac{1}{2k} \, \langle \sum_{i=0}^k B^i \, (B \, D + D \, B) \, B^{k-i}, D \rangle \\
	&= \frac{1}{2k} \, \langle \sum_{i=0}^k B^{i+1} \, D \, B^{k-i} + \sum_{i=0}^k B^i \, D \, B^{k-i+1}, D \rangle > 0, \qquad \forall \, D \in \Sym(3) \! \setminus \! \{0\}.
	\end{alignedat}
	\end{equation}
\end{example}
\begin{example}[Zaremba-Jaumann, $m=-1$ {[Almansi strain]}]
	\begin{align}
	\langle \frac{\DD^{\ZJ}}{\DD t}\bigg[\underbrace{\frac12 \, (\id - B^{-1}}_{\substack{\text{strain} \\ \text{tensor}}})\bigg], D \rangle &= -\frac12 \, \langle \frac{\DD^{\ZJ}}{\DD t}[B^{-1}], D \rangle \overset{\text{chain rule}}{=} -\frac12 \, \langle -B^{-1} \, \left( \frac{\DD^{\ZJ}}{\DD t}[B] \right) \, B^{-1}, D \rangle \\
	&= \frac12 \, \langle B^{-1} \, (B \, D + D \, B) \, B^{-1}, D \rangle = \langle D \, B^{-1}, D \rangle + \langle B^{-1} \, D , D \rangle > 0, \quad \forall \, D \in \Sym(3) \! \setminus \! \{0\}. \notag
	\end{align}
\end{example}
\begin{example}[Zaremba-Jaumann, $m=-k$]
For this example, first observe the relation
	\begin{equation}
	\begin{alignedat}{2}
	0 = \DD_B(\id).H = \DD_B(B^k \, B^{-k}).H &= (\DD_B B^k.H) \, B^{-k} + B^k \, (\DD_B B^{-k}.H) \\
	\iff \qquad \DD_B B^{-k}.H &= -B^{-k} \, (\DD_B B^k.H) \, B^{-k}.
	\end{alignedat}
	\end{equation}
Then we obtain similarly to the previous two examples
	\begin{align}
	\langle \frac{\DD^{\ZJ}}{\DD t}\bigg[\underbrace{\frac{1}{2k} \, (\id - B^{-k}}_{\substack{\text{strain} \\ \text{tensor}}})\bigg], D \rangle &= -\frac{1}{2k} \, \langle \frac{\DD^{\ZJ}}{\DD t}[B^{-k}], D \rangle \overset{\text{chain rule}}{=} -\frac{1}{2k} \, \langle -B^{-k} \, \left( \frac{\DD^{\ZJ}}{\DD t}[B^k] \right) \, B^{-k}, D \rangle \notag \\
	\overset{\eqref{eqZJmk}}&{=} \frac{1}{2k} \, \langle B^{-k} \, \left( \sum_{i=0}^k B^i \, (B \, D + D \, B) \, B^{k-i} \right) \, B^{-k}, D \rangle \\
	&= \frac{1}{2k} \, \langle \sum_{i=0}^k B^{i+1-k} \, D \, B^{-i} + \sum_{i=0}^k B^{i-k} \, D \, B^{-i+1}, D \rangle > 0, \qquad \forall \, D \in \Sym(3) \! \setminus \! \{0\} \, . \notag
	\end{align}
\end{example}
\section{Tangent stiffness tensors for a family of corotational rates} \label{app13}
In order to approach the problem of classification of invertible or positive corotational rates, we consider for an arbitrary corotational rate $\frac{\DD^{\circ}}{\DD t}$ the expression
	\begin{align}
	\label{eqisthatcorrect1}
	\frac{\DD^{\circ}}{\DD t}[B] = \frac{\DD}{\DD t}[B] + B \, \Omega^{\circ} - \Omega^{\circ} \, B = \mathbb{A}^{\circ}(B).D
	\end{align}
with $\mathbb{A}^{\circ}(B): \Sym(3) \to \Sym(3)$ and we need to understand the structure of the induced stiffness tensor $\mathbb{A}^{\circ}(B)$. Concerning this problem, we have from \cite{xiao98_1}:
\subsection{Representation for material spins and first properties of the tensor $\mathbb{A}^{\circ}(B)$}
\begin{thm}
Let the spin tensor $\Omega^{\circ}$ of an objective, corotational rate be associated with the deformation and rotation of a deforming material body as indicated by 
	\begin{align}
	\Omega^{\circ} = \Upsilon(B, D, W)
	\end{align}
and moreover let the tensor function $\Upsilon$ be continuous with respect to the argument $B$. Then the corotational rate of any Eulerian strain measure $e$ defined by $\Omega^{\circ}$ is objective if and only if
	\begin{align}
	\label{eqxiao1}
	\Omega^{\circ} = W + \widetilde \Upsilon(B,D),
	\end{align}
where $\widetilde \Upsilon$ is an isotropic skewsymmetric tensor-valued function of $B$ and $D$ that is continuous with respect to the argument $B$.
\end{thm}
\begin{proof}
See proof of Theorem 4 in \cite[p.~22]{xiao98_1}.
\end{proof}
\begin{rem}
The objectivity is easy to see. Since $W \mapsto \dot Q \, Q^T + Q \, W \, Q^T$ and $\widetilde \Upsilon(B,D) \mapsto Q \, \widetilde \Upsilon(B,D) \, Q^T$, we have altogether $\Omega^{\circ} \mapsto \dot Q \, Q^T + Q \, \Omega^{\circ} \, Q^T$.
\end{rem}
\begin{definition}[\textbf{Material spins}] \label{appmaterialspins}
Xiao et al.~\cite[p.25]{xiao98_1} continue to define a physically reasonably large subclass of spin tensors for \textbf{objective corotational rates} which includes all known corotational rates, the so-called \textbf{material spins}, by requiring $\widetilde \Upsilon(B,D)$ to be \textbf{linear in $D$} and \textbf{isotropic in $B$ and $D$} as well as satisfying the \textbf{homogeneity condition}\footnote
{
Motivated by the fact that the spins $W$ and $\Omega^R$ do not change under a uniform dilation $F \mapsto \alpha \, F$.
}
$\widetilde \Upsilon(\alpha \, B, D) = \widetilde \Upsilon(B, D)$ for all $\alpha > 0$. Standard representation theorems for isotropic, skew-symmetric tensor functions then yield the format (see also \cite{ghavam2007, korobeynikov2011})
	\begin{align}
	\label{eqsecondhalf}
	\Omega^{\circ} = W + \widetilde \Upsilon(B,D) = W + \nu_1 \, \sk(B \, D) + \nu_2 \, \sk(B^2 \, D) + \nu_3 \, \sk(B^2 \, D \, B),
	\end{align}
where each coefficient $\nu_k$ is an isotropic invariant of $B$.
\end{definition}
\noindent Recalling the identity $\frac{\DD}{\DD t}[B] = L \, B + B \, L^T = D \, B + B \, D + W \, B - B \, W$ and reconsidering \eqref{eqisthatcorrect1} then gives
	\begin{equation}
	\label{eqAAb}
	\begin{alignedat}{2}
	\frac{\DD^{\circ}}{\DD t}[B] &= \frac{\DD}{\DD t}[B] + B \, \Omega^{\circ} - \Omega^{\circ} \, B \\
	\overset{\eqref{eqsecondhalf}}&{=}\frac{\DD}{\DD t}[B] + B \, (W + \nu_1 \, \sk(B \, D) + \nu_2 \, \sk(B^2 \, D) + \nu_3 \, \sk(B^2 \, D \, B)) \\
	& \qquad \qquad \quad \! - (W + \nu_1 \, \sk(B \, D) + \nu_2 \, \sk(B^2 \, D) + \nu_3 \, \sk(B^2 \, D \, B)) \, B \\
	&= \frac{\DD}{\DD t}[B] + B \, W - W \, B + \nu_1 \, (B \, \sk(B \, D) - \sk(B \, D) \, B) \\
	& \qquad \qquad \!+ \nu_2 \, (B \, \sk(B^2 \, D) - \sk(B^2 \, D) \, B) + \nu_3 \, (B^2 \, \sk(B \, D \, B) - \sk(B^2 \, D \, B) \, B) \\
	&= D \, B + B \, D + \nu_1 \, (B \, \sk(B \, D) - \sk(B \, D) \, B) \\
	&\qquad +\nu_2 \, (B \, \sk(B^2 \, D) - \sk(B^2 \, D) \, B) + \nu_3 \, (B \, \sk(B^2 \, D \, B) - \sk(B^2 \, D \, B) \, B) \\
	&=: \mathbb{A}^{\circ}(B).D \, .
	\end{alignedat}
	\end{equation}
Thus, in principle, we have an explicit representation for $\mathbb{A}^{\circ}(B)$ at our disposal. \\
\\
Observe, that for $B, D \in \Sym(3)$, we have the identities
	\begin{equation}
	\label{eqA230}
	\begin{alignedat}{2}
	\sk(B \, D) &= \frac12 \, (B \, D - D \, B) = \frac12 \, [B,D], \\
	\sk(B^2 \, D) &= \frac12 \, (B^2 \, D - D \, B^2) = \frac12 [B \, (B \, D - D \, B) + (B \, D - D \, B) \, B] = \frac12 [B \, [B,D] + [B,D] \, B], \\
	\sk(B^2 \, D \, B) &= \frac12 \, (B^2 \, D \, B - B \, D \, B^2) = \frac12 \, B \, (B \, D - D \, B) \, B = \frac12 \, B \, [B,D] \, B,
	\end{alignedat}
	\end{equation}
so that we may reformulate \eqref{eqAAb} in terms of the commutator bracket $[B,D]$ as
	\begin{equation}
	\label{eqrepresA}
	\begin{alignedat}{2}
	\mathbb{A}^{\circ}(B).D = D \, B + B \, D &+ \frac{\nu_1}{2} \, (B \, [B,D] - [B,D] \, B) \\
	&+ \frac{\nu_2}{2} \, (B^2 \, [B,D] - [B,D] \, B^2) + \frac{\nu_3}{2} \, (B^2 \, [B,D] \, B - B \, [B,D] \, B^2) \, .
	\end{alignedat}
	\end{equation}
For example, the choice $\nu_1 = \nu_2 = \nu_3 = 0$ leads back to the known Zaremba-Jaumann derivative
	\begin{align}
	\frac{\DD^{\ZJ}}{\DD t}[B] = B \, D + D \, B \, .
	\end{align}
Since
	\begin{align}
	\label{eqminorsymA}
	(B \, A - A \, B)^T = A^T \, B^T - B^T \, A^T = -A \, B + B \, A \qquad \text{for} \quad B \in \Sym^{++}(3), \; A \in \mathfrak{so}(3)
	\end{align}
this shows that $\mathbb{A}^{\circ}(B)$ is a fourth order linear operator with $\mathbb{A}^{\circ}(B): \Sym(3) \to \Sym(3)$ (minor symmetry).

By application of the chain rule (cf.~Proposition \ref{apppropa20}), we obtain a semi-explicit representation for the induced tangent stiffness tensor for the class of corotational rates with material spins \eqref{eqsecondhalf}
	\begin{align}
	\label{eqapp227}
	\H^{\circ}(\sigma).D = \frac{\DD^{\circ}}{\DD t}[\sigma] = \DD_{\log B} \widehat \sigma(\log B). \DD_B \log B. \mathbb{A}^{\circ}(B).D = \DD_B \sigma(B).\mathbb{A}^{\circ}(B).D
	\end{align}
with $\mathbb{A}^{\circ}(B)$ given in \eqref{eqAAb}. This identity extends the result in Norris \cite[Lemma 1]{Norris2008} from primary matrix functions to arbitrary isotropic tensor functions $B \mapsto \sigma(B)$.\footnote
{
Norris' result, \cite[Lemma 1]{Norris2008} is expressed as $\frac{\DD^{\circ}}{\DD t}[\widetilde \sigma(V)] = \DD_V \widetilde \sigma(V). \frac{\DD^{\circ}}{\DD t}[V] = \DD_V \widetilde \sigma(V).\mathbb{Q}(V).D$ for a primary matrix function $\widetilde \sigma$. He is also requiring \cite[Sec.~5]{Norris2008} that $\mathbb{Q}(V) \cong \mathbb{A}^{\circ}(B)$ is invertible.
}
\noindent Additionally, the tensor $\mathbb{A}^{\circ}(B)$ is \textbf{major symmetric} as shown in the next
\begin{prop}[Major symmetry of $\mathbb{A}^{\circ}(B)$] \label{propmajorsymA}
The fourth order tensor $\mathbb{A}^{\circ}(B)$ given by
	\begin{equation}
	\begin{alignedat}{2}
	\mathbb{A}^{\circ}(B).D = D \, B + B \, D &+ \frac{\nu_1}{2} \, (B \, [B,D] - [B,D] \, B) \\
	&+ \frac{\nu_2}{2} \, (B^2 \, [B,D] - [B,D] \, B^2) + \frac{\nu_3}{2} \, (B^2 \, [B,D] \, B - B \, [B,D] \, B^2)
	\end{alignedat}
	\end{equation}
is major symmetric (self-adjoint).
\end{prop}
\begin{proof}
We show major symmetry of the four parts of $\mathbb{A}^{\circ}(B)$. Obviously, we have
	\begin{align}
	\langle D_1 \, B + B \, D_1, D_2 \rangle = \langle D_1, D_2 \, B \rangle + \langle D_1, B \, D_2 \rangle = \langle D_2 \, B + B \, D_2, D_1 \rangle.
	\end{align}
For the $\nu_1$-part it follows
	\begin{align}
	\label{eqnu1}
	\langle B \, [B,D_1] - [B, D_1] \, B, D_2 \rangle = \langle [B, D_1], [B, D_2] \rangle = \langle [B, D_2], [B, D_1] \rangle = \langle B \, [B, D_2] - [B, D_2] \, B, D_1 \rangle
	\end{align}
and for the $\nu_2$-part we have
	\begin{equation}
	\begin{alignedat}{2}
	\langle [B^2, [B,D_1]], D_2 \rangle &= \langle B^2 \, [B, D_1] - [B, D_1] \, B^2, D_2 \rangle = \langle [B, D_1], [B^2, D_2] \rangle \\
	&= \langle D_1, [B,[B^2, D_2]] \rangle \overset{!}{=} \langle [B^2, [B, D_2]], D_1 \rangle \, ,
	\end{alignedat}
	\end{equation}
so that we need to have $[B,[B^2,D]] = [B^2,[B,D]]$. However, this is true, since
	\begin{align}
	[B^2,[B,D]] = B^3 \, D - B^2 \, D \, B - B \, D \, B^2 + D \, B^3 = [B,[B^2,D]] \, .
	\end{align}
What remains is the $\nu_3$-part given by
	\begin{align}
	\langle B^2 \, [B, D_1] \, B - B \, [B,D_1] \, B^2, D_2 \rangle = \langle B \, [B, D_1] - [B, D_1] \, B, \underbrace{B \, D_2 \, B}_{=: \, Q} \rangle = \langle B \, [B, D_1] - [B, D_1] \, B, Q \rangle \, ,
	\end{align}
but at this point we can use \eqref{eqnu1} to obtain
	\begin{align}
	\langle B \, [B, D_1] - [B, D_1] \, B, Q \rangle = \langle B \, [B, Q] - [B, Q] \, B, D_1 \rangle = \langle B \, [B, B \, D_2 \, B] - [B, B \, D_2 \, B] \, B, D_1 \rangle
	\end{align}
as well as $[B, B \, D \, B] = B^2 \, D \, B - B \, D \, B^2 = B \, [B,D] \, B$, yielding
	\begin{align}
	\langle B \, [B, B \, D_2 \, B] - [B, B \, D_2 \, B] \, B, D_1 \rangle = \langle B^2 \, [B, D_2] \, B - B \, [B, D_2] \, B^2, D_1 \rangle \, ,
	\end{align}
which completes the proof.
\end{proof}
\subsection{Positive definiteness of $\mathbb{A}^{\circ}(B)$}
For $\frac{\DD^{\ZJ}}{\DD t}[B]$ and $\frac{\DD^{\GN}}{\DD t}[B]$ we have, respectively
	\begin{align}
	\label{eqBforZJGN}
	\frac{\DD^{\ZJ}}{\DD t}[B] = B \, D + D \, B =: \mathbb{A}^{\ZJ}(B).D \qquad \text{and} \qquad \frac{\DD^{\GN}}{\DD t}[B] = 2 \, V \, D \, V = 2 \, \sqrt{B} \, D \, \sqrt{B} =: \mathbb{A}^{\GN}(B).D \, ,
	\end{align}
so that in both cases $\mathbb{A}^{\circ}(B)$ is invertible and indeed positive definite since ($\lambda_{\min}(B) = \lambda_{\min}^2(V)$)
	\begin{align}
	\langle B \, D + D \, B, D \rangle = 2 \, \langle B \, D, D \rangle \ge 2 \, \lambda_{\min}(B) \, \norm{D}^2 \qquad \text{and} \qquad \langle 2 \, V \, D \, V, D \rangle \ge 2 \, \lambda^2_{\min}(V) \, \norm{D}^2 \, .
	\end{align}
Furthermore, for $\frac{\DD^{\log}}{\DD t}[B] = \mathbb{A}^{\log}(B).D$ we have
	\begin{equation}
	\begin{alignedat}{2}
	\DD_{\log B} \widehat \sigma(\log B). \DD_B \log B. \mathbb{A}^{\log}(B).D \overset{\eqref{eqapp227}}&{=} \DD_{\log B} \widehat \sigma(\log B). \frac{\DD^{\log}}{\DD t}[\log B] = 2 \, \DD_{\log B} \widehat \sigma(\log B).D \\
	 \overset{\widehat \sigma(\log B) = \log B}{\iff} \quad \DD_B \log B. \mathbb{A}^{\log}(B) &= 2 \, \id \quad \iff \quad \mathbb{A}^{\log}(B) := 2 \, [\DD_B \log B]^{-1} \in \Sym^{++}_4(6),	
	\end{alignedat}
	\end{equation}
so that $\mathbb{A}^{\log}$ is also positive definite (cf.~\cite[Appendix]{CSP2024}). These examples underline again that it is physically reasonable to require that $\mathbb{A}^{\circ}(B)$ should always be a positive definite fourth order tensor for any suitable corotational rate with material spin \eqref{eqsecondhalf}: we need to require for any corotational rate with material spin in addition
	\begin{align}
	\langle \frac{\DD^{\circ}}{\DD t}[B], D \rangle = \langle \mathbb{A}^{\circ}(B).D, D \rangle > 0 \qquad \forall \, D \in \Sym(3) \! \setminus \! \{0\} \qquad \iff \qquad \mathbb{A}^{\circ}(B) \in \Sym^{++}_4(6).
	\end{align}
Checking the possible positive definiteness of $\mathbb{A}^{\circ}(B)$ leads to the evaluation of the expression $\langle \mathbb{A}^{\circ}(B).D, D \rangle$, consisting of the components
	\begin{equation}
	\label{eqA238}
	\begin{alignedat}{2}
	\langle D \, B + B \, D , D \rangle &= 2 \, \langle B \, D, D \rangle \ge 2 \, \lambda_{\min}(B) \, \norm{D}^2, \\
	\langle B \, \sk(B \, D) - \sk(B \, D) \, B, D \rangle &= \langle \sk(B \, D), B \, D - D \, B \rangle \\
	&= \frac12 \, \langle B \, D - D \, B, B \, D - D \, B \rangle = \frac12 \, \norm{[B,D]}^2 \ge 0, \\
	\langle B \, \sk(B^2 \, D) - \sk(B^2 \, D) \, B, D \rangle &= \langle \sk(B^2 \, D), B \, D - D \, B \rangle \\
	&= \frac12 \, \langle B^2 \, D - D \, B^2, B \, D - D \, B \rangle \\
	&= \langle B^2 \, D, B \, D - D \, B \rangle = \langle B^2 \, D, [B, D] \rangle, \\
	\langle B \, \sk(B^2 \, D \, B) - \sk(B^2 \, D \, B) \, B, D \rangle &= \langle \sk(B^2 \, D \, B), B \, D - D \, B \rangle \\
	&= \frac12 \, \langle B^2 \, D \, B - B \, D \, B^2, B \, D - D \, B \rangle \\
	&= \langle B^2 \, D \, B, B \, D - D \, B \rangle = \langle B^2 \, D \, B, [B, D] \rangle \, .
	\end{alignedat}
	\end{equation}
Combining these components yields preliminary
	\begin{align}
	\label{eqposdefAA}
	\langle \mathbb{A}^{\circ}(B).D, D \rangle = 2 \, \langle B \, D, D \rangle + \frac{\nu_1}{2} \, \norm{[B,D]}^2 + \nu_2 \, \langle B^2 \, D, [B,D] \rangle + \nu_3 \, \langle B^2 \, D \, B, [B,D] \rangle \, ,
	\end{align}
which e.g.~becomes positive, if $\nu_1 \ge 0$ and $\nu_2 = \nu_3 = 0$. \\
\\
For the following investigation, due to isotropy, we can assume that $B = \diag(\lambda_1^2, \lambda_2^2, \lambda_3^2) = b \in \R^{3 \times 3}$ is diagonal with $\lambda_1, \lambda_2, \lambda_3 > 0$. This is true since for $B = Q^T \, b \, Q$ we have
	\begin{equation}
	\label{eqappAA1}
	\begin{alignedat}{2}
	\langle B^2 \, D, [B,D] \rangle &= \langle Q^T \, b^2 \, Q \, D, Q^T \, b \, Q \, D - D \, Q^T \, b \, Q \rangle \\
	&= \langle b^2 \, \underbrace{Q \, D \, Q^T}_{\overline D}, b \, \underbrace{Q \, D \, Q^T}_{\overline D} - \underbrace{Q \, D \, Q^T}_{\overline D} \, b \rangle = \langle b^2 \, \overline{D}, [b, \overline{D}] \rangle \, ,
	\end{alignedat}
	\end{equation}
as well as
	\begin{equation}
	\label{eqappAA2}
	\begin{alignedat}{2}
	\langle B^2 \, D \, B, [B,D] \rangle &= \langle Q^T \, b^2 \, Q \, D \, Q^T \, b \, Q, Q^T \, b \, Q \, D - D \, Q^T \, b \, Q \rangle \\
	&= \langle b^2 \, \underbrace{Q \, D \, Q^T}_{\overline D} b, b \, \underbrace{Q \, D \, Q^T}_{\overline D} - \underbrace{Q \, D \, Q^T}_{\overline D} \, b \rangle = \langle b^2 \, \overline{D} \, b, [b, \overline{D}] \rangle \, .
	\end{alignedat}
	\end{equation}
We proceed below by showing that both expressions \eqref{eqappAA1} and \eqref{eqappAA2} are non-negative for all $b = \diag(\lambda_1^2, \lambda_2^2, \lambda_3^2)$ and $\overline{D} \in \Sym(3)$. Setting $(\overline{D})_{ij} = d_{ij}$, $i,j = 1,2,3$ 
we obtain
	\begin{equation}
	\begin{alignedat}{2}
	\langle b^2 \, \overline{D}, [b, \overline{D}] \rangle
	&= d_{12}^2 \left(\lambda _1^2+\lambda _2^2\right) \left(\lambda _1^2-\lambda _2^2\right)^2 + d_{13}^2 \left(\lambda _1^2-\lambda _3^2\right)^2 \left(\lambda _1^2+\lambda _3^2\right) +d_{23}^2 \left(\lambda _2^2-\lambda _3^2\right)^2 \left(\lambda _2^2+\lambda _3^2\right) \geq 0, \\
	\langle b^2 \, \overline{D} \, b, [b, \overline{D}] \rangle
	&= d_{12}^2 \lambda _1^2 \lambda _2^2 \left(\lambda _1^2-\lambda _2^2\right)^2 + \lambda _3^2 \left(d_{23}^2 \lambda _2^2 \left(\lambda _2^2-\lambda _3^2\right)^2 + d_{13}^2 \left(\lambda _1^3-\lambda _1 \lambda _3^2\right)^2\right) \geq 0.
	\end{alignedat}
	\end{equation}
Since by $\eqref{eqA238}_1$ we also have $2 \, \langle B \, D, D \rangle \ge 2 \, \lambda_{\min}(B) \, \norm{D}^2$, it follows
\begin{lem} \label{lemA.37}
Let $\nu_1, \nu_2, \nu_3 \ge 0$. Then $\mathbb{A}^{\circ}(B)$ given by
	\begin{equation}
	\label{eqdisplayA}
	\begin{alignedat}{2}
	\mathbb{A}^{\circ}(B).D = D \, B + B \, D &+ \frac{\nu_1}{2} \, (B \, [B,D] - [B,D] \, B) \\
	&+ \frac{\nu_2}{2} \, (B^2 \, [B,D] - [B,D] \, B^2) + \frac{\nu_3}{2} \, (B^2 \, [B,D] \, B - B \, [B,D] \, B^2)
	\end{alignedat}
	\end{equation}
is positive definite, i.e.
	\begin{equation}
	\begin{alignedat}{2}
	\langle \mathbb{A}^{\circ}(B).D, D \rangle = 2 \, \langle B \, D, D \rangle &+ \frac{\nu_1}{2} \, \norm{[B,D]}^2 + \nu_2 \, \langle B^2 \, D, [B,D] \rangle \\
	&+ \nu_3 \, \langle B^2 \, D \, B, [B,D] \rangle > 0 \qquad \forall \, D \in \Sym(3) \setminus \{0\} \\
	&\implies \qquad \mathbb{A}^{\circ}(B) \in \Sym^{++}_4(6) \, .
	\end{alignedat}
	\end{equation}
\end{lem}
\begin{rem}
This property of the representation formula \eqref{eqAAb} seems to have gone unnoticed hitherto.
\end{rem}
\subsection{Conditions for the invertibility of $\mathbb{A}^{\circ}(B)$} \label{appendixAA}
We would also like to understand for which range of weight functions $\nu_1, \nu_2, \nu_3$ and $B \in \Sym^{++}(3)$ we can expect invertibility of $\mathbb{A}^{\circ}(B)$ given in \eqref{eqAAb}. This is a pertinent minimal requirement for any reasonable corotational rate, since by definition
	\begin{align}
	\frac{\DD^{\circ}}{\DD t}[B] = \mathbb{A}^{\circ}(B).D.
	\end{align}
If $\mathbb{A}^{\circ}(B)$ were not invertible for some $B \in \Sym^{++}(3)$, there would exist a non-zero stretch rate $D \in \Sym(3) \! \setminus \! \{0\}$ for which the corotational derivative $\frac{\DD^{\circ}}{\DD t}[B] = 0$, which is physically absurd. \\
\\
\noindent It is clear that the quadratic function $D \mapsto \langle \mathbb{A}^{\circ}(B).D, D \rangle$ is convex in $\Sym(3) \cong \mathbb{R}^6$ if and only if the matrix $\widetilde{\mathbb{A}}^\circ(B) \in \Sym^{++}(6)$ defined by
	\begin{align}
	\langle \widetilde{\mathbb{A}}^{\circ}(B).
		\begin{pmatrix}
		d_{11}\\d_{22}\\d_{33}\\d_{12}\\d_{13}\\d_{23}
		\end{pmatrix},
		\begin{pmatrix}
		d_{11}\\d_{22}\\d_{33}\\d_{12}\\d_{13}\\d_{23}
		\end{pmatrix}\rangle_{\R^6}
	:= \left\{
		\begin{array}{l}
		\langle D \, B+B \, D,D\rangle \\
		\qquad \quad +\nu_1\, \langle B\, \sk(B \, D)-\sk(B \, D) \, B ,D\rangle \\
		\qquad \quad +\nu_2\, \langle B\, \sk(B^2 \, D)-\sk(B^2 \, D) \, B ,D\rangle \\
		\qquad \quad + \nu_3\, \langle B\, \sk(B^2 \, D \, B)-\sk(B^2 \, D \, B) \, B ,D\rangle
		\end{array}
	\right.
	\end{align}
is positive semi-definite and it is strictly convex if and only if the matrix $\widetilde{\mathbb{A}}^{\circ}(B)$ is positive definite. One can show that the matrix $\widetilde{\mathbb{A}}^{\circ}(B)$ has the singular values
	\begin{equation}
	\begin{alignedat}{2}
	a_{11} &= 2 \, \lambda _1^2 , \qquad \qquad a_{22} = 2 \, \lambda _2^2, \qquad \qquad a_{33} = 2 \, \lambda _3^2,\\
	a_{44} &= \frac{1}{2} \left\{ 4 \left(\lambda _1^2+\lambda _2^2\right) + 2 \, \nu_1 \left(\lambda _1^2-\lambda _2^2\right)^2 + \nu_2 \left(\lambda _1^2+\lambda _2^2\right) \left(\lambda _1^2-\lambda _2^2\right)^2 + \nu_3 \, \lambda _1^2 \, \lambda _2^2 \left(\lambda _1^2-\lambda _2^2\right)^2 \right\} \\ 
	&= 2 \, (\lambda_1^2 + \lambda_2^2) + \frac12 \, (\lambda_1^2 - \lambda_2^2)^2 \, \bigg\{2 \, \nu_1 + \nu_2 \, (\lambda_1^2 + \lambda_2^2) + \nu_3 \, \lambda_1^2 \, \lambda_2^2 \bigg\} \, , \\
	a_{55} &= \frac{1}{2} \left\{ 4 \left(\lambda _1^2+\lambda _3^2\right) + 2 \, \nu_1 \left(\lambda _1^2-\lambda _3^2\right)^2 + \nu_2 \left(\lambda _1^2+\lambda _3^2\right) \left(\lambda _1^2-\lambda _3^2\right)^2 + \nu_3 \, \lambda _1^2 \, \lambda _3^2 \left(\lambda _1^2-\lambda _3^2\right)^2 \right\} \\
	&= 2 \, (\lambda_1^2 + \lambda_3^2) + \frac12 \, (\lambda_1^2 - \lambda_3^2)^2 \, \bigg\{2 \, \nu_1 + \nu_2 \, (\lambda_1^2 + \lambda_3^2) + \nu_3 \, \lambda_1^2 \, \lambda_3^2 \bigg\} \, , \\
	a_{66} &= \frac{1}{2} \left\{ 4 \left(\lambda _2^2+\lambda _3^2\right) + 2 \, \nu_1 \left(\lambda _2^2-\lambda _3^2\right)^2 + \nu_2 \left(\lambda _2^2+\lambda _3^2\right) \left(\lambda _2^2-\lambda _3^2\right)^2 + \nu_3 \, \lambda _2^2 \, \lambda _3^2 \left(\lambda _2^2-\lambda _3^2\right)^2 \right\} \\
	&= 2 \, (\lambda_2^2 + \lambda_3^2) + \frac12 \, (\lambda_2^2 - \lambda_3^2)^2 \, \bigg\{2 \, \nu_1 + \nu_2 \, (\lambda_2^2 + \lambda_3^2) + \nu_3 \, \lambda_2^2 \, \lambda_3^2 \bigg\} \, .
	\end{alignedat}
	\end{equation}
Therefore, if $\nu_1, \nu_2, \nu_3 \ge 0$ and $\lambda_i > 0$, the matrix $\widetilde{\mathbb{A}}^{\circ}(B)$ is positive definite. On the other hand, the matrix $\widetilde{\mathbb{A}}^{\circ}(B)$ (the operator $\mathbb{A}^{\circ}(B)$) is invertible if $\nu_1, \nu_2, \nu_3$ are such that $a_{44}, a_{55}, a_{66} \neq 0$ for all $\lambda_1, \lambda_2, \lambda_3 > 0$, leading to the invertibility statement:
\begin{lem}
The fourth order tensor $\mathbb{A}^{\circ}(B)$ is invertible if and only if for all $\lambda_1, \lambda_2, \lambda_3 > 0$ we have
	\begin{equation}
	\begin{alignedat}{2}
	2 \, (\lambda_1^2 + \lambda_2^2) + \frac12 \, (\lambda_1^2 - \lambda_2^2)^2 \, \bigg\{2 \, \nu_1 + \nu_2 \, (\lambda_1^2 + \lambda_2^2) + \nu_3 \, \lambda_1^2 \, \lambda_2^2 \bigg\} &\neq 0, \\
	2 \, (\lambda_1^2 + \lambda_3^2) + \frac12 \, (\lambda_1^2 - \lambda_3^2)^2 \, \bigg\{2 \, \nu_1 + \nu_2 \, (\lambda_1^2 + \lambda_3^2) + \nu_3 \, \lambda_1^2 \, \lambda_3^2 \bigg\} &\neq 0, \\
	2 \, (\lambda_2^2 + \lambda_3^2) + \frac12 \, (\lambda_2^2 - \lambda_3^2)^2 \, \bigg\{2 \, \nu_1 + \nu_2 \, (\lambda_2^2 + \lambda_3^2) + \nu_3 \, \lambda_2^2 \, \lambda_3^2 \bigg\} &\neq 0.
	\end{alignedat}
	\end{equation}
\end{lem}
\noindent Regarding positive definiteness of $\mathbb{A}^{\circ}(B)$ we likewise have:
\begin{lem}
The fourth order tensor $\mathbb{A}^{\circ}(B)$ is positive definite if and only if for all $\lambda_1, \lambda_2, \lambda_3 > 0$ we have
	\begin{equation}
	\begin{alignedat}{2}
	2 \, (\lambda_1^2 + \lambda_2^2) + \frac12 \, (\lambda_1^2 - \lambda_2^2)^2 \, \bigg\{2 \, \nu_1 + \nu_2 \, (\lambda_1^2 + \lambda_2^2) + \nu_3 \, \lambda_1^2 \, \lambda_2^2 \bigg\} &> 0, \\
	2 \, (\lambda_1^2 + \lambda_3^2) + \frac12 \, (\lambda_1^2 - \lambda_3^2)^2 \, \bigg\{2 \, \nu_1 + \nu_2 \, (\lambda_1^2 + \lambda_3^2) + \nu_3 \, \lambda_1^2 \, \lambda_3^2 \bigg\} &> 0, \\
	2 \, (\lambda_2^2 + \lambda_3^2) + \frac12 \, (\lambda_2^2 - \lambda_3^2)^2 \, \bigg\{2 \, \nu_1 + \nu_2 \, (\lambda_2^2 + \lambda_3^2) + \nu_3 \, \lambda_2^2 \, \lambda_3^2 \bigg\} &> 0.
	\end{alignedat}
	\end{equation}
\end{lem}
\begin{definition}
We call any $\mathbb{A}^{\circ}(B)$ of the form \eqref{eqdisplayA} \textbf{totally positive} if $\nu_1, \nu_2, \nu_3 \ge 0$.
\end{definition}
\noindent Every totally positive corotational rate is clearly a positive corotational rate. In this sense, $\frac{\DD^{\ZJ}}{\DD t}, \frac{\DD^{\GN}}{\DD t}$ and $\frac{\DD^{\log}}{\DD t}$ qualify as positive corotational rates and $\frac{\DD^{\ZJ}}{\DD t}$ is in addition totally positive.
\begin{cor}
We know for all objective corotational rates $\frac{\DD^{\circ}}{\DD t}$
	\begin{align}
	\frac{\DD^{\circ}}{\DD t}[\sigma(B)] = \DD_B \sigma(B). \frac{\DD^{\circ}}{\DD t}[B].
	\end{align}
Consider again
	\begin{align}
	\frac{\DD^{\circ}}{\DD t}[B] = \frac{\DD}{\DD t}[B] - \Omega^{\circ} \, B + B \, \Omega^{\circ} = \mathbb{A}^{\circ}(B).D
	\end{align}
and assume that $[B, D] = B \, D - D \, B = 0$. Then $\frac{\DD^{\circ}}{\DD t}[B] = 2 \, B \, D$ coincides for all corotational rates $\frac{\DD^{\circ}}{\DD t}$ with material spins $\Omega^{\circ}$ defined in \eqref{eqsecondhalf}.
\end{cor}
\begin{proof}
Using the representation \eqref{eqrepresA} of $\mathbb{A}^{\circ}(B)$ given by
	\begin{equation}
	\begin{alignedat}{2}
	\mathbb{A}^{\circ}(B).D = D \, B + B \, D &+ \frac{\nu_1}{2} \, (B \, [B,D] - [B,D] \, B) \\
	&+ \frac{\nu_2}{2} \, (B^2 \, [B,D] - [B,D] \, B^2) + \frac{\nu_3}{2} \, (B^2 \, [B,D] \, B - B \, [B,D] \, B^2) \, ,
	\end{alignedat}
	\end{equation}
we see that for $[B,D] = 0$ it only remains
	\begin{align}
	\frac{\DD^{\circ}}{\DD t}[B] &= \mathbb{A}^{\circ}(B).D = B \, D + D \, B = 2 \, B \, D. \notag \qedhere
	\end{align}
\end{proof}
\begin{example}
If $[B,D] = 0 \quad \big(\iff [V,D] = 0\big)$, we directly verify for the most common objective corotational rates
	\begin{equation}
	\begin{alignedat}{2}
	\frac{\DD^{\ZJ}}{\DD t}[B] = B \, D + D \, B &= 2 \, B \, D, \qquad \frac{\DD^{\GN}}{\DD t}[B] = 2 \, V \, D \, V = 2 \, V^2 \, D = 2 \, B \, D, \\
	&\text{and\footnotemark} \qquad \frac{\DD^{\log}}{\DD t}[B] = 2 \, [\DD_B \log B]^{-1}.D = 2 \, B \, D.
	\end{alignedat}
	\end{equation}
\end{example}
\footnotetext
{
The last result is easily seen in the scalar case since here we have $\DD_B \log B = B^{-1}$ and thus \newline $2 \, [\DD_B \log B]^{-1} . D = 2 \, (B^{-1})^{-1} \, D = 2 \, B \, D$.
}

\section{In-depth analysis of the stiffness tensor $\mathbb{A}^{\circ}(B)$}
We would now like to continue our investigation of the stiffness tensor $\mathbb{A}^{\circ}(B)$, possibly leading to sharper statements and necessitating a slight shift of notation and methodology. Indeed, we need to use projection operators and the representation in principal axes. As a result we obtain necessary and sufficient conditions for the positive definiteness of $\mathbb{A}^{\circ}(B)$ as well as sharp statements on the invertibility of $\mathbb{A}^{\circ}(B)$.

Let ${A},{H}, X \in \R^{3 \times 3}$. We define the \textbf{symmetric tensor product} ${A}\!\overset{\text{sym}}{\otimes}\!{H}$ for second-order tensors according to the definitions given in \cite{Curnier1994,holzapfel2000,PeyrautANM2009} and set
	\begin{equation}
	\label{2}
	({A}\!\overset{\text{sym}}{\otimes}\!{B}).{X} := {A}\, (\text{sym}\,{X}) \, {B}^T.
	\end{equation}
The left stretch tensor ${V}$ can be represented in the classical spectral form
	\begin{equation}
	\label{3}
	{V}=\sum_{k=1}^{3} \lambda_k \, {n}_k\otimes{n}_k, \qquad V = Q \, \diag(\lambda_1, \lambda_2, \lambda_3) \, Q^T, \qquad Q = \big(n_1 | n_2 | n_3\big) \in \OO(3),
	\end{equation}
where $\lambda_k\in \R^+$ are the principal stretches, $\{{n}_k\}$ ($k=1,2,3$) is the triad of the corresponding subordinate right-oriented orthonormal eigenvectors (principal directions) of the tensor ${V}$, the symbol $\otimes$ denotes the dyadic product of vectors and $Q = (n_1 | n_2 | n_3)$ denotes the orthogonal matrix with columns $n_1, n_2, n_3$.

For multiple eigenvalues $\lambda_k$, the corresponding eigenvectors ${n}_k$ are defined ambiguously. This ambiguity can be circumvented by representing the tensor ${V}$ in terms of \textbf{eigenprojections} (see e.g.~\cite{Bertram2021,korobeynikov2011,KorobeynikovAM2023,LuehrCMAME1990})
	\begin{equation}
	\label{4}
	{V}=\sum^m_{i=1}\lambda_i \, {V}_i.
	\end{equation}
Here $\lambda_i$ are all different $m$ eigenvalues\footnote
{
The number $m$ ($1\leq m\leq 3$) will be called the \emph{eigenindex}.
}
of the tensor ${V}$ and ${V}_i$ ($i=1,\ldots,m$) are the subordinate \emph{eigenprojections}. We introduce the eigenprojections of the tensor ${V}$ using Sylvester's formula:
\begin{large}
	\begin{equation}
	\label{5}
	    {V}_i=
	    \begin{cases}
	   \dd \prod^m\limits_{\substack{j=1\\ i\neq j}} \frac{{V}-\lambda_j \, \id}{\lambda_i-\lambda_j}, & \text{if}\ m=2,3 \\
	    \id, & \text{if}\ m=1.
	   \end{cases}
	\end{equation}
\end{large}
The eigenprojections have the following properties \cite{LuehrCMAME1990}:
	\begin{equation}
	\label{6}
	   {V}_i \, {V}_j=
	   \begin{cases}
	    {V}_i & \text{if}\ i=j \\
	    {0} & \text{if}\ i\neq j
	   \end{cases},
	\qquad \sum^m_{i=1}{V}_i=\id, \qquad \text{tr}\,{V}_i=m_i\qquad (i,j=1,\ldots,m),
	\end{equation}
where $m_i$ denotes the multiplicity of an eigenvalue $\lambda_i$. \\
We point out the representation formula for the left (Eulerian) Cauchy-Green deformation tensor
	\begin{equation}
	\label{7}
	    B \equiv {F} \, {F}^T = {V}^2= \sum_{k=1}^{3} \lambda_k^2 \, {n}_k\otimes{n}_k = \sum^m_{i=1}\lambda_i^2 \,  {V}_i = Q \, \diag(\lambda_1^2, \lambda_2^2, \lambda_3^2) \, Q^T.
	\end{equation}
Corollary 2.1 of \cite{korobeynikov2018} implies the next Proposition:
\begin{prop}
\label{Pr:1-1}
Let ${H} \in \Sym(3)$ be an arbitrary symmetric second order tensor and let the tensor ${V}$ have the spectral representation \eqref{4}. We may represent the  tensor ${H}\in \Sym(3)$ in the following form:\footnote{Hereinafter, the notation $\sum_{i\neq j=1}^{m}$ denotes the summation over $i,j=1,\ldots, m$ and $i\neq j$ and this summation is assumed to vanish when $m=1$.}
	\begin{equation}
	\label{7-a}
	{H} = \widehat{{H}} + \widetilde{{H}},\qquad \quad \widehat{{H}} := \sum_{i=1}^{m} {V}_i\, {H}\,{V}_i,\qquad \quad
	\widetilde{{H}} := \sum_{i\neq j=1}^{m} {V}_i\, {H}\,{V}_j,
	\end{equation}
so that $\widehat{{H}},\, \widetilde{{H}}\in \Sym(3)$ are the components of the tensor ${H}$ that are coaxial and orthogonal to the tensor ${V}$.\footnote{Tensors ${X},\,{Y}\in \Sym(3)$ will be called orthogonal if the equality $\langle {X}, {Y} \rangle=0$ is satisfied.}
\end{prop}
Here, we let  $\dot{{H}}:= \frac{\DD}{\DD t} H$ denote the \emph{material time derivative}. Note, that for the stretch tensor $D$ the expression \eqref{7-a} reduces to (cf.~\cite[Eq.~$(115)_2$]{korobeynikov2011})

	\begin{equation}
	\label{9-a}
	 D =  \sum_{k=1}^{m} \frac{\dot{\lambda}_k}{\lambda_k}{V}_k + \sum_{k\neq l=1}^{m} {V}_k\,  D\,{V}_l = \sum_{k=1}^{m} V_k \, D \, {V}_k + \sum_{k\neq l=1}^{m} {V}_k\,  D\,{V}_l =
		\begin{pmatrix}
		d_{11} & 0 & 0 \\
		0 & d_{22} & 0 \\
		0 & 0 & d_{33}
		\end{pmatrix}
		+
		\begin{pmatrix}
		0 & d_{12} & d_{13} \\
		d_{12} & 0 & d_{23} \\
		d_{13} & d_{23} & 0
		\end{pmatrix} \, ,
	\end{equation}
however the components $d_{ij} = \langle n_i, D .n_j \rangle$ ($i,j=1,2,3$) are defined ambiguously when $m = 1,2$.
\subsection{Family of material spin tensors and associated corotational rates}
Let ${h}\in \R^{3 \times 3}$ be any second-order Eulerian tensor. Consider any Eulerian corotational derivative of the form
	\begin{equation}
	\label{10}
	  \frac{\DD^{\circ}}{\DD t}[h] \equiv \frac{\DD}{\DD t}[h] + {h} \, \Omega^{\circ} - \Omega^{\circ} \, {h},
	\end{equation}
where the tensor $\Omega^{\circ} \in \mathfrak{so}(3)$ is associated with the corotational stress rate $\frac{\DD^{\circ}}{\DD t}[h]$ and belongs to the material spin tensor family derived by Xiao et al. (cf., \cite{xiao98_2,xiao98_1,XiaoIJSS1998}) of the following form
	\begin{equation}
	\label{11}
	    \Omega^{\circ} := W + \widetilde \Upsilon({V}, D) = W + \widetilde \Upsilon(B,D) \qquad \text{by abuse of notation (cf.~\ref{eqxiao1}).}
	\end{equation}
Hereinafter, $\widetilde \Upsilon \in \mathfrak{so}(3)$ is an isotropic tensor function of its arguments which is linear in $ D$ and has the form (cf.~Xiao et al.~\cite{xiao98_2,xiao98_1,XiaoIJSS1998})
	\begin{equation}
	\label{12}
	  \widetilde \Upsilon({V}, D) = \sum_{i\neq j=1}^{m}g_{ij} \, {V}_i \,  D \, {V}_j,\quad g_{ij} = g(\lambda_i,\lambda_j),\qquad g_{ji}=-g_{ij}, \qquad g(\alpha \, \lambda_i, \alpha \, \lambda_j) = g(\lambda_i, \lambda_j) \quad \forall \, \alpha>0.
	\end{equation}
The specific forms of the function $g(\lambda_i,\lambda_j)$ for material spin tensors are considered in \cite{xiao98_2,xiao98_1,XiaoIJSS1998}. In \cite{korobeynikov2011}, additional continuity restrictions\footnote
{
As described in Theorem 2.2 of \cite{korobeynikov2011}, continuity of the scalar function $g(\lambda_i,\lambda_j)$ is a sufficient condition for the continuity of the tensorial function $\widetilde \Upsilon({V}, D)$ with respect to the argument $V$.
}
are imposed on the forms of the functions $g(\lambda_i,\lambda_j)$ that allow these spin tensors to be meaningfully used in applications, i.e.~it is required that $g\colon \R^+ \times \R^+$ is continuous.

Below we give the expressions for the quantities $g_{ij}$ generating spin tensors from the family of material spin tensors of the form \eqref{11} and the identification of the classical corotational rates associated with them:
\begin{enumerate}
  \item \emph{The Zaremba--Jaumann rate} $\frac{\DD^{\ZJ}}{\DD t}$ associated with the vorticity tensor $ W$\\
\begin{equation}\label{13}
  g_{ij}^{\ZJ}=0.
\end{equation}
  \item \emph{The Green--Naghdi rate} $\frac{\DD^{\GN}}{\DD t}$ associated with the polar spin tensor $\Omega^R$\\
\begin{equation}\label{14}
  g_{ij}^{\GN}=\frac{\lambda_j - \lambda_i}{\lambda_i + \lambda_j}, \qquad \text{$g_{ij}^{\GN}$ is continuous for $(\lambda_i, \lambda_j) \in \R^+ \times \R^+$}.
\end{equation}
  \item \emph{The logarithmic rate} $\frac{\DD^{\log}}{\DD t}$ associated with the spin tensor\footnote
{
Note the validity of the equation $ \frac{\DD^{\log}}{\DD t}[\log V] = D$ (cf.~\cite{xiao97}).
}
$\Omega^{\log}$ \\
\begin{equation}\label{15}
    g_{ij}^{\log}\equiv \frac{\lambda_i^2 + \lambda_j^2}{\lambda_j^2 - \lambda_i^2}+\frac{1}{\log\lambda_i-\log\lambda_j}, \qquad \text{$g_{ij}^{\log}$ is continuous\footnotemark \
    for $(\lambda_i, \lambda_j) \in \R^+ \times \R^+$}.
\end{equation}
	\footnotetext{$g_{ij}^{\log}$ is continuous despite appearance.}
  \item \emph{The Gurtin--Spear rate} $\frac{\DD^{\GS}}{\DD t}$ associated with the twirl tensor of the Eulerian triad $\Omega^E$ \cite{gurtin1983relationship} i.e. \break $\Omega^{\GS} = \Omega^E = \dot Q^E \, (Q^E)^T$ from $V = Q^E \, \diag(V) \, (Q^E)^T$, where $Q^E$ is the rotation connected to the Eulerian triad.

\begin{equation}\label{16}
    g_{ij}^{\GS}\equiv \frac{\lambda_i^2 + \lambda_j^2}{\lambda_j^2 - \lambda_i^2}, \qquad \text{$g_{ij}^{\GS}$ is \textbf{not} continuous for $\lambda_i = \lambda_j > 0$}.
\end{equation}
\end{enumerate}

Note that the spin tensors $W$, $\Omega^R$, and $\Omega^{\log}$ belong to the family of continuous material spin tensors, while the spin $\Omega^{\GS}$ belongs to the class of material spins, but it is not continuous, i.e.~$\widetilde \Upsilon(V,D)$ is not continuous with respect to $V$ (cf.~\cite{korobeynikov2011}). \\
\\
As an additional non-standard example we consider the Aifantis rate \cite[eq.~(2.17)]{Zbib1988} with the spin tensor
	\begin{align}
	\Omega^{\Aif}(\sigma) := W + \zeta \, (\sigma^{\iso} \, D - D \, \sigma^{\iso}) \, ,
	\end{align}
where we restrict our attention to isochoric Cauchy stresses of the type $\sigma^{\iso}$ with the property $\sigma^{\iso}(\alpha \, B) = \sigma^{\iso}(B)$ in order to obtain a corotational rate with material spin according to Definition \ref{appmaterialspins} (homogeneity). This gives rise to the following two examples
\begin{example} \label{exampleA1}
Consider the Cauchy-elastic law
$\sigma^{\iso}(B) = \frac{B}{(\det B)^{\frac13}} - \id$. Rewriting the Aifantis spin $\Omega^{\Aif}$ for this choice of $\sigma^{\iso}(B)$ yields the spin tensor
	\begin{equation}
	\label{eqfirstaif}
	\begin{alignedat}{2}
	\Omega^{\Aif,1} &= W + \zeta \, \left(\left(\frac{B}{(\det B)^{\frac13}}-\id\right) \, D - D \, \left(\frac{B}{(\det B)^{\frac13}}-\id\right)\right) = W + \zeta \, \left(\frac{B}{(\det B)^{\frac13}} \, D - D \, \frac{B}{(\det B)^{\frac13}}\right) \\
	&= W + \frac{\zeta}{(\det B)^{\frac13}} \, [B,D] = W + \frac{2 \, \zeta}{(\det B)^{\frac13}} \, \sk(B \, D)
	\end{alignedat}
	\end{equation}
which in view of the representation formula \eqref{eqsecondhalf} is equivalent to setting $\nu_1 = \frac{2 \, \zeta}{(\det B)^{\frac13}}$ and $\nu_2 = \nu_3 = 0$ so that $\frac{\DD^{\Aif,1}}{\DD t}$ for $\sigma^{\iso}(B) = \frac{B}{(\det B)^{\frac13}} - \id$ is a totally positive rate if $\zeta \ge 0$. \\
\\
To find an explicit expression for the scalar function $g(\lambda_i, \lambda_j)$ in \eqref{12}, we first note that by \eqref{eqfirstaif}
	\begin{align}
	\label{eqkorob01}
	\widetilde \Upsilon(V,D) = \frac{\zeta}{(\det B)^{\frac13}} \, (B \, D - D \, B) \, .
	\end{align}
By the properties \eqref{6} and \eqref{7} we then obtain
	\begin{equation}
	\label{eqkorob02}
	\begin{alignedat}{2}
	B \, D - D \, B &= (\sum_{i=1}^m \lambda_i^2 \, V_i) \, (\sum_{k=1}^m \frac{\dot \lambda_k}{\lambda_k} \, V_k + \sum_{k \neq l = 1}^m V_k \, D \, V_l) - (\sum_{k=1}^m \frac{\dot \lambda_k}{\lambda_k} \, V_k + \sum_{k \neq l = 1}^m V_k \, D \, V_l) \, (\sum_{i=1}^m \lambda_i^2 \, V_i) \\
	&= \sum_{i=1}^m \lambda_i \, \dot \lambda_i \, V_i + \sum_{i \neq j = 1}^m \lambda_i^2 \, V_i \, D \, V_j - \sum_{i=1}^m \lambda_i \dot \lambda_i \, V_i - \sum_{i \neq j = 1}^m \lambda_j^2 \, V_i \, D \, V_j = \sum_{i \neq j = 1}^m (\lambda_i^2 - \lambda_j^2) \, V_i \, D \, V_j \, .
	\end{alignedat}
	\end{equation}
Substituting \eqref{eqkorob02} into \eqref{eqkorob01} and comparing the resulting expression with $\eqref{12}_1$ yields
	\begin{align}
	g_{ij}^{\Aif, 1} = \frac{\zeta}{(\det B)^{\frac13}} \, (\lambda_i^2 - \lambda_j^2), \quad \qquad \text{$g_{ij}^{\Aif,1}$ is continuous for $(\lambda_i, \lambda_j) \in \R^+ \times \R^+$} \, .
	\end{align}
\end{example}
\begin{example} \label{exampleA2}
Consider the Cauchy-elastic law $\sigma^{\iso}(B) = \frac{B}{(\det B)^{\frac13}} - \left(\frac{B}{(\det B)^{\frac13}}\right)^{-1}$. First we use the Cayley-Hamilton formula for $(3 \times 3)$-matrices
	\begin{align}
	B^3 - \tr(B) \, B^2 + \tr(\Cof B) \, B - \det B \, \id = 0
	\end{align}
and multiply it by $B^{-1}$ to obtain an alternate expression for $B^{-1}$ in terms of non-negative exponents of $B$:
	\begin{equation}
	B^2 - \tr(B) \, B + \tr(\Cof B) \, \id - \det B \, B^{-1} = 0 \quad \iff \quad B^{-1} = \frac{1}{\det B} \, (B^2 - \tr(B) \, B + \tr(\Cof B) \, \id) \, .
	\end{equation}
Rewriting the Aifantis spin $\Omega^{\Aif}$ for this choice of $\sigma^{\iso}(B)$ we obtain the spin tensor
	\begin{align}
	\Omega^{\Aif,2} &= W + \zeta \, (\sigma^{\iso} \, D - D \, \sigma^{\iso}) = W + \zeta \, \left(\left(\frac{B}{(\det B)^{\frac13}}-\left(\frac{B}{(\det B)^{\frac13}}\right)^{-1}\right) \, D - D \, \left(\frac{B}{(\det B)^{\frac13}}-\left(\frac{B}{(\det B)^{\frac13}}\right)^{-1}\right)\right) \notag \\
	&= W + \frac{\zeta}{(\det B)^{\frac13}} \, (B \, D - D \, B) + \zeta \, (\det B)^{\frac13} \, (-B^{-1} \, D + D \, B^{-1}) \notag \\
	&= W + \frac{\zeta}{(\det B)^{\frac13}} \, [B, D] + \frac{\zeta}{(\det B)^{\frac23}} \, (-B^2 \, D + \tr(B) \, B \, D - \tr(\Cof B) \, D + D \, B^2 - \tr(B) \, D \, B + \tr(\Cof B) \, D) \notag \\
	\label{eqaifantis02}
	&= W + \frac{\zeta}{(\det B)^{\frac13}} \, [B, D] + \frac{\zeta}{(\det B)^{\frac23}} \, (-[B^2,D] + \tr(B) \, [B,D]) \\
	&= W + \zeta \, \left(\frac{1}{(\det B)^{\frac13}} + \frac{\tr(B)}{(\det B)^{\frac23}}\right) \, [B, D] - \zeta \, \frac{1}{(\det B)^{\frac23}} \, [B^2,D] \notag \\
	&= W + 2 \, \zeta \, \left(\frac{1}{(\det B)^{\frac13}} + \frac{\tr(B)}{(\det B)^{\frac23}}\right) \, \sk(B \, D) - 2 \, \zeta \, \frac{1}{(\det B)^{\frac23}} \, \sk(B^2 \, D) \notag
	\end{align}
which in view of the representation formula \eqref{eqsecondhalf} is equivalent to setting
	\begin{align}
	\nu_1 = 2 \, \zeta \, \left(\frac{1}{(\det B)^{\frac13}} + \frac{\tr(B)}{(\det B)^{\frac23}}\right) > 0, \qquad \nu_2 = - 2 \, \zeta \, \frac{1}{(\det B)^{\frac23}} < 0, \qquad \nu_3 = 0, \qquad \text{for} \; \zeta > 0.
	\end{align}
\end{example}
showing that $\frac{\DD^{\Aif,2}}{\DD t}$ is not totally positive.\\
For the determination of $g_{ij}^{\Aif, 2}(\lambda_i, \lambda_j)$ we first note that by \eqref{eqaifantis02} we have
	\begin{align}
	\widetilde \Upsilon(V,D) = \zeta \, \left(\frac{1}{(\det B)^{\frac13}} \, (B \, D - D \, B) - (\det B)^{\frac13} \, (B^{-1} \, D - D \, B^{-1}) \right).
	\end{align}
In analogy to \eqref{eqkorob02} we obtain
	\begin{align}
	\label{eqkorob03}
	B^{-1} \, D - D \, B^{-1} = \sum_{i \neq j = 1}^m (\lambda_i^{-2} - \lambda_j^{-2}) \, V_i \, D \, V_j \, .
	\end{align}
Using the expressions \eqref{eqkorob02} and \eqref{eqkorob03} then leads to
	\begin{align}
	g_{ij}^{\Aif, 2} = \zeta \, \big[(\det B)^{-\frac13} \, (\lambda_i^2 - \lambda_j^2) + (\det B)^{\frac13} \, (\lambda_j^{-2} - \lambda_i^{-2})\big], \quad \text{$g_{ij}^{\Aif,2}$ is continuous for $(\lambda_i, \lambda_j) \in \R^+ \times \R^+$} \, .
	\end{align}
\begin{rem}
	Note that the scalar functions $g_{ij}^{\Aif,1}$ and $g_{ij}^{\Aif,2}$ are represented in a form (see (16) in \cite{korobeynikov2011}), which provides sufficient conditions for the continuity of the tensorial functions $\widetilde \Upsilon^{\Aif,1}(V,D)$ and $\widetilde \Upsilon^{\Aif,2}(V,D)$ with respect to $V$ (cf. Theorem 2.2 in \cite{korobeynikov2011}).
\end{rem}
\subsection{Explicit expressions for corotational rates of the left Cauchy--Green deformation tensor associated with the material spin tensors}
We have the following expression for the tensor $\frac{\DD^{\ZJ}}{\DD t}[B]$ (see $\eqref{eqBforZJGN}_1$)
	\begin{equation}
	\label{17}
	  \frac{\DD^{\ZJ}}{\DD t}[B] =B \, D + D \, B.
	\end{equation}
Then, from the definition of the material spin tensors \eqref{11} and the relations \eqref{10} and \eqref{17}, we have
	\begin{equation}
	\label{18}
	 \frac{\DD^{\circ}}{\DD t}[B] = \frac{\DD^{\ZJ}}{\DD t}[B] + B\, \widetilde \Upsilon({V}, D) - \widetilde \Upsilon({V}, D)\, B.
	\end{equation}
Using expressions $\eqref{7}_3$, $\eqref{12}_1$, and the eigenprojection properties \eqref{6}, we deduce
	\begin{equation}
	\label{19}
	\begin{alignedat}{2}
	  B\, \widetilde \Upsilon - \widetilde \Upsilon\, B &= (\sum^m_{k=1}\lambda_k^2 \,  {V}_k)\,(\sum_{i\neq j=1}^{m}g_{ij} \,  {V}_i \,  D \, {V}_j)-
	  (\sum_{i\neq j=1}^{m}g_{ij} \,  {V}_i \,  D \, {V}_j) \, (\sum^m_{k=1}\lambda_k^2 \, {V}_k) \\
	  &=\sum_{i\neq j=1}^{m}g_{ij} \, \lambda_i^2 \, {V}_i\, D \, {V}_j - \sum_{i\neq j=1}^{m}g_{ij} \, \lambda_j^2 \, {V}_i \,  D \, {V}_j=
	  \sum_{i\neq j=1}^{m}g_{ij} \, (\lambda_i^2 - \lambda_j^2) \, {V}_i\, D \, {V}_j.
	\end{alignedat}
	\end{equation}
We have from \eqref{17}, \eqref{18}, and \eqref{19}
	\begin{equation}
	\label{20}
	  \boxed{\frac{\DD^{\circ}}{\DD t}[B] = B \, D + D \, B + \sum_{i\neq j=1}^{m}g_{ij} \, (\lambda_i^2 - \lambda_j^2) \, {V}_i\, D \, {V}_j},
	\end{equation}
or, in alternative form
	\begin{equation}
	\label{21}
	 \frac{\DD^{\circ}}{\DD t}[B]= \mathbb{A}^{\circ}(B). D
	\end{equation}
where $\mathbb{A}^{\circ}(B)$ is a minor and major symmetric (cf.~\eqref{eqminorsymA} and Proposition \ref{propmajorsymA}) fourth-order stiffness tensor, alternatively written as
	\begin{equation}
	\label{22}
	  \mathbb{A}^{\circ}(B) := B\!\overset{\text{sym}}{\otimes}\!\id + \id\!\overset{\text{sym}}{\otimes}\! B +
	  \sum_{i\neq j=1}^{m} \, g_{ij} \, (\lambda_i^2 - \lambda_j^2) \, {V}_i\!\overset{\text{sym}}{\otimes}\!{V}_j.
	\end{equation}
The symmetry properties of the tensor $\mathbb{A}^{\circ}(B)$ follow also from Theorem 2.1 in \cite{korobeynikov2018}.

Next, we transform the term $B \, D + D \, B$ using the expressions \eqref{7} and \eqref{9-a} and the eigenprojection properties in \eqref{6} 
	\begin{equation}
	\label{23}
	\begin{alignedat}{2}
	  B \,  D + D \, B &= (\sum^m_{i=1}\lambda_i^2 \, {V}_i)\, (\sum_{k=1}^{m} \frac{\dot{\lambda}_k}{\lambda_k} \, {V}_k + \sum_{k\neq l=1}^{m} {V}_k\,  D\,{V}_l) + (\sum_{k=1}^{m} \frac{\dot{\lambda}_k}{\lambda_k} \, {V}_k + \sum_{k\neq l=1}^{m} {V}_k\,  D\,{V}_l)\, (\sum^m_{i=1}\lambda_i^2 \, {V}_i) \\
	  &= 2\sum^m_{i=1}\lambda_i \, \dot{\lambda}_i \, {V}_i + \sum_{k\neq l=1}^{m}\lambda_k^2 \, {V}_k\,  D\,{V}_l + \sum_{k\neq l=1}^{m}\lambda_l^2 \, {V}_k\,  D\,{V}_l \\
	&= 2\sum^m_{i=1}\lambda_i \, \dot{\lambda}_i \, {V}_i + \sum_{i\neq j=1}^{m}(\lambda_i^2 + \lambda_j^2) \, {V}_i\,  D\,{V}_j.
	\end{alignedat}
	\end{equation}
We rewrite the expression \eqref{20} using \eqref{23} and get
	\begin{equation}
	\label{24}
	 \frac{\DD^{\circ}}{\DD t}[B] = 2 \sum^m_{i=1}\lambda_i \, \dot{\lambda}_i \, {V}_i + \sum_{i\neq j=1}^{m}[\lambda_i^2 + \lambda_j^2 + g_{ij} \, (\lambda_i^2 - \lambda_j^2)] \, {V}_i\, D \, {V}_j.
	\end{equation}
\subsection{Positivity of corotational rates associated with material spin tensors}
We intend to obtain an alternative explicit form of
	\begin{equation}\label{25}
	  \langle \frac{\DD^{\circ}}{\DD t}[B], D \rangle =  \langle \mathbb{A}^{\circ}(B).D, D \rangle.
	\end{equation}
Note that the first summand in the r.h.s.~of \eqref{24} is coaxial with the tensor ${V}$. Denoting
	\begin{equation}
	\label{26}
	  z_{ij} := \lambda_i^2 + \lambda_j^2 + g_{ij} \, (\lambda_i^2 - \lambda_j^2),
	\end{equation}
we have $z_{ji}=z_{ij}$ since $g_{ij} = -g_{ji}$. Then using Theorem 2.2 in \cite{korobeynikov2018}, we claim that the second summand on the r.h.s.~of \eqref{24} is orthogonal to the tensor ${V}$. Note also that the first summand on the r.h.s.~of \eqref{9-a} is coaxial with the tensor ${V}$ and the second one is orthogonal to this tensor. Hence, by using coaxiality and orthogonality, we can write
	\begin{equation}
	\label{27}
	  \langle \frac{\DD^{\circ}}{\DD t}[B], D \rangle= 2 \, \langle \sum_{k=1}^{m} \frac{\dot{\lambda}_k}{\lambda_k} \, {V}_k , \sum^m_{i=1}\lambda_i \, \dot{\lambda}_i \, {V}_i \rangle + \langle \sum_{k\neq l=1}^{m} {V}_k\,  D\,{V}_l, \sum_{i\neq j=1}^{m}z_{ij} \, {V}_i\, D \, {V}_j \rangle.
	\end{equation}
We rewrite the expressions on the r.h.s.~of \eqref{27} using principal axes instead of eigenprojections \break ($d_{ij} = \langle n_i, D . n_j \rangle$) which yields
	\begin{equation}
	\label{28}
	  \langle \frac{\DD^{\circ}}{\DD t}[B], D \rangle= 2 \, \langle \sum_{k=1}^3 \frac{\dot{\lambda}_k}{\lambda_k} \, {n}_k\otimes {n}_k, \sum^3_{i=1}\lambda_i \, \dot{\lambda}_i \, {n}_i\otimes {n}_i \rangle + \langle \sum_{k\neq l=1}^{3} d_{kl} \, {n}_k\otimes {n}_l, \sum_{i\neq j=1}^3 z_{ij} \, d_{ij} \, {n}_i\otimes {n}_j\rangle.
	\end{equation}
Then we use the property of contractions ($\delta_{ij}$ are the Kronecker symbols)
	$
	  \langle {n}_i\otimes {n}_j, {n}_k\otimes {n}_l \rangle = \delta_{ik} \, \delta_{jl},
	$
and obtain the final expression
	\begin{equation}\label{30}
	\langle \mathbb{A}^{\circ}(B).D, D \rangle =  \langle \frac{\DD^{\circ}}{\DD t}[B], D \rangle= 2\sum_{i=1}^3 \dot{\lambda}_i^2 +  \sum_{i\neq j=1}^3 z_{ij} \, d_{ij}^2 = 2 \, \sum_{i=1}^3  \lambda_i^2 \, d_{ii}^2 +  \sum_{i\neq j=1}^3 z_{ij} \, d_{ij}^2  \, .
	\end{equation}
Since $D = (d_{ij})$ is arbitrary, we infer from \eqref{30} the \textbf{necessary and sufficient conditions} for the positive definiteness of objective corotational rates associated with material spin tensors
	\begin{equation}
	\label{31}
	  \boxed{z_{ij}>0\ (i,j=1,2,3)}.
	\end{equation}
Suppose that the scalar function $g(\lambda_i, \lambda_j)$ can be rewritten as
	\begin{align}
	\label{eqkorob04}
	g(\lambda_i, \lambda_j) = g(\mcZ), \quad \mcZ \colonequals \frac{\lambda_i}{\lambda_j} \, .
	\end{align}
Then from \eqref{13}--\eqref{16} we obtain the following expressions for the scalar functions of classical corotational rates:
	\begin{align}
	\label{eqkorob05}
	g^{\ZJ}(\mcZ) = 0,  \quad \qquad g^{\log}(\mcZ) = \frac{1 + \mcZ^2}{1 - \mcZ^2} + \frac{1}{\log \mcZ},  \quad \qquad g^{\GN}(\mcZ) = \frac{1-\mcZ}{1 + \mcZ}, \quad \qquad g^{\GS}(\mcZ) = \frac{1 + \mcZ^2}{1- \mcZ^2} \, .
	\end{align}
Note that the scalar functions $g_{ij}^{\Aif, 1}$ and $g_{ij}^{\Aif, 2}$ cannot be written in the form \eqref{eqkorob05}. 

Assuming that the representation \eqref{eqkorob04} exists, we rewrite the quantities in \eqref{26} as follows
	\begin{align}
	z_{ij} = \lambda_j^2 \, [\mcZ^2 + 1 + g(\mcZ) \, (\mcZ^2-1)] = \lambda_j^2 \, \overline g(\mcZ),
	\end{align}
where
	\begin{align}
	\overline g(\mcZ) = \mcZ^2 + 1 + g(\mcZ) \, (\mcZ^2-1), \qquad \text{and} \qquad \mcZ := \frac{\lambda_i}{\lambda_j}.
	\end{align}
Then, the necessary and sufficient conditions can be rewritten as
	\begin{align}
	\label{eqkorob06}
	\boxed{\overline g(\mcZ) > 0 \qquad \forall \, \mcZ > 0,}
	\end{align}
and \textbf{invertibility} of $\mathbb{A}^{\circ}(B)$ is satisfied if and only if $z_{ij} \neq 0, (i,j = 1,2,3)$ or $\overline g(\mcZ) \neq 0 \; \forall \, \mcZ > 0$.
\subsection{Testing some classical corotational rates for positivity}

We apply the necessary and sufficient conditions \eqref{31} and \eqref{eqkorob06} to classical rates based on spin tensors from the family of material spins  ($i,j=1,2,3$).
\begin{enumerate}
  \item For the Zaremba--Jaumann rate $\frac{\DD^{\ZJ}}{\DD t}$ from \eqref{13}, \eqref{26} and \eqref{eqkorob05} we have\\
	\begin{equation}\label{32}
	  z_{ij}^{\ZJ}=\lambda_i^2 + \lambda_j^2>0 \qquad \text{or} \qquad \overline g^{\ZJ}(\mcZ) = \mcZ^2 + 1 > 0.
	\end{equation}
  \item For the Green--Naghdi rate $\frac{\DD^{\GN}}{\DD t}$ from \eqref{14}, \eqref{26} and \eqref{eqkorob05} we have\\
	\begin{equation}\label{33}
	  z_{ij}^{\GN}=2 \, \lambda_i \, \lambda_j>0 \qquad \text{or} \qquad \overline g^{\GN}(\mcZ) = 2 \, \mcZ > 0.
	\end{equation}
  \item For the logarithmic rate $\frac{\DD^{\log}}{\DD t}$ from \eqref{15}, \eqref{26} and \eqref{eqkorob05} we have\\
	\begin{equation}\label{34}
	    z_{ij}^{\log}= \lambda_j^2(y_{ij}+1)\frac{y_{ij}-1}{\log y_{ij}}>0,\qquad y_{ij}\coloneqq\frac{\lambda_i}{\lambda_j} \qquad \text{or} \qquad \overline g^{\log}(\mcZ) = \frac{ \mcZ^2 - 1}{\log \mcZ} > 0.
	\end{equation}
	which can be rewritten as
	\begin{equation*}
		z_{ij}^{\log} = \frac{\lambda_i^2 - \lambda_j^2}{\log \lambda_i - \log \lambda_j} > 0 \qquad \text{or} \qquad \overline g^{\log}(\mcZ) = \frac{ \mcZ^2 - 1}{\log \mcZ} > 0.
	\end{equation*}
  \item For the Gurtin--Spear rate $\frac{\DD^{\GS}}{\DD t}$ from \eqref{16}, \eqref{26} and \eqref{eqkorob05} we have\\
	\begin{equation}\label{35}
	    z_{ij}^{\GS}=0 \qquad \text{and} \qquad \overline g^{\GS}(\mcZ) = 0.
	\end{equation}
\end{enumerate}
We can plot these functions in accordance to Figure 1 in \cite{korobeynikov2011} (cf.~Figure \ref{fig2}).

	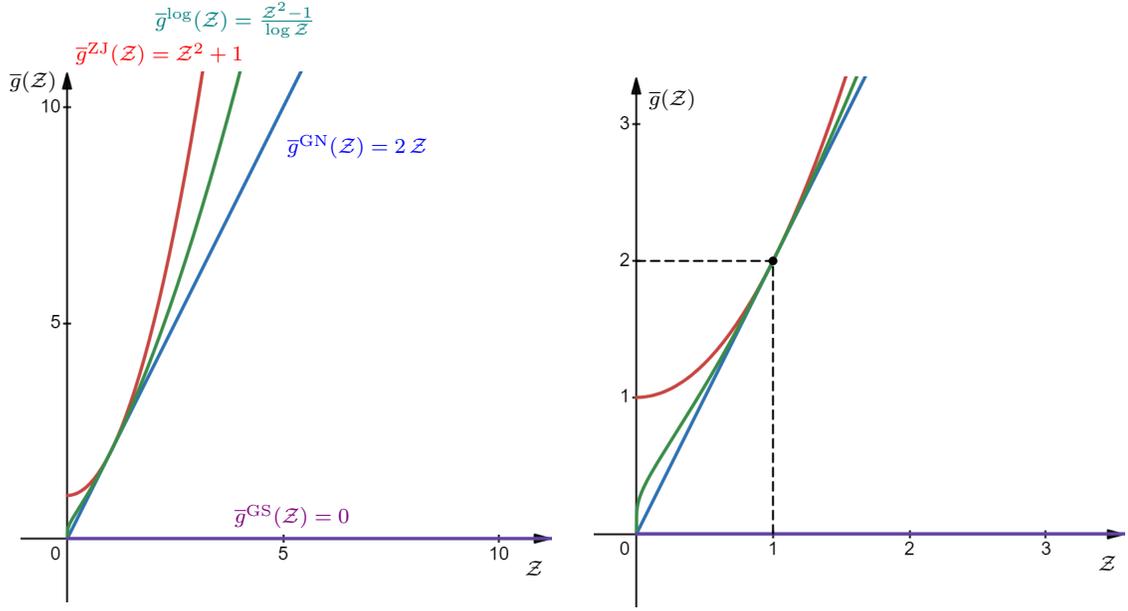
\begin{figure}[h!]
		\begin{center}		
		\begin{minipage}[h!]{0.4\linewidth}
			\centering
			\includegraphics[scale=0.25]{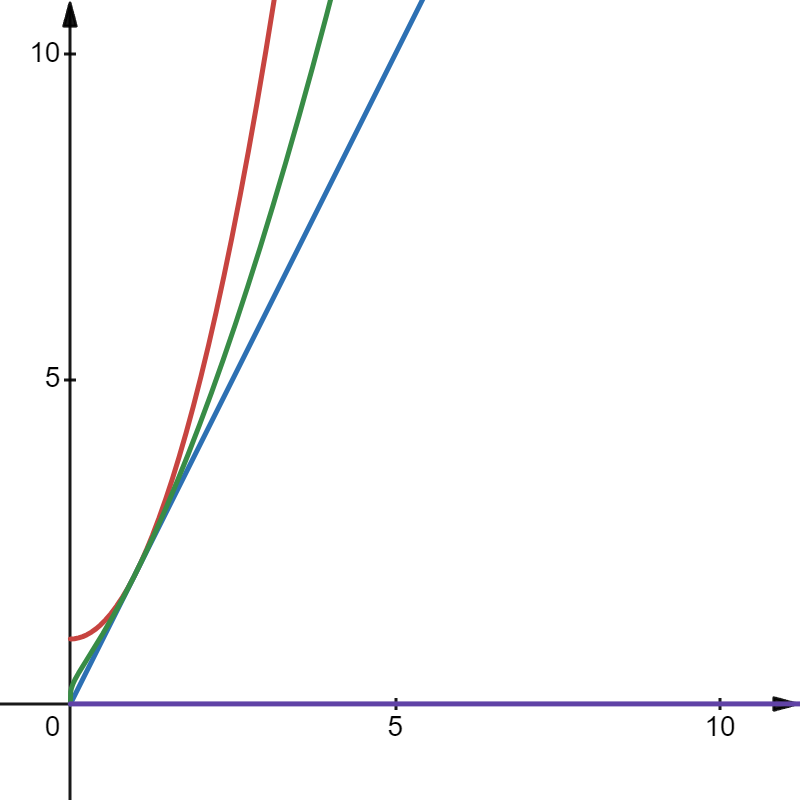}
			\put(-205,195){\footnotesize $\overline g(\mcZ)$}
			\put(-180,205){\footnotesize $\red{\overline g^{\ZJ}(\mcZ) = \mcZ^2+1}$}
			\put(-150,218){\footnotesize $\teal{\overline g^{\log}(\mcZ) = \frac{\mcZ^2-1}{\log \mcZ}}$}
			\put(-100,170){\footnotesize $\blue{\overline g^{\GN}(\mcZ) = 2 \, \mcZ}$}
			\put(-120,30){\footnotesize $\violet{\overline g^{\GS}(\mcZ) = 0}$}
			\put(-10,10){\footnotesize $\mcZ$}
		\end{minipage} \qquad
		\begin{minipage}[h!]{0.4\linewidth}
			\centering
			\vspace*{30pt}
			\includegraphics[scale=0.25]{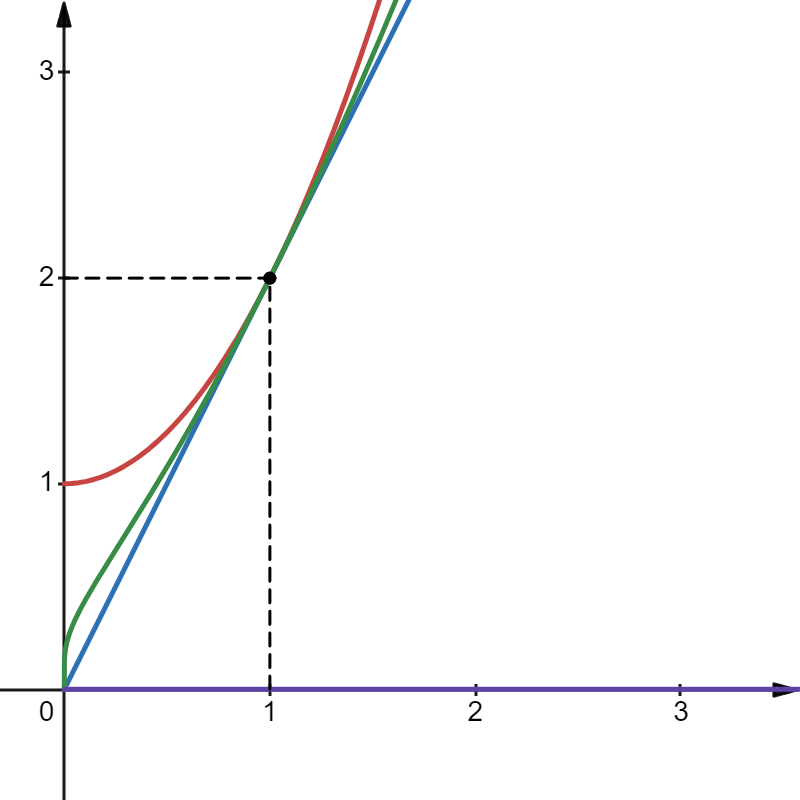}
			\put(-10,14){\footnotesize $\mcZ$}
			\put(-180,190){\footnotesize $\overline g(\mcZ)$}
		\end{minipage}
		\caption{Illustration of the characteristic functions $\overline g(\mcZ)$ from \eqref{eqkorob05}. We observe that the positive corotational rates found so far all have a characteristic function $\overline g(\mcZ) > 0$ which satisfies $\overline g(1) = 2$ and $\overline g'(1) = 2$ as well as convexity for $\mcZ \ge 1$.}
		\label{fig2}
		\end{center}
	\end{figure}

\noindent We conclude from \eqref{32}-\eqref{35} that the necessary and sufficient conditions for the positive definiteness are satisfied for the Zaremba-Jaumann, Green-Naghdi, and logarithmic rates. However, the Gurtin-Spear rate is not positive definite due to the possibility that
	\begin{align}
	\langle \frac{\DD^{\GS}}{\DD t}[B], D \rangle = \langle \mathbb{A}^{\GS}(B).D, D \rangle = 0	
	\end{align}
for some movements with $ D\neq{0}$ (for example, for some movements with $d_{ii}=0$ and $d_{ij}\neq0\ i\neq j$). Since $\mathbb{A}^{\GS}(B)$ is major symmetric, this implies that $\mathbb{A}^{\GS}(B).D = 0$ for some $D \in \Sym(3) \! \setminus \! \{0\}$. Hence, $\mathbb{A}^{\GS}(B)$ is \textbf{not invertible} as already observed by Xiao et al.~\cite[eq.~(4.68)]{xiao98_1}. Note that the Gurtin--Spear rate is associated with the twirl tensor of the Eulerian triad $\Omega^{\GS}=\dot Q^E \, (Q^E)^T$ that does not belong to the family of continuous spin tensors (cf., \cite{korobeynikov2011}).
%
\begin{enumerate}
\item[5] For the Aifantis rate $\frac{\DD^{\Aif, 1}}{\DD t}$ from Example \ref{exampleA1} we have
	\begin{align}
	z_{ij}^{\Aif,1} = \lambda_i^2 + \lambda_j^2 + \frac{\zeta}{(\det B)^{\frac13}} \, (\lambda_i^2 - \lambda_j^2)^2 > 0, \qquad \text{if} \quad \zeta \ge 0 \,.
	\end{align}
It follows
that a sufficient condition for the positivity of this rate is the condition $\zeta \ge 0$.
\item[6] For the Aifantis rate $\frac{\DD^{\Aif,2}}{\DD t}$ from Example \ref{exampleA2} we have
	\begin{equation}
	\begin{alignedat}{2}
	z_{ij}^{\Aif,2} &= \lambda_i^2 + \lambda_j^2 + \zeta \, \left[(\det B)^{-\frac13} \, (\lambda_i^2 - \lambda_j^2)^2 + (\det B)^{\frac13} \, \left(-2 + \frac{\lambda_i^2}{\lambda_j^2} + \frac{\lambda_j^2}{\lambda_i^2} \right)\right] \\
	\overset{\mcZ = \frac{\lambda_i}{\lambda_j}}&{=} \lambda_i^2 + \lambda_j^2 + \zeta \, \bigg[(\det B)^{-\frac13} \, (\lambda_i^2 - \lambda_j^2)^2 + (\det B)^{\frac13} \, \bigg(\underbrace{\mcZ^2 + \frac{1}{\mcZ^2} - 2}_{\ge 0}\bigg)\bigg] > 0, \quad \text{if} \quad \zeta \ge 0 \, .
	\end{alignedat}
	\end{equation}
We observe once again that the rate $\frac{\DD^{\Aif, 2}}{\DD t}$ is positive if $\zeta \ge 0$.
\end{enumerate}
\begin{rem}
	We have not been able to find a corotational rate in the class of continuous material spins that is only invertible but not positive. Clearly,
	\begin{equation*}
		\det \mathbb{A}^{\circ}(B) \neq 0 \quad \text{and} \quad \mathbb{A}^{\circ}(B) \in \Sym^{++}_4(6)
	\end{equation*}
	are different conditions, but $\mathbb{A}^{\circ}(B)$ is not arbitrary but constraint by the representation \eqref{eqrepresA}. However, it is too early to conjecture that invertible rates coincide with positive rates.
\end{rem}

\section{Conclusion}
While all objective tensor rates $\frac{\DD^{\sharp}}{\DD t}$ can be used to formulate equivalent rate-form equations of a given hyper- or Cauchy-elastic model by using the corresponding induced tangent stiffness tensor (cf.~\cite{CSP2024}), only the corotational rates $\frac{\DD^{\circ}}{\DD t}$ disclose the underlying physical mechanisms and satisfy a number of structure preserving geometric properties and should therefore be preferred in the modelling of rate-form equilibrium problems. 

Among all corotational rates it is still possible to single out a meaningful subclass of what we call invertible corotational rates and positive corotational rates. Necessary and sufficient conditions for both classes are derived. Well-known corotational derivatives such as the Zaremba-Jaumann, the Green-Naghdi or the logarithmic derivative already belong to the newly defined class of positive corotational rates. 

In a future contribution we venture to generalize the main result of \cite{tobedone, CSP2024}, i.e.~that the corotational stability postulate (CSP) is equivalent to the strong {\bf T}rue-{\bf S}tress-{\bf T}rue-{\bf S}train-{\bf M}onotonicity (TSTS-M$^{++}$) \textbf{for all positive corotational rates $\frac{\DD^{\circ}}{\DD t}$}, meaning that for the isotropic Cauchy stress $\widehat \sigma(\log B) := \sigma(B)$
	\begin{align*}
	\langle \frac{\DD^{\circ}}{\DD t}[\sigma], D \rangle > 0 \quad \forall \, D \in \Sym(3) \! \setminus \! \{0\} \quad \text{(CSP)} \qquad \iff \qquad \sym \, \DD_{\log B} \widehat \sigma(\log B) \in \Sym^{++}_4(6) \quad \text{(TSTS-M$^{++}$)},
	\end{align*}
which implies the monotonicity in the logarithmic strain tensor $\log B$
	\begin{align*}
	\langle \widehat \sigma(\log B_1) - \widehat \sigma(\log B_2), \log B_1 - \log B_2 \rangle > 0 \qquad \text{(TSTS-M$^+$)},
	\end{align*}
conferring to the CSP a far reaching generality and highlighting again the special role played by the logarithmic strain tensor in setting up constitutive equations (cf.~\cite{NeffGhibaLankeit,NeffGhibaPoly}). Here, the (TSTS-M$^{+}$) condition is a strong candidate for \textit{stress increases with strain} in isotropic nonlinear elasticity.\\

\begingroup
\footnotesize

\bibliographystyle{plain} 
\bibliography{Leblondrefs, Korobeynikov_positive_2024}

\begin{thebibliography}{10}

\bibitem{asghari2008}
M.~Asghari and R.~Naghdabadi.
\newblock On the objective corotational rates of {E}ulerian strain measures.
\newblock {\em Journal of Elasticity}, 90:175--207, 2008.

\bibitem{Aubram2017}
D.~Aubram.
\newblock Notes on rate equations in nonlinear continuum mechanics.
\newblock {\em to appear in Mathematics and Mechanics of Solids;
  (arXiv:1709.10048)}, 2024.

\bibitem{bellini2015}
C.~Bellini and S.~Federico.
\newblock Green-{N}aghdi rate of the {K}irchhoff stress and deformation rate:
  the elasticity tensor.
\newblock {\em Zeitschrift für Angewandte Mathematik und Physik},
  66(3):1143--1163, 2015.

\bibitem{Bertram2021}
A.~Bertram.
\newblock {\em Elasticity and {P}lasticity of {L}arge {D}eformations}.
\newblock Springer, Cham, 4th edition, 2021.

\bibitem{biezeno1928}
C.~B. Biezeno and H.~Hencky.
\newblock On the general theory of elastic stability.
\newblock {\em Koninklijke Akademie van Wettenschappen te Amsterdam},
  31:569--592, 1928.

\bibitem{bruhns2004}
O.~T. Bruhns, A.~Meyers, and H.~Xiao.
\newblock On non-corotational rates of {O}ldroyd's type and relevant issues in
  rate constitutive formulations.
\newblock {\em Proceedings: Mathematical, Physical and Engineering Sciences},
  460(2043):909--928, 2004.

\bibitem{Cotter1955TENSORSAW}
B.~A. Cotter and R.~S. Rivlin.
\newblock Tensors associated with time-dependent stress.
\newblock {\em Quarterly of Applied Mathematics}, 13:1036--1041, 1955.

\bibitem{Curnier1994}
A.~Curnier.
\newblock {\em Computational Methods in Solid Mechanics}.
\newblock Kluwer, Dordrecht, 1994.

\bibitem{tobedone}
M.~V. d'Agostino, S.~Holthausen, D.~Bernardini, A.~Sky, and P.~Neff.
\newblock A constitutive condition for idealized isotropic {C}auchy elasticity
  involving the logarithmic strain.
\newblock {\em to appear in: Journal of Elasticity, arXiv:2409.01811}, 2024.

\bibitem{daleckii1965}
J.~L. Daleckii and S.~G. Krein.
\newblock Integration and differentiation of functions of hermitian operators
  and applications to the theory of perturbations.
\newblock {\em American Mathematical Society Translations: Series 2}, 47:1--30,
  1965.

\bibitem{dienes1979}
J.~K. Dienes.
\newblock On the analysis of rotation and stress rate in deforming bodies.
\newblock {\em Acta Mechanica}, 32:217--232, 1979.

\bibitem{dienes1987discussion}
J.~K. Dienes.
\newblock A discussion of material rotation and stress rate.
\newblock {\em Acta Mechanica}, 65:1--11, 1987.

\bibitem{doyle1956}
T.~C. Doyle and J.~L. Ericksen.
\newblock Nonlinear elasticity.
\newblock {\em Advances in Applied Mechanics}, 4:53--115, 1956.

\bibitem{fiala2009}
Z.~Fiala.
\newblock Is the logarithmic time derivative simply the {Z}aremba-{J}aumann
  derivative?
\newblock {\em Engineering Mechanics, National Conference with International
  Participation, Svratka, Czech Republic, May 11-14}, 211:227--240, 2009.

\bibitem{fiala2016}
Z.~Fiala.
\newblock Geometry of finite deformations and time-incremental analysis.
\newblock {\em International Journal of Non-Linear Mechanics}, 81(1), 2016.

\bibitem{fiala2020objective}
Z.~Fiala.
\newblock Objective time derivatives revised.
\newblock {\em Zeitschrift f{\"u}r angewandte Mathematik und Physik}, 71(1):4,
  2020.

\bibitem{ghavam2007}
K.~Ghavam and R.~Naghdabadi.
\newblock Spin tensors associated with corotational rates and corotational
  integrals in continua.
\newblock {\em International Journal of Solids and Structures}, 44:5222--5235,
  2007.

\bibitem{govindjee1997}
S.~Govindjee.
\newblock Accuracy and stability for integration of {J}aumann stress rate
  equations in spinning bodies.
\newblock {\em Engineering Computations}, 14(1):14--30, 1997.

\bibitem{Green_McInnis_1967}
A.~E. Green and B.~C. McInnis.
\newblock Generalized hypo-elasticity.
\newblock {\em Proceedings of the Royal Society of Edinburgh. Section A.
  Mathematical and Physical Sciences}, 67(3):220–230, 1967.

\bibitem{Green1965}
A.~E. Green and P.~M. Naghdi.
\newblock A general theory of an elastic-plastic continuum.
\newblock {\em Archive for Rational Mechanics and Analysis}, 18(4):251--281,
  1965.

\bibitem{guo63}
Z.~H. Guo.
\newblock Time derivatives of tensor fields in nonlinear continuum mechanics.
\newblock In {\em Archiwum Mechaniki Stosowanej}, volume~15, pages 131--163.
  Polish Scientific Publishers, 1963.

\bibitem{Gurtin2010}
M.~Gurtin, E.~Fried, and L.~Anand.
\newblock {\em The Mechanics and Thermodynamics of Continua}.
\newblock Cambridge: Cambridge University Press., 2010.

\bibitem{gurtin1983relationship}
M.~E. Gurtin and K.~Spear.
\newblock On the relationship between the logarithmic strain rate and the
  stretching tensor.
\newblock {\em International Journal of Solids and Structures}, 19(5):437--444,
  1983.

\bibitem{hashiguchi2009elastoplasticity}
K.~Hashiguchi.
\newblock {\em Elastoplasticity Theory}, volume~42.
\newblock Springer, 2009.

\bibitem{hill1970constitutive}
R.~Hill.
\newblock Constitutive inequalities for isotropic elastic solids under finite
  strain.
\newblock {\em Proceedings of the Royal Society A: Mathematical, Physical and
  Engineering Sciences}, 314(1519):457--472, 1970.

\bibitem{hill1978}
R.~Hill.
\newblock Aspects of invariance in solid mechanics.
\newblock {\em Advances in Applied Mechanics}, 18:1--75, 1978.

\bibitem{holzapfel2000}
G.~A. Holzapfel.
\newblock {\em {Nonlinear Solid Mechanics: A Continuum Approach for
  Engineering}}.
\newblock Wiley, 2000.

\bibitem{horn2013}
R.~A. Horn and C.~R. Johnson.
\newblock {\em Matrix Analysis}.
\newblock Cambridge University Press, New York, 2nd edition, 2013.

\bibitem{jaumann1905}
G.~Jaumann.
\newblock {\em {Die Grundlagen der Bewegungslehre von einem modernen
  Standpunkte aus dargestellt.}}
\newblock Johann Ambrosius Barth, Leipzig, 1905.

\bibitem{jaumann1911geschlossenes}
G.~Jaumann.
\newblock {G}eschlossenes {S}ystem physikalischer und chemischer
  {D}ifferentialgesetze.
\newblock {\em Sitzungsberichte der Mathematisch-Naturwissenschaftlichen Classe
  der Kaiserlichen Akademie der Wissenschaften Wien 2a}, 120:385--530, 1911.

\bibitem{kolev2024objective}
B.~Kolev and R.~Desmorat.
\newblock Objective rates as covariant derivatives on the manifold of
  riemannian metrics.
\newblock {\em Archive for Rational Mechanics and Analysis}, 248(4):66, 2024.

\bibitem{korobeynikovbook2000}
S.~N. Korobeynikov.
\newblock {\em {Nonlinear Strain Analysis of Solids (monograph, 262 pages)}}.
\newblock Novosibirsk, Siberian Branch of the Russian Academy of Sciences
  (editor), 2000 (in Russian).

\bibitem{korobeynikov2008objective}
S.~N. Korobeynikov.
\newblock Objective tensor rates and applications in formulation of
  hyperelastic relations.
\newblock {\em Journal of Elasticity}, 93(2):105--140, 2008.

\bibitem{korobeynikov2011}
S.~N. Korobeynikov.
\newblock Families of continuous spin tensors and applications in continuum
  mechanics.
\newblock {\em Acta Mechanica}, 216(1--4):301--332, 2011.

\bibitem{korobeynikov2018}
S.~N. Korobeynikov.
\newblock Basis-free expressions for families of objective strain tensors,
  their rates, and conjugate stress tensors.
\newblock {\em Acta Mechanica}, 229(3):1061--1098, 2018.

\bibitem{KorobeynikovAM2023}
S.~N. Korobeynikov.
\newblock Discussion of ``{T}he general basis-free spin and its concise proof''
  by {M}eng and {C}hen, {A}cta {M}ech.,
\newblock {\em Acta Mechanica}, 234(2):825--829, 2023.

\bibitem{korobeynikov2023}
S.~N. Korobeynikov.
\newblock Families of {H}ooke-like isotropic hyperelastic material models and
  their rate formulations.
\newblock {\em Archive of Applied Mechanics}, 93:3863--3893, 2023.

\bibitem{korobeynikov2023book}
S.~N. Korobeynikov and A.~Larichkin.
\newblock {\em Objective Algorithms for Integrating Hypoelastic Constitutive
  Relations Based on Corotational Stress Rates}.
\newblock Springer Cham, 2023.

\bibitem{korobeynikov2024}
S.~N. Korobeynikov and A.~Y. Larichkin.
\newblock Simulating body deformations with initial stresses using {H}ooke-like
  isotropic hypoelasticity models based on corotational stress rates.
\newblock {\em Zeitschrift für Angewandte Mathematik und Mechanik}, 104(2),
  e202300568, 2024.

\bibitem{LankeitNeffNakatsukasa}
J.~Lankeit, P.~Neff, and Y.~Nakatsukasa.
\newblock The minimization of matrix logarithms: {O}n a fundamental property of
  the unitary polar factor.
\newblock {\em Linear Algebra and its Applications}, 449(0):28--42, 2014.

\bibitem{lehmann1991}
T.~Lehmann, Z.~H. Guo, and H.~Y. Liang.
\newblock The conjugacy between {C}auchy stress and logarithm of the left
  stretch tensor.
\newblock {\em European Journal of Mechanics A/Solids}, 10:395--404, 1991.

\bibitem{LuehrCMAME1990}
C.~P. Luehr and M.~B. Rubin.
\newblock The significance of projection operators in the spectral
  representation of symmetric second order tensors.
\newblock {\em Computer Methods in Applied Mechanics and Engineering},
  84(3):243--246, 1990.

\bibitem{macmillan1992spin}
E.~H. MacMillan.
\newblock On the spin of tensors.
\newblock {\em Journal of Elasticity}, 27:69--84, 1992.

\bibitem{Marsden83}
J.E. Marsden and J.R. Hughes.
\newblock {\em Mathematical {F}oundations of {E}lasticity.}
\newblock Prentice-Hall, Englewood Cliffs, New Jersey, 1983.

\bibitem{mehrabadi1987}
M.~M. Mehrabadi and S.~Nemat-Nasser.
\newblock Some basic kinematical relations for finite deformations of continua.
\newblock {\em Mechanics of Materials}, 6:127--138, 1987.

\bibitem{metzger1986objective}
D.~R. Metzger and R.~N. Dubey.
\newblock Objective tensor rates and frame indifferent constitutive models.
\newblock {\em Mechanics Research Communications}, 13(2):91--96, 1986.

\bibitem{meyers2000some}
A.~Meyers, P.~Schie{\ss}e, and O.~T. Bruhns.
\newblock {Some comments on objective rates of symmetric Eulerian tensors with
  application to Eulerian strain rates}.
\newblock {\em Acta Mechanica}, 139(1):91--103, 2000.

\bibitem{naghdi1961}
P.~M. Naghdi and W.~L. Wainwright.
\newblock On the time derivative of tensors in mechanics of continua.
\newblock {\em Quarterly of Applied Mathematics}, 19(2):95--109, 1961.

\bibitem{Neff_Osterbrink_Martin_Hencky13}
P.~Neff, B.~Eidel, and R.~J. Martin.
\newblock Geometry of logarithmic strain measures in solid mechanics.
\newblock {\em Archive for Rational Mechanics and Analysis}, 222:507--572,
  2016.

\bibitem{NeffGhibaLankeit}
P.~Neff, I.~D. Ghiba, and J.~Lankeit.
\newblock The exponentiated {H}encky-logarithmic strain energy. {P}art {I}:
  {C}onstitutive issues and rank--one convexity.
\newblock {\em Journal of Elasticity}, 121:143--234, 2015.

\bibitem{NeffGhibaPoly}
P.~Neff, I.~D. Ghiba, J.~Lankeit, R.~J. Martin, and D.J. Steigmann.
\newblock The exponentiated {H}encky-logarithmic strain energy. {P}art {II}:
  {C}oercivity, planar polyconvexity and existence of minimizers.
\newblock {\em Zeitschrift für Angewandte Mathematik und Physik},
  66:1671--1693, 2015.

\bibitem{CSP2024}
P.~Neff., S.~Holthausen, M.~V. d'Agostino, D.~Bernardini, A.~Sky, I.~D. Ghiba,
  and R.~J. Martin.
\newblock Hypo-elasticity, {C}auchy-elasticity, corotational stability and
  monotonicity in the logarithmic strain.
\newblock {\em submitted, arXiv:1234.56789 (dummy)}, 2024.

\bibitem{Neffpolardecomp}
P.~Neff, J.~Lankeit, and A.~Madeo.
\newblock On {G}rioli's minimum property and its relation to {C}auchy's polar
  decomposition.
\newblock {\em International Journal of Engineering Science}, 80:209--217,
  2014.

\bibitem{Neff_Nagatsukasa_logpolar13}
P.~Neff, Y.~Nakatsukasa, and A.~Fischle.
\newblock A logarithmic minimization property of the unitary polar factor in
  the spectral norm and the {Frobenius} matrix norm.
\newblock {\em SIAM Journal on Matrix Analysis and Applications},
  35:1132--1154, 2014.

\bibitem{Noll55}
W.~Noll.
\newblock On the continuity of the solid and fluid states.
\newblock {\em Journal of Rational Mechanics and Analysis}, 4:3--81, 1955.

\bibitem{Norris2008}
A.~N. Norris.
\newblock Eulerian conjugate stress and strain.
\newblock {\em Journal of Mechanics of Materials and Structures},
  3(2):243--260, 2008.

\bibitem{Ogden83}
R.W. Ogden.
\newblock {\em Non-Linear Elastic Deformations.}
\newblock Mathematics and its Applications. Ellis Horwood, Chichester, 1st
  edition, 1983.

\bibitem{oldroyd1950}
J.~G. Oldroyd.
\newblock On the formulation of rheological equation of state.
\newblock {\em Proceedings of the Royal Society of London. Series A,
  Mathematical and Physical Sciences.}, 200:523--541, 1950.

\bibitem{palizi2020consistent}
M.~Palizi, S.~Federico, and S.~Adeeb.
\newblock Consistent numerical implementation of hypoelastic constitutive
  models.
\newblock {\em Zeitschrift f{\"u}r angewandte Mathematik und Physik}, 71:1--23,
  2020.

\bibitem{PeyrautANM2009}
F.~Peyraut, Z.~Q. Feng, Q.~C. He, and N.~Labed.
\newblock Robust numerical analysis of homogeneous and non-homogeneous
  deformations.
\newblock {\em Applied Numerical Mathematics}, 59(7):1499--1514, 2009.

\bibitem{pinsky1983}
P.~M. Pinsky, M.~Ortiz, and K.~S. Pister.
\newblock Numerical integration of rate constitutive equations in finite
  deformation analysis.
\newblock {\em Computer Methods in Applied Mechanics and Engineering},
  40(2):137--158, 1983.

\bibitem{pozdeev1986}
A.~A. Pozdeev, P.~V. Trusov, and Y.~I. Nyashin.
\newblock {\em {Large Elastopastic Strains}}.
\newblock Nauka, Moscow, 1986 (in Russian).

\bibitem{prager1961}
W.~Prager.
\newblock An elementary discussion of definitions of stress rates.
\newblock {\em Quarterly of Applied Mathematics}, 18(4):403--407, 1960.

\bibitem{prager1962}
W.~Prager.
\newblock On higher rates of stress and deformation.
\newblock {\em Journal of the Mechanics and Physics of Solids}, 10(2):133--138,
  1962.

\bibitem{reinhardt1995eulerian}
W.~D. Reinhardt and R.~N. Dubey.
\newblock {E}ulerian strain-rate as a rate of logarithmic strain.
\newblock {\em Mechanics Research Communications}, 22(2):165--170, 1995.

\bibitem{reinhardt1996application}
W.~D. Reinhardt and R.~N. Dubey.
\newblock Application of objective rates in mechanical modeling of solids.
\newblock {\em Journal of Applied Mechanics}, 63:692--698, 1996.

\bibitem{reinhardt1996}
W.~D. Reinhardt and R.~N. Dubey.
\newblock Coordinate-independent representation of spins in continuum
  mechanics.
\newblock {\em Journal of Elasticity}, 42:133--144, 1996.

\bibitem{richter1948isotrope}
H.~Richter.
\newblock Das isotrope {E}lastizitätsgesetz.
\newblock {\em Zeitschrift für Angewandte Mathematik und Mechanik},
  28(7-8):205--209, 1948.

\bibitem{richter1949hauptaufsatze}
H.~Richter.
\newblock {Verzerrungstensor, Verzerrungsdeviator und Spannungstensor bei
  endlichen Formänderungen}.
\newblock {\em Zeitschrift für Angewandte Mathematik und Mechanik},
  29(3):65--75, 1949.

\bibitem{Richter50}
H.~Richter.
\newblock {Zum Logarithmus einer Matrix.}
\newblock {\em Archiv der Mathematik}, 2:360--363, 1950.

\bibitem{Richter52}
H.~Richter.
\newblock {Zur Elastizitätstheorie endlicher Verformungen.}
\newblock {\em Mathematische Nachrichten}, 8:65--73, 1952.

\bibitem{seth1961}
B.~R. Seth.
\newblock {\em Generalized strain measure with applications to physical
  problems.}
\newblock Technical report, Defense Technical Information Center, 1961.

\bibitem{simo2006computational}
J.~C. Simo.
\newblock {\em Computational Inelasticity}.
\newblock Springer Science \& Business Media, 2006.

\bibitem{sowerby1984rotations}
R.~Sowerby and E.~Chu.
\newblock Rotations, stress rates and strain measures in homogeneous
  deformation processes.
\newblock {\em International Journal of Solids and Structures},
  20(11-12):1037--1048, 1984.

\bibitem{szabo1989comparison}
L.~Szab{\'o} and M.~Balla.
\newblock Comparison of some stress rates.
\newblock {\em International Journal of Solids and Structures}, 25(3):279--297,
  1989.

\bibitem{truesdellremarks}
C.~A. Truesdell.
\newblock Remarks on hypo-elasticity.
\newblock {\em Journal of Research of the National Bureau of Standards, Section
  B: Mathematics and Mathematical Physics}, 67B(3), 1963.

\bibitem{truesdell1966}
C.~A. Truesdell.
\newblock {\em {The Elements of Continuum Mechanics.}}
\newblock Springer-Verlag, Berlin, 1966.

\bibitem{Truesdell65}
C.~A. Truesdell and W.~Noll.
\newblock The {N}on-{L}inear {F}ield {T}heories of {M}echanics.
\newblock In S.~Flügge, editor, {\em Handbuch der {P}hysik}, page volume
  III/3. Springer, Heidelberg, 1965.

\bibitem{xiao97}
H.~Xiao, O.~T. Bruhns, and A.~Meyers.
\newblock Logarithmic strain, logarithmic spin and logarithmic rate.
\newblock {\em Acta Mechanica}, 124(1):89--105, 1997.

\bibitem{xiao1998direct}
H.~Xiao, O.~T. Bruhns, and A.~Meyers.
\newblock {Direct relationship between the Lagrangean logarithmic strain and
  the Lagrangean stretching and the Lagrangean Kirchhoff stress}.
\newblock {\em Mechanics Research Communications}, 25(1):59--67, 1998.

\bibitem{xiao98_2}
H.~Xiao, O.~T. Bruhns, and A.~Meyers.
\newblock {Objective corotational rates and unified work-conjugacy relation
  between Eulerian and Lagrangean strain and stress measures.}
\newblock {\em Archives of Mechanics}, 50(6):1015--1045, 1998.

\bibitem{xiao1998objective}
H.~Xiao, O.~T. Bruhns, and A.~Meyers.
\newblock On objective corotational rates and their defining spin tensors.
\newblock {\em International Journal of Solids and Structures},
  35(30):4001--4014, 1998.

\bibitem{xiao98_1}
H.~Xiao, O.~T. Bruhns, and A.~Meyers.
\newblock Strain rates and material spins.
\newblock {\em Journal of Elasticity}, 52(1):1--41, 1998.

\bibitem{xiao1999natural}
H.~Xiao, O.~T. Bruhns, and A.~Meyers.
\newblock {A natural generalization of hypoelasticity and Eulerian rate type
  formulation of hyperelasticity}.
\newblock {\em Journal of Elasticity}, 56:59--93, 1999.

\bibitem{xiao2006elastoplasticity}
H.~Xiao, O.~T. Bruhns, and A.~Meyers.
\newblock Elastoplasticity beyond small deformations.
\newblock {\em Acta Mechanica}, 182(1):31--111, 2006.

\bibitem{XiaoIJSS1998}
H.~Xiao, O.T. Bruhns, and A.~Meyers.
\newblock On objective corotational rates and their defining spin tensors.
\newblock {\em International Journal of Solids and Structures},
  35(30):4001--4014, 1998.

\bibitem{zaremba1903forme}
S.~Zaremba.
\newblock Sur une forme perfectionnée de la théorie de la relaxation.
\newblock {\em Bulletin International de l'Academie des Sciences de Cracovie},
  pages 534--614, 1903.

\bibitem{Zbib1988}
H.~M. Zbib and E.~C. Aifantis.
\newblock {On the concept of relative and plastic spins and its implications to
  large deformation theories. Part I: Hypoelasticity and vertex-type
  plasticity}.
\newblock {\em Acta Mechanica}, 75(1):15--33, 1988.

\bibitem{zohdi2006}
T.~I. Zohdi.
\newblock Uncertainty growth in hypoelastic material models.
\newblock {\em Mathematics and Mechanics of Solids}, 11(6):555--562, 2006.

\end{thebibliography}
\endgroup
\begin{appendix}
\section{Notation} \label{appendixnotation}
\textbf{The deformation $\varphi(x,t)$, the material time derivative $\frac{\DD}{\DD t}$ and the partial time derivative $\partial_t$} \\
\\
In accordance with \cite{Marsden83} we agree on the following convention regarding an elastic deformation $\varphi$ and time derivatives of material quantities:

Given two sets $\Omega, \Omega_{\xi} \subset \R^3$ we denote by $\varphi: \Omega \to \Omega_{\xi}, x \mapsto \varphi(x) = \xi$ the deformation from the \emph{reference configuration} $\Omega$ to the \emph{current configuration} $\Omega_{\xi}$. A \emph{motion} of $\Omega$ is a time-dependent family of deformations, written $\xi = \varphi(x,t)$. The \emph{velocity} of the point $x \in \Omega$ is defined by $\overline{V}(x,t) = \partial_t \varphi(x,t)$ and describes a vector emanating from the point $\xi = \varphi(x,t)$ (see also Figure \ref{yfig1}). Similarly, the velocity viewed as a function of $\xi \in \Omega_{\xi}$ is denoted by $v(\xi,t)$. 

	\begin{figure}[h!]
		\begin{center}		
		\begin{minipage}[h!]{0.8\linewidth}
			\centering
			\hspace*{-40pt}
			\includegraphics[scale=0.4]{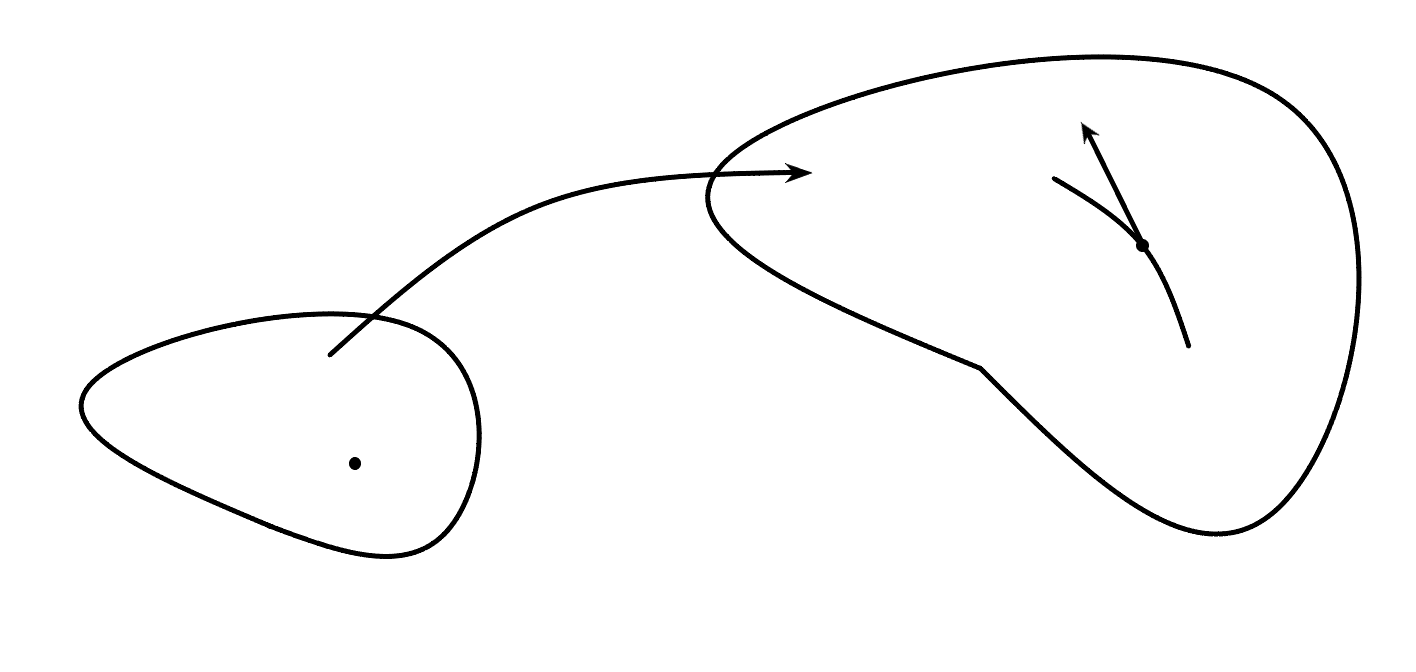}
			\put(-40,30){\footnotesize $\Omega_\xi$}
			\put(-340,25){\footnotesize $\Omega_x$}
			\put(-316,64){\footnotesize $x$}
			\put(-280,148){\footnotesize $\varphi(x,t)$}
			\put(-104,168){\footnotesize $\overline V(x,t) \!=\! v(\xi,t)$}
			\put(-88,119){\footnotesize $\xi$}
			\put(-105,90){\footnotesize curve $t \mapsto \varphi(x,t)$}
			\put(-85,80){\footnotesize  for $x$ fixed}
		\end{minipage} 
		\caption{Illustration of the deformation $\varphi(x,t): \Omega_x \to \Omega_{\xi}$ and the velocity $\overline V(x,t) = v(\xi,t)$.}
		\label{yfig1}
		\end{center}
	\end{figure}
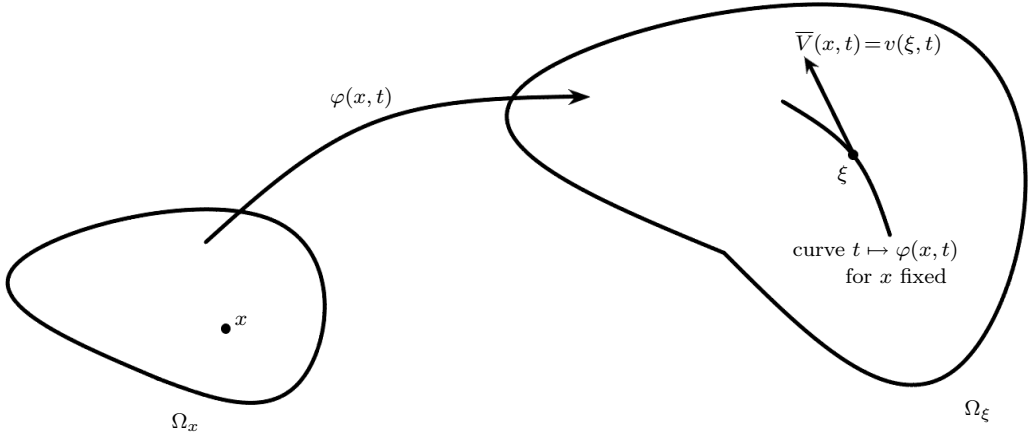

Considering an arbitrary material quantity $Q(x,t)$ on $\Omega$, equivalently represented by $q(\xi,t)$ on $\Omega_\xi$, we obtain by the chain rule for the time derivative of $Q(x,t)$
	\begin{align}
	\frac{\DD}{\DD t}q(\xi,t) \colonequals \frac{\dif}{\dif t}[Q(x,t)] = \DD_\xi q(\xi,t).v(\xi,t) + \partial_t q(\xi,t) \, .
	\end{align}
Since it is always possible to view any material quantity $Q(x,t) = q(\xi,t)$ from two different angles, namely by holding $x$ or $\xi$ fixed, we agree to write
	\begin{itemize}
	\item $\dot q \colonequals \dd \frac{\DD}{\DD t}[q]$ for the material (substantial) derivative of $q$ with respect to $t$ holding $x$ fixed and
	\item $\partial_t q$ for the derivative of $q$ with respect to $t$ holding $\xi$ fixed.
	\end{itemize}
For example, we obtain the velocity gradient $L := \DD_\xi v(\xi,t)$ by
	\begin{align}
	L = \DD_\xi v(\xi,t) = \DD_\xi \overline V(x,t) \overset{\text{def}}&{=} \DD_\xi \frac{\dif}{\dif t} \varphi(x,t) = \DD_{\xi} \partial_t \varphi(\varphi^{-1}(\xi,t),t) = \partial_t \DD_x \varphi(\varphi^{-1}(\xi,t),t) \, \DD_\xi \big(\varphi^{-1}(\xi,t)\big) \notag \\
&=  \partial_t \DD_x \varphi(\varphi^{-1}(\xi,t),t) \, (\DD_x \varphi)^{-1}(\varphi^{-1}(\xi,t),t) = \dot F(x,t) \, F^{-1}(x,t) = L \, ,
	\end{align}
where we used that $\partial_t = \frac{\dif}{\dif t} = \frac{\DD}{\DD t}$ are all the same, if $x$ is fixed. \\
\\
As another example, when determining a corotational rate $\frac{\DD^{\circ}}{\DD t}$ we write
	\begin{align}
	\frac{\DD^{\circ}}{\DD t}[\sigma] = \frac{\DD}{\DD t}[\sigma] + \sigma \, \Omega^{\circ} - \Omega^{\circ} \, \sigma = \dot \sigma + \sigma \, \Omega^{\circ} - \Omega^{\circ} \, \sigma \, .
	\end{align}
However, if we solely work on the current configuration, i.e.~holding $\xi$ fixed, we write $\partial_t v$ for the time-derivative of the velocity (or any quantity in general). \\
\\
\noindent \textbf{Inner product} \\
\\
For $a,b\in\R^n$ we let $\langle {a},{b}\rangle_{\R^n}$  denote the scalar product on $\R^n$ with associated vector norm $\norm{a}_{\R^n}^2=\langle {a},{a}\rangle_{\R^n}$. We denote by $\R^{n\times n}$ the set of real $n\times n$ second order tensors, written with capital letters. The standard Euclidean scalar product on $\R^{n\times n}$ is given by
$\langle {X},{Y}\rangle_{\R^{n\times n}}=\tr{(X Y^T)}$, where the superscript $^T$ is used to denote transposition. Thus the Frobenius tensor norm is $\norm{X}^2=\langle {X},{X}\rangle_{\R^{n\times n}}$, where we usually omit the subscript $\R^{n\times n}$ in writing the Frobenius tensor norm. The identity tensor on $\R^{n\times n}$ will be denoted by $\id$, so that $\tr{(X)}=\langle {X},{\id}\rangle$. \\
\\
\noindent \textbf{Frequently used spaces} 
	\begin{itemize}
	\item $\Sym(n), \rm \Sym^+(n)$ and $\Sym^{++}(n)$ denote the symmetric, positive semi-definite symmetric and positive definite symmetric second order tensors respectively. Note that $\Sym^{++}(n)$ is considered herein only as an algebraic subset of $\Sym(n)$, not endowed with a Riemannian geometry \cite{fiala2009, fiala2016, fiala2020objective, kolev2024objective}.
	\item ${\rm GL}(n):=\{X\in\R^{n\times n}\;|\det{X}\neq 0\}$ denotes the general linear group.
	\item ${\rm GL}^+(n):=\{X\in\R^{n\times n}\;|\det{X}>0\}$ is the group of invertible matrices with positive determinant.
	\item $\mathrm{O}(n):=\{X\in {\rm GL}(n)\;|\;X^TX=\id\}$.
	\item ${\rm SO}(n):=\{X\in {\rm GL}(n,\R)\;|\; X^T X=\id,\;\det{X}=1\}$.
	\item $\mathfrak{so}(3):=\{X\in\mathbb{R}^{3\times3}\;|\;X^T=-X\}$ is the Lie-algebra of skew symmetric tensors.
	\item The set of positive real numbers is denoted by $\R_+:=(0,\infty)$, while $\overline{\R}_+=\R_+\cup \{\infty\}$.
	\end{itemize}
\textbf{Frequently used tensors}
	\begin{itemize}
	\item $F = \DD \varphi(x,t)$ is the Fréchet derivative (Jacobian matrix) of the deformation $\varphi(\cdot,t) : \Omega_x \to \Omega_{\xi} \subset \R^3$. $\varphi(x,t)$ is usually assumed to be a diffeomorphism at every time $t \ge 0$ so that the inverse mapping $\varphi^{-1}(\cdot,t) : \Omega_{\xi} \to \Omega_x$ exists.
	\item $C=F^T \, F$ is the right Cauchy-Green strain tensor.
	\item $B=F\, F^T$ is the left Cauchy-Green (or Finger) strain tensor.
	\item $U = \sqrt{F^T \, F} \in \Sym^{++}(3)$ is the right stretch tensor, i.e.~the unique element of ${\rm Sym}^{++}(3)$ with $U^2=C$.
	\item $V = \sqrt{F \, F^T} \in \Sym^{++}(3)$ is the left stretch tensor, i.e.~the unique element of ${\rm Sym}^{++}(3)$ with $V^2=B$.
	\item $\log V = \frac12 \, \log B$ is the spatial logarithmic strain tensor or Hencky strain.
	\item We write $V = Q$ diag($\lambda_1, \lambda_2, \lambda_3$) $Q^T$, where $\lambda_i \in \R_+$ are the principal stretches.
	\item $L = \dot F \, F^{-1} = \DD_\xi v(\xi)$ is the spatial velocity gradient.
	\item $D = \sym \, L$ is the spatial rate of deformation, the Eulerian strain rate tensor.
	\item $W = \sk \, L$ is the vorticity tensor.
	\item We also have the polar decomposition $F = R \, U = V R \in {\rm GL}^+(3)$ with an orthogonal matrix $R \in \OO(3)$ (cf. Neff et al.~\cite{Neffpolardecomp}), see also \cite{LankeitNeffNakatsukasa,Neff_Nagatsukasa_logpolar13}.
	\end{itemize}
\noindent \textbf{Frequently used rates}
	\begin{multicols}{2}
	\begin{itemize}
	\item $\dd \frac{\DD^{\sharp}}{\DD t}$ denotes an arbitrary objective derivative,
	\item $\dd \frac{\DD^{\circ}}{\DD t} \begin{array}{l} \text{denotes an arbitrary corotational} \\ \text{derivative,} \end{array}$
	\item $\dd \frac{\DD^{\ZJ}}{\DD t} \begin{array}{l} \text{denotes the Zaremba-Jaumann} \\ \text{derivative,} \end{array}$
	\item $\dd \frac{\DD^{\GN}}{\DD t}$ denotes the Green-Naghdi derivative.
	\item $\dd \frac{\DD^{\log}}{\DD t}$ denotes the logarithmic derivative.
	\item $\dd \frac{\DD}{\DD t}$ denotes the material derivative.
	\end{itemize}
	\end{multicols}
\noindent \textbf{Calculus with the material derivative - some examples} \\
\\
Consider the spatial Cauchy stress
	\begin{align}
	\label{eqmat01}
	\sigma(\xi,t) \colonequals \Sigma(B) = \Sigma(F(x,t) \, F^T(x,t)) = \Sigma(F(\varphi^{-1}(\xi,t),t) \, F^T(\varphi^{-1}(\xi,t),t)) \, .
	\end{align}
Then, on the one hand we have for the material derivative
	\begin{align}
	\label{eqmat02}
	\frac{\DD}{\DD t}[\sigma] = \DD_\xi \sigma(\xi,t).v(\xi,t) + \partial_t \sigma(\xi,t) \cdot 1
	\end{align}
and on the other hand equivalently
	\begin{equation}
	\label{eqmat03}
	\begin{alignedat}{2}
	\frac{\DD}{\DD t}[\sigma] &= \frac{\DD}{\DD t}[\Sigma(F(x,t) \, F(x,t)^T)] \overset{(1)}{=} \frac{\dif}{\dif t}[\Sigma(F(x,t) \, F^T(x,t))] \\
	\overset{\substack{\text{standard} \\ \text{chain rule}}}&{=} \DD_B \Sigma(F(x,t) \, F^T(x,t)). \frac{\dif}{\dif t}[(F(x,t) \, F^T(x,t))] = \DD_B \Sigma(F(x,t) \, F^T(x,t)).(\dot F \, F^T + F \, \dot{F}^T) \\
	&= \DD_B \Sigma(F(x,t) \, F^T(x,t)).[\dot F \, F^{-1} \, F \, F^T + F \, F^T \, F^{-T} \, \dot{F}^T] =\DD_B \Sigma(F(x,t) \, F^T(x,t)).[L \, B + B \, L^T]\\
	 &= \DD_B \Sigma(B). [L \, B + B \, L^T]\, .
	\end{alignedat}
	\end{equation}
In $\eqref{eqmat03}_{1}$ we have used the fact that there is already a material representation which allows to set $\frac{\DD}{\DD t} = \frac{\dif}{\dif t}$. Of course, \eqref{eqmat02} is equivalent to \eqref{eqmat03}. From the context it should be clear which representation of $\sigma$ (referential or spatial) we are working with and by abuse of notation we do not distinguish between $\sigma$ and $\Sigma$. \\
\\
The same must be observed when calculating with corotational derivatives
	\begin{align}
	\label{eqmat04}
	\frac{\DD^{\circ}}{\DD t}[\sigma] = \frac{\DD}{\DD t}[\sigma] + \sigma \, \Omega^{\circ} - \Omega^{\circ} \, \sigma, \qquad \Omega^{\circ} = \frac{\DD}{\DD t}Q^{\circ}(x,t) \, (Q^{\circ})^T(x,t) = \frac{\dif}{\dif t} Q^{\circ}(x,t) \, (Q^{\circ})^T(x,t) \, .
	\end{align}
Here, we have
	\begin{align}
	\label{eqmat05}
	\frac{\DD^{\circ}}{\DD t}[\sigma] \overset{(\ast \ast)}&{=} Q^{\circ}(x,t) \, \frac{\DD}{\DD t}[(Q^{\circ})^T(x,t) \, \sigma \, Q^{\circ}(x,t)] \, (Q^{\circ})^T(x,t) \\
	&=Q^{\circ}(x,t) \, \left\{\frac{\DD}{\DD t}(Q^{\circ})^T(x,t) \, \sigma \, Q^{\circ}(x,t) + (Q^{\circ})^T(x,t) \, \frac{\DD}{\DD t}[\sigma] \, Q^{\circ}(x,t) + (Q^{\circ})^T(x,t) \, \sigma \, \frac{\DD}{\DD t}Q^{\circ}(x,t) \right\} \, (Q^{\circ})^T(x,t) \notag \\
	&=Q^{\circ}(x,t) \, \bigg\{\frac{\dif}{\dif t}(Q^{\circ})^T(x,t) \, \sigma \, Q^{\circ}(x,t) + (Q^{\circ})^T(x,t) \, \underbrace{\frac{\DD}{\DD t}[\sigma]}_{(\ast \ast \ast)} \, Q^{\circ}(x,t) + (Q^{\circ})^T(x,t) \, \sigma \, \frac{\dif}{\dif t}Q^{\circ}(x,t) \bigg\} \, (Q^{\circ})^T(x,t) \notag 
	\end{align}
and we can decide for $(\ast \ast \ast)$ to continue the calculus with \eqref{eqmat02} or \eqref{eqmat03}. In either case one has to decide viewing the functions as defined on the reference configuration $\Omega$ or in the spatial configuration $\Omega_\xi$.

In \eqref{eqmat05} we used $Q = Q(x,t)$ and $\Omega = \Omega(x,t)$. This means that the ``Lie-type'' representation $(\ast \, \ast)$ necessitates the definition of a reference configuration, so that we can switch between $\xi = \varphi(x,t)$ and $x$.

The interpretation $(\ast \, \ast)$ is most clearly represented for the Green-Naghdi rate, in which the spin \break $\Omega^{\GN} \colonequals \frac{\dif}{\dif t}R(x,t) \, R^T(x,t) = \dot R(x,t) \, R^T(x,t)$ is defined according to the polar decomposition $F = R \, U$ and in
	\begin{align}
	\frac{\DD^{\GN}}{\DD t}[\sigma] = \frac{\DD}{\DD t}[\sigma] + \sigma \, \Omega^{\GN} - \Omega^{\GN} \, \sigma = R \, \frac{\DD}{\DD t}[R^T \, \sigma \, R] \, R^T
	\end{align}
the term $[R^T \, \sigma \, R]$ is called \emph{corotational stress tensor} (cf.~\cite[p.~142]{Marsden83}). \\
\\
\noindent \textbf{Tensor domains} \\
\\
Denoting the reference configuration by $\Omega_x$ with tangential space $T_x \Omega_x$ and the current/spatial configuration by $\Omega_\xi$ with tangential space $T_\xi \Omega_\xi$ as well as $\varphi(x) = \xi$, we have the following relations (see also Figure \ref{yfig2}):

	\begin{figure}[h!]
		\begin{center}		
		\begin{minipage}[h!]{0.8\linewidth}
			\centering
			\hspace*{-80pt}
			\includegraphics[scale=0.5]{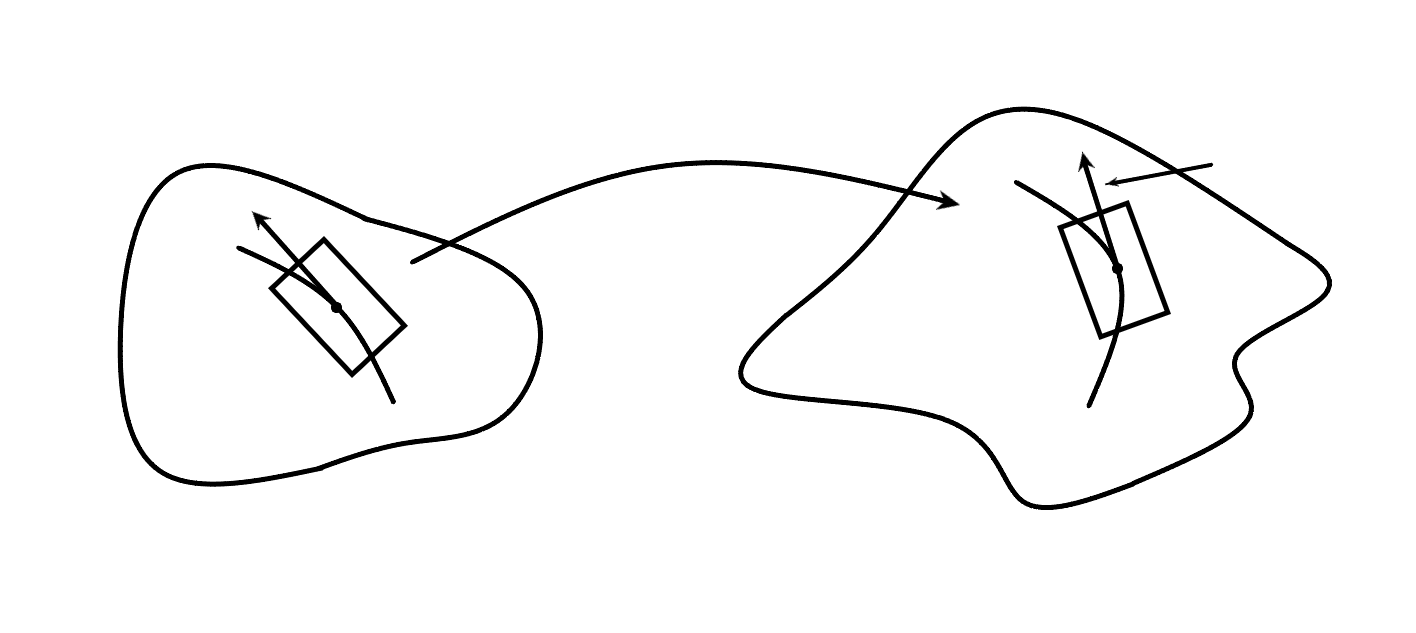}
			\put(-100,45){\footnotesize $\Omega_\xi$}
			\put(-390,55){\footnotesize $\Omega_x$}
			\put(-412,115){\footnotesize $x$}
			\put(-430,155){\footnotesize $\dot \gamma(0)$}
			\put(-377,105){\footnotesize $T_x \Omega_x$}
			\put(-383,85){\footnotesize $\gamma(s)$}
			\put(-280,183){\footnotesize $\varphi(x,t_0)$}
			\put(-120,131){\footnotesize $\xi$}
			\put(-78,181){\footnotesize $\frac{\dif}{\dif s}\varphi(\gamma(s),t_0)\bigg\vert_{s=0}$}
			\put(-88,119){\footnotesize $T_\xi \Omega_\xi$}
			\put(-115,88){\footnotesize $\varphi(\gamma(s),t_0)$}
		\end{minipage} 
		\caption{Illustration of the curve $s \mapsto \varphi(\gamma(s),t_0), \; \gamma(0) = x$ for a fixed time $t = t_0$ with vector field \break $s \mapsto \frac{\dif}{\dif s} \varphi(\gamma(s),t) \in T_\xi \Omega_\xi$.}
		\label{yfig2}
		\end{center}
	\end{figure}
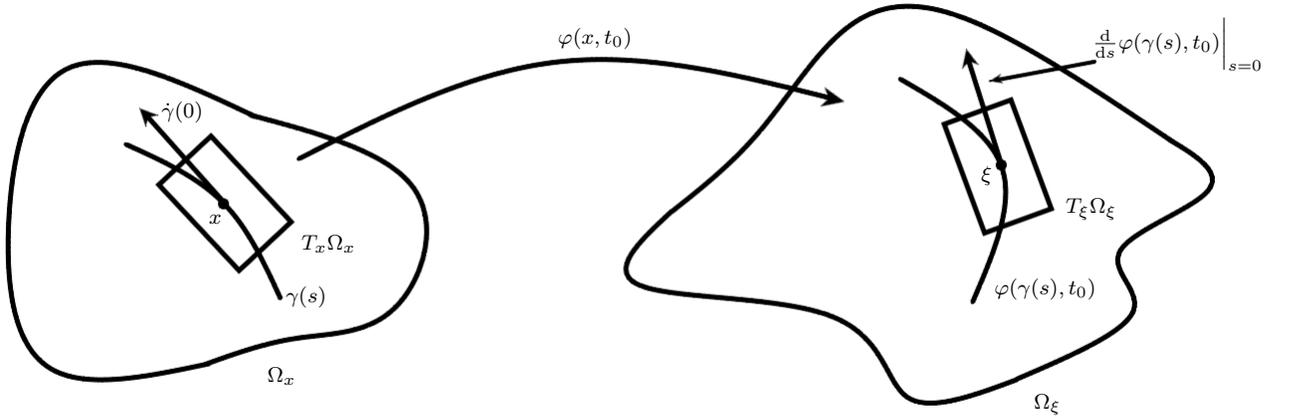
	\begin{multicols}{2}
	\begin{itemize}
	\item $F\colon T_x \Omega_x \to T_\xi \Omega_\xi$
	\item $Q\colon T_x \Omega_x \to T_\xi \Omega \xi$
	\item $F^T\colon T_\xi \Omega_\xi \to T_x \Omega_x$
	\item $Q^T\colon T_\xi \Omega_\xi \to T_x \Omega_x$
	\item $C = F^T \, F\colon T_x \Omega_x \to T_x \Omega_x$
	\item $B = F \, F^T\colon T_\xi \Omega_\xi \to T_\xi \Omega_\xi$
	\item $\sigma\colon T_\xi \Omega_\xi \to T_\xi \Omega_\xi$
	\item $\tau\colon T_\xi \Omega_\xi \to T_\xi \Omega_\xi$
	\item $S_2 \colon T_x \Omega_x \to T_x \Omega_x$
	\item $S_1 \colon T_x \Omega_x \to T_\xi \Omega_\xi$
	\item $Q^T \, \sigma \, Q\colon T_x \Omega_x \to T_x \Omega_x$
	\end{itemize}
	\end{multicols}

\noindent \textbf{Primary matrix functions} \\
\\
We define primary matrix functions as those functions $\Sigma \colon \Sym^{++}(3) \to \Sym(3)$, such that
	\begin{align}
	\Sigma(V) = \Sigma(Q^T \, \text{diag}_V(\lambda_1, \lambda_2, \lambda_3) \, Q) = Q^T \Sigma(\text{diag}_V(\lambda_1, \lambda_2, \lambda_3)) \, Q = Q^T \,
		\begin{pmatrix}
		f(\lambda_1) & 0 & 0 \\
		0 & f(\lambda_2) & 0 \\
		0 & 0 & f(\lambda_3)
		\end{pmatrix} \, Q
	\end{align}
with one given real-valued scale-function $f \colon \mathbb{R}_+ \to \mathbb{R}$. Any primary matrix function is an isotropic matrix function but not vice-versa as shows e.g.~$\Sigma(V) = \det V \, \id$. \\
\\
\textbf{List of additional definitions and useful identities}
	\begin{itemize}
	\item For two metric spaces $X, Y$ and a linear map $L: X \to Y$ with argument $v \in X$ we write $L.v:=L(v)$. This applies to a second order tensor $A$ and a vector $v$ as $A.v$ as well as a fourth order tensor $\C$ and a second order tensor $H$ as $\C.H$.
	\item We define $J = \det{F}$ and denote by $\Cof(X) = (\det X)X^{-T}$ the \emph{cofactor} of a matrix in ${\rm GL}^{+}(3)$.
	\item We define $\sym X = \frac12 \, (X + X^T)$ and $\sk X = \frac12 \, (X - X^T)$ as well as $\dev X = X - \frac13 \, \tr(X) \, \id$.
	\item For all vectors $\xi,\eta\in\R^3$ we have the tensor or dyadic product $(\xi\otimes\eta)_{ij}=\xi_i\,\eta_j$.
	\item $S_1=\DD_F \WW(F) = \sigma \, \Cof F$ is the non-symmetric first Piola-Kirchhoff stress tensor.
	\item $S_2=F^{-1}S_1=2\,\DD_C \widetilde{\WW}(C)$ is the symmetric second  Piola-Kirchhoff stress tensor.
	\item $\sigma=\frac{1}{J}\,  S_1\, F^T=\frac{1}{J}\,  F\,S_2\, F^T=\frac{2}{J}\DD_B \widetilde{\WW}(B)\, B=\frac{1}{J}\DD_V \widetilde{\WW}(V)\, V = \frac{1}{J} \, \DD_{\log V} \widehat \WW(\log V)$ is the symmetric Cauchy stress tensor.
	\item $\sigma = \frac{1}{J} \, F\, S_2 \, F^T = \frac{2}{J} \, F \, \DD_C \widetilde{\WW}(C) \, F^T$ is the ''\emph{Doyle-Ericksen formula}'' \cite{doyle1956}.
	\item For $\sigma: \Sym(3) \to \Sym(3)$ we denote by $\DD_B \sigma(B)$ with $\sigma(B+H) = \sigma(B) + \DD_B \sigma(B).H + o(H)$ the Fréchet-derivative. For $\sigma: \Sym^+(3) \subset \Sym(3) \to \Sym(3)$ the same applies. Similarly, for $\WW : \R^{3 \times 3} \to \R$ we have $\WW(X + H) = \WW(X) + \langle \DD_X \WW(X), H \rangle + o(H)$.
	\item $\tau = J \, \sigma = 2\, \DD_B \widetilde{\WW}(B)\, B $ is the symmetric Kirchhoff stress tensor.
	\item $\tau = \DD_{\log V} \widehat{\WW}(\log V)$ is the ``\emph{Richter-formula}'' \cite{richter1948isotrope, richter1949hauptaufsatze}.
	\item $\sigma_i =\dd\frac{1}{\lambda_1\lambda_2\lambda_3}\dd\lambda_i\frac{\partial g(\lambda_1,\lambda_2,\lambda_3)}{\partial \lambda_i}=\dd\frac{1}{\lambda_j\lambda_k}\dd\frac{\partial g(\lambda_1,\lambda_2,\lambda_3)}{\partial \lambda_i}, \ \ i\neq j\neq k \neq i$ are the principal Cauchy stresses (the eigenvalues of the Cauchy stress tensor $\sigma$), where $g:\mathbb{R}_+^3\to \mathbb{R}$ is the unique function  of the singular values of $U$ (the principal stretches) such that $\WW(F)=\widetilde{\WW}(U)=g(\lambda_1,\lambda_2,\lambda_3)$.
	\item $\sigma_i =\dd\frac{1}{\lambda_1\lambda_2\lambda_3}\frac{\partial \widehat{g}(\log \lambda_1,\log \lambda_2,\log \lambda_3)}{\partial \log \lambda_i}$, where $\widehat{g}:\mathbb{R}^3\to \mathbb{R}$ is the unique function such that \\ \hspace*{0.3cm} $\widehat{g}(\log \lambda_1,\log \lambda_2,\log \lambda_3):=g(\lambda_1,\lambda_2,\lambda_3)$.
	\item $\tau_i =J\, \sigma_i=\dd\lambda_i\frac{\partial g(\lambda_1,\lambda_2,\lambda_3)}{\partial \lambda_i}=\frac{\partial \widehat{g}(\log \lambda_1,\log \lambda_2,\log \lambda_3)}{\partial \log \lambda_i}$ \, . 
	\end{itemize}

\vspace*{2em}
\noindent \textbf{Conventions for fourth-order symmetric operators, minor and major symmetry} \\
\\
For a fourth order linear mapping $\C : \Sym(3) \to \Sym(3)$ we agree on the following convention. \\
\\
We say that $\C$ has \emph{minor symmetry} if
	\begin{align}
	\C.S \in \Sym(3) \qquad \forall \, S \in \Sym(3).
	\end{align}
This can also be written in index notation as $C_{ijkm} = C_{jikm} = C_{ijmk}$. If we consider a more general fourth order tensor $\C : \R^{3 \times 3} \to \R^{3 \times 3}$ then $\C$ can be transformed having minor symmetry by considering the mapping $X \mapsto \sym(\C. \sym X)$ such that $\C: \R^{3 \times 3} \to \R^{3 \times 3}$ is minor symmetric, if and only if $\C.X = \sym(\C.\sym X)$. \\
\\
We say that $\C$ has \emph{major symmetry} (or is \emph{self-adjoint}, respectively) if
	\begin{align}
	\langle \C. S_1, S_2 \rangle = \langle \C. S_2, S_1 \rangle \qquad \forall \, S_1, S_2 \in \Sym(3).
	\end{align}
Major symmetry in index notation is understood as $C_{ijkm} = C_{kmij}$. \\
\\
The set of positive-definite, major symmetric fourth order tensors mapping $\R^{3 \times 3} \to \R^{3 \times 3}$ is denoted as $\Sym^{++}_4(9)$, in case of additional minor symmetry, i.e.~mapping $\Sym(3) \to \Sym(3)$ as $\Sym^{++}_4(6)$. By identifying $\Sym(3) \cong \R^6$, we can view $\C$ as a linear mapping in matrix form $\widetilde \C: \R^6 \to \R^6$. \newline If $H \in \Sym(3) \cong \R^6$ has the entries $H_{ij}$, we can write
	\begin{align}
	\label{eqvec1}
	h = \text{vec}(H) = (H_{11}, H_{22}, H_{33}, H_{12}, H_{23}, H_{31}) \in \R^6 \qquad \text{so that} \qquad \langle \C.H, H \rangle_{\Sym(3)} = \langle \widetilde \C.h, h \rangle_{\R^6}.
	\end{align}
If $\C: \Sym(3) \to \Sym(3)$, we can define $\bfsym \C$ by
	\begin{align}
	\langle \C.H, H \rangle_{\Sym(3)} = \langle \widetilde \C.h, h \rangle_{\R^6} = \langle \sym \widetilde \C. h, h \rangle_{\R^6} =: \langle \bfsym \C.H, H \rangle_{\Sym(3)}, \qquad \forall \, H \in \Sym(3).
	\end{align}
Major symmetry in these terms can be expressed as $\widetilde \C \in \Sym(6)$. \emph{In this text, however, we omit the tilde-operation and ${\bf sym}$ and write in short $\sym\C\in {\rm Sym}_4(6)$ if no confusion can arise.} In the same manner we speak about $\det \C$ meaning $\det \widetilde \C$. \\
\\
A linear mapping $\C : \R^{3 \times 3} \to \R^{3 \times 3}$ is positive definite if and only if
	\begin{align}
	\label{eqposdef1}
	\langle \C.H, H \rangle > 0 \qquad \forall \, H \in \R^{3 \times 3} \qquad \iff \qquad \C \in \Sym^{++}_4(9)
	\end{align}
and analogously it is positive semi-definite if and only if
	\begin{align}
	\label{eqpossemidef1}
	\langle \C.H, H \rangle \ge 0 \qquad \forall \, H \in \R^{3 \times 3} \qquad \iff \qquad \C \in \Sym^+_4(9).
	\end{align}
For $\C: \Sym(3) \to \Sym(3)$, after identifying $\Sym(3) \cong \R^6$, we can reformulate \eqref{eqposdef1} as $\widetilde \C \in \Sym^{++}(6)$ and \eqref{eqpossemidef1} as $\widetilde \C \in \Sym^+(6)$. 

\end{appendix}

\end{document}